\documentclass[aap]{imsart}

\usepackage{amsthm,amsmath,amsfonts,amssymb,mathrsfs}
\usepackage[numbers]{natbib}
\usepackage[colorlinks,citecolor=blue,urlcolor=blue]{hyperref}
\usepackage{graphicx}
\usepackage[export]{adjustbox} 
\usepackage{subcaption}
\usepackage{enumitem}

\startlocaldefs
\usepackage[figurename=Figure]{caption}
\theoremstyle{plain}

\newtheorem{algorithm}{Algorithm}
\newtheorem{theorem}{Theorem}[section]
\newtheorem{proposition}{Proposition}[section]
\newtheorem{lemma}{Lemma}[section]

\theoremstyle{remark}
\newtheorem{remark}{Remark}
\newtheorem{corollary}{Corollary}

\newtheorem*{example}{Example}

\def\B{\mathcal{B}}
\def\M{\mathcal{M}}
\def\N{\mathbb{N}}

\def\R{\mathbb{R}}

\def\P{\mathbb{P}}
\def\E{\mathbb{E}}
\def\MM{\mathscr{M}}
\def\FF{\mathscr{F}}

\def\LL{\mathcal{L}}

\renewcommand{\Phi}{\varPhi}
\renewcommand{\epsilon}{\varepsilon}

\newcommand{\var}{\operatorname{var}}

\renewcommand{\limsup}{\varlimsup}
\renewcommand{\liminf}{\varliminf}
\renewcommand{\d}{\text{\rm\,d}}
\newcommand{\eqdist}{\stackrel{\text{(d)}}{=}}

\def\MM{\msc{M}}
\def\CC{\mcl{C}}

\def\TT{\msc{T}}
\def\tt{\mathbf{t}}
\def\ttt{\mathrm{t}}
\def\XX{\msc{X}}
\def\YY{\msc{Y}}

\newcommand{\Ind}[1]{\mathbf{1}_{\left\{#1\right\}}}

\newcommand{\x}[1]{[x^{#1}]}
\renewcommand{\d}{{\mathrm{d}}}
\renewcommand{\L}{\mcl{L}}

\newcommand{\mcl}{\mathcal}

\newcommand{\msc}{\mathscr}
\newcommand{\Ll}{\left}
\newcommand{\Rr}{\right}
\newcommand{\bracket}[1]{\left\langle{#1}\right\rangle}

\endlocaldefs

\begin{document}

\begin{frontmatter}
\title{A growth-fragmentation-isolation process on random recursive trees and contact tracing}
\runtitle{A branching process on random recursive trees}

\begin{aug}
\author[A]{\fnms{Vincent} \snm{Bansaye}\ead[label=e1]{vincent.bansaye@polytechnique.edu}},
\author[B]{\fnms{Chenlin} \snm{Gu}\ead[label=e2]{guchenlin@hotmail.com}}
\and
\author[C]{\fnms{Linglong} \snm{Yuan}\ead[label=e3]{yuanlinglongcn@gmail.com}}
\address[A]{CMAP, Ecole polytechnique, IPP, \printead{e1}}

\address[B]{Mathematics Department, NYU Shanghai \& NYU-ECNU Institute of
Mathematical Sciences,}
\address[B]{Yau Mathematical Sciences Center, Tsinghua University, \printead{e2}}

\address[C]{Department of Mathematical Sciences,
University of Liverpool, \printead{e3}} 
\end{aug}

\begin{abstract}
We consider a random process on  recursive trees, with three types of events. Vertices give birth at a constant rate (growth), each edge may be removed independently (fragmentation of the tree) and clusters (or trees) are  frozen with a rate proportional to their sizes (isolation of connected component). A phase transition occurs when the isolation  is able to stop the growth fragmentation  process and cause extinction. When the  process survives, the number of clusters increases exponentially and we prove that the normalized empirical measure of clusters a.s. converges to a limit law on  recursive trees. We  exploit the branching structure associated to the size of clusters, which is inherited from the splitting property of random recursive trees. This work is motivated by the control of epidemics and contact tracing where clusters correspond to trees of infected individuals that can be identified  and  isolated. We complement  this work  by providing  results on the Malthusian exponent to describe the effect of control policies on epidemics. 
\end{abstract}



\end{frontmatter}

\section{Introduction}\label{sec:Introduction}
The evolution of  random trees is motivated by various fields: algorithms, queuing systems, population modeling, etc.
The random deletion of edges of a tree has been studied in particular 
by \cite{meir1974cutting, baur2014cutting}. Initially, Meir and Moon \cite{meir1974cutting}
 were interested in the number of steps needed to isolate a distinguished vertex in a random recursive tree, when every deleted edge is chosen uniformly.
 Bertoin \cite{JB2} and Marzouk \cite{Marzouk} have then studied processes where sets  of  vertices  can be  burnt. More precisely,  a connected component of the graph is removed (i.e.\ isolated) at each step. This component is determined
 by a uniform  choice among the vertices. Such dynamics combine the fragmentation of the tree (when an edge is deleted) and the isolation of connected components of the tree (when a vertex provokes a fire).
 
 In this work, we are interested in the long-time behavior of similar dynamics when the random recursive tree  grows, following a binary branching process.
 Our original motivation is to study the effects of 
 control policies of an epidemic.
 The growth of the infected population is modeled by a Yule process in this work, i.e. a binary Markov branching process. 
 The discrete structure of the Yule process  is a random recursive tree. The connected new vertices are the new infected individuals. 
 The fragmentation occurs when one edge is removed, interpreted as the loss of infector-infectee information: tracing of this contact becomes impossible. Various reasons may explain the loss or absence of information on contacts, including memory and storage of information or the fact that the contact is not accessible. We restrict ourselves here to a simple model with a single parameter accounting for a  rate at which contacts (edges) get lost (removed) independently. 
Growth and fragmentation will generate connected components, called \textit{clusters}. Each  cluster is a set of connected infected individuals that can be isolated all together as soon as one individual in the cluster is detected. An isolated cluster is frozen, in the sense that no more event will
  happen to it. It is interpreted as that 
 all the individuals in the cluster are either under treatment or self-isolating.  This is the contact tracing strategy which has been for decades the central public health response to control infectious disease outbreaks. 
 
Numerous  research papers  highlight the importance of contact-tracing for controlling epidemics, using either simulations or real data, and we mention  \cite{keeling2020efficacy, fetzer2021measuring, du2022contact} for Covid-19. 
For  mathematical modeling, we refer to  \cite{Lambert, AGGM, Barlow, bertoin2022model} for related studies which are motivated by contact tracing and exploit a branching structure. Let us describe the main differences.  In  \cite{Lambert}, the author provides   explosion criterion (threshold for $R_0$, the reproduction number), for different levels of tracing. The work \cite{AGGM} is more general regarding the characteristics of epidemics but does not allow for backward tracing  (i.e.\ tracing and isolation of ancestors when a descendant is identified and isolated).  In \cite{Barlow},   a discrete-time model is studied and  a Malthusian behavior is exhibited, with different methods and results. Here we exploit random recursive trees  and the growth fragmentation structure. Our model allows to study the effect of tracing along time, in a Markovian dynamical way. It  allows to keep a simple description of the probabilistic structure at any time, to achieve  fine asymptotic analysis and to make emerge tractable   key quantities for epidemics (by theory or simulations). In particular, we will exhibit a Malthusian coefficient which drives the main features of the epidemics, more precisely the speed of explosion or extinction (depending on its sign). We will see how it changes with the three parameters of the model, which allows to see the respective effect of social distancing measures (increase of infection rate),  detection effort (increase of $\theta$) and tracing effort (decrease of $\gamma$). The role of $\beta$ happens to be quite surprising for us and somewhat reinforce the interest of ``isolation-tracing'' strategy compared to social distancing. Finally, in  \cite{bertoin2022model}  a  model related to ours is considered, where the contact information can be lost  when infection occurs
(and not later). In this model, fragmentation disappears and  more explicit
eigenelements and criterion for explosion are given. Combining both effects (instantaneous or continuous loss of contact) would be interesting.

   \begin{figure}
     \centering
     \includegraphics[scale =0.5]{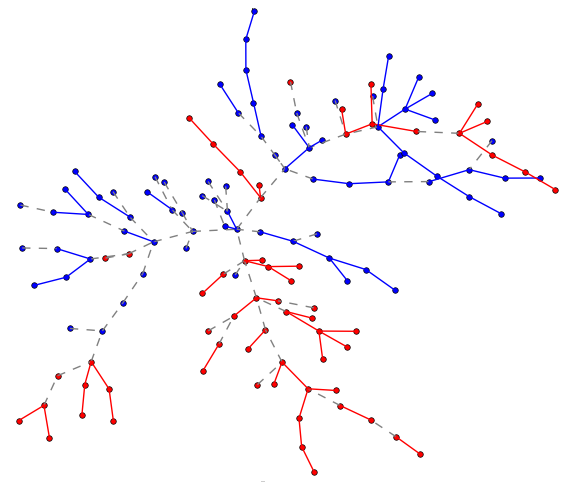}
     \caption{An illustration of the growth-fragmentation-isolation process with 62 active vertices (in red) and 77 inactive vertices (in blue). A cluster (or a connected component, a tree) is formed by solid lines, while the dashed lines split  clusters. }
     \label{fig:cover}
 \end{figure}
 
 Our growth-fragmentation-isolation model is thus a  branching process  to study  the effect of identification-tracing-isolation strategy in the context of simple epidemics, with loss of contact information along time. The second author also considered a similar model with numerical simulations in \cite{gu2020mathematical} to study the propagation of Covid-19.   The approximation of the outbreak by a branching process
 is classical in the early stages of an epidemic, when the whole population is large and the proportion of susceptible individuals is close to one, see for instance \cite{BallDonnelly}. We also obtain  in this work estimations on the speed of convergence of infected profiles, which allow to control in which time window  the Malthusian exponential growth indeed captures the dynamics of the epidemics.
 There are  more epidemiological and control features that could be incorporated to reflect epidemics like   Covid-19. In particular, no recovery happens in our setting, and we consider only large homogeneously mixed populations. 
 We expect  some extensions of our results on the long-time behavior of infected population which would  take into account these features,   even if the probabilistic structure of clusters becomes more complex. 
We give a more detailed discussion in Section~\ref{sec:generalise}. 
 
 Random recursive trees (RRT) have a nice splitting property that allows us to characterize a cluster by its size. 
More precisely, considering the collection  of active (non-isolated) clusters,  the size process turns out to be a branching process with  a countable set of types (i.e.\ type being the size). At fixed time, conditionally on the sizes of the clusters, the collection of clusters are independent RRTs. We  then study the ergodic properties of the first moment semigroup of this branching process and obtain  a phase transition depending on the sign of the maximal eigenvalue (Malthusian growth rate). We  describe the a.s.\ behavior of the process when the active clusters survive, proving a strong law of large numbers for the empirical measure of sizes 
and a Kesten--Stigum type result for the growth of the population. With the knowledge on the size process, we can deduce the a.s.\ behavior of the process of active clusters conditioned on survival.  We can also characterize the a.s.\ behavior of the process of isolated clusters when the active clusters survive, which is by itself non-Markovian. 

 The asymptotic analysis of the mean behavior, the weak convergence and the estimation of the speed of convergence uses now well  developed techniques for branching processes with possibly infinitely many types. We refer  to \cite{bertoin3, bertoinwatson, bansaye2019non, Marguet, TVB, watson} and references therein for  related works on  the asymptotic analysis of growth fragmentation processes. Roughly, it relies on the ergodic properties of the size of a typical cluster, and the fact that the common ancestor of two clusters uniformly chosen at large times is found at small times,  see for instance Theorem 2 in  \cite{Athreyarapid}, see also \cite{harris2020coalescent}.   More precisely, we exploit  the fact that large clusters fragment fast,  and with high probability give one small cluster and one large cluster at fragmentation. Together with isolation, it allows one to control the size of a typical cluster and  the eigenelements of the first moment semigroup. In particular, we prove that the harmonic function is bounded and large clusters have no major impact on the growth of epidemics. Indeed, large clusters are isolated before creating too many small clusters, since isolation occurs at the same scale as fragmentation.
Once active clusters are well described, we can treat the isolated clusters using an additive functional.

Besides, the fact that the number of types is infinite and the loss of Markov property for the isolated clusters raise some technical  difficulties  to get a.s.\ limits (strong convergence).  We refer to \cite{athreya1968some,AsmussenHering, EHK} for classical references  on strong law of large numbers of some classes of multi-type branching processes. We adapt here the argument of \cite{athreya1968some} for strong convergence.  In the context of growth fragmentation,  let us mention respectively  \cite{BW} and \cite{neutrons,watson}  for $L^1$ and strong  convergence.



\bigskip

Let us describe the model more formally. We introduce a  stochastic process on a dynamic tree $G_t = (V_t, E_t)$ with two functions $\Psi_t: V_t \to \{0,1\}, \, \eta_t: E_t \to \{0,1\}.$
\begin{itemize}
\item We identify the vertex set $V_t$ as the set of patients (individuals infected up to time $t$), and label them with the infection time $v \in [0,\infty)$. The function $\Psi_t$ tells us the state of a vertex at time $t$, where vertex $v$ is \textit{active} if $\Psi_t(v)=1$ and $v$ is \textit{inactive} if $\Psi_t(v) = 0$. Only active vertices can infect new ones.  
\item We identify the edge set $E_t$ as the set of (direct) infection links between patients, and the function $\eta_t$ tells us the state of an edge at time $t$. For an edge $e$, we say $e$ is \text{open} if $\eta_t(e)=1$ which means the infection link can still be retrieved (i.e. it can be found who infected the infectee), otherwise $\eta_t(e)=0$ and it is \textit{closed} (i.e.\ it is not possible to know who infected the infectee). A set of vertices connected by open edges is called a  \textit{cluster}. 

\end{itemize}
The growth-fragmentation-isolation model (GFI) process $(G_t, \Psi_t, \eta_t)_{t \geq 0}$ is a Markov jump process, starting from an active vertex as patient zero $V_0 = \{0\}, \Psi_0(0) = 1$,  governed by three positive parameters $(\beta, \theta, \gamma) \in \R_+^3 = (0,\infty)^3$ representing three types of events, with the notation $\P$ for the probability and $\E$ for the associated expectation.
\begin{itemize}
\item \emph{Infection} (growth): every active vertex $v$ independently attaches a new vertex in an exponential waiting time with parameter $\beta$. When a new vertex $u$ is created and attached to $v$, it is active (i.e.\ $\Psi_t(u)=1$) and the edge $\{u,v\}$ is open (i.e.\ $\eta_t(\{u,v\}) = 1$).  

\item \emph{Information decay} (fragmentation): every open edge $e$ independently becomes closed 
in an exponential waiting time with parameter $\gamma$.

\item \emph{Confirmation and contact tracing} (isolation): every active vertex independently gets ``confirmed''  in an exponential waiting time with parameter $\theta$, and once a vertex is confirmed, the associated cluster is isolated and every vertex in this cluster becomes inactive. 
\end{itemize}See Figure~\ref{fig:GFI} for an illustration of this model. If $\gamma = \theta = 0$, this is the well-known Yule tree process. As vertices are indexed by infection times, every cluster is a labeled  \textit{recursive tree} (see Section~\ref{subsec:RRT} for a rigorous definition). 



\begin{figure}
    \centering
    \begin{subfigure}[t]{0.45\textwidth}
        \centering
        \includegraphics[width=0.72\linewidth]{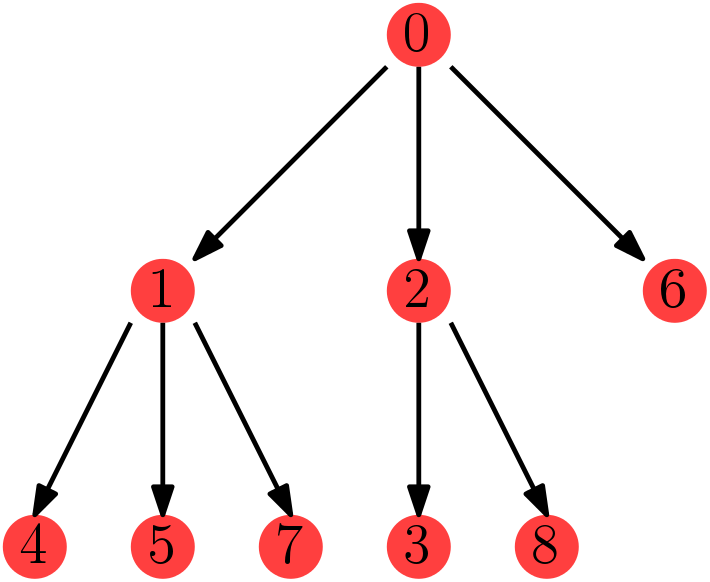} 
        \caption{Growth: starting from vertex $0$, the vertices are attached one by one, and form a recursive tree.} \label{fig:GFI1}
    \end{subfigure}
    \hfill
    \begin{subfigure}[t]{0.45\textwidth}
        \centering
        \includegraphics[width=0.72\linewidth]{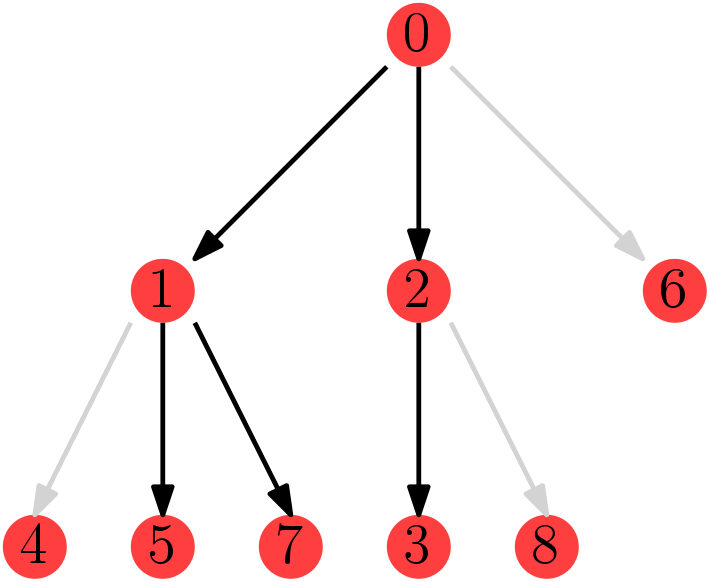} 
        \caption{Fragmentation: 
        links get lost 
        over time; for example the grey links $\{0,6\}, \{1,4\}, \{2,8\}$ in the tree are lost. } \label{fig:GFI2}
    \end{subfigure}

    \vspace{1cm}
    \begin{subfigure}[t]{0.45\textwidth}
    	\vskip 0pt
        \centering
        \includegraphics[width=0.72\linewidth]{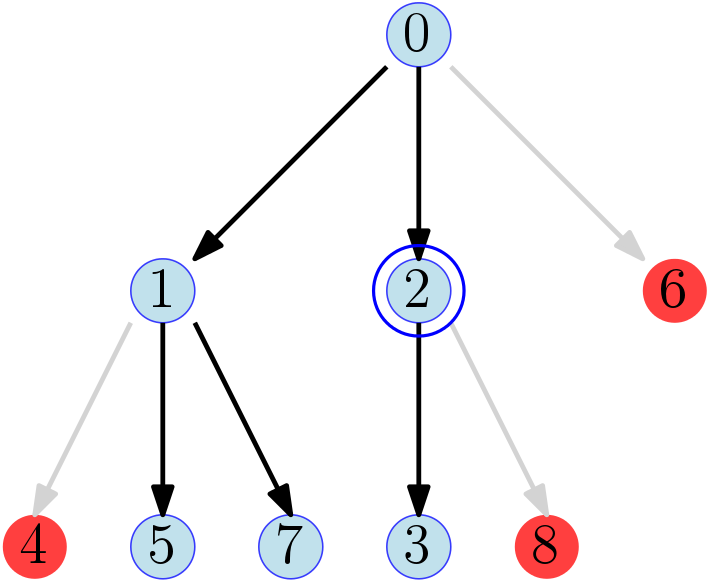} 
        \caption{Isolation: when the vertex $2$ is confirmed,  all the vertices  in the same cluster are isolated. These are the vertices in blue $\{0,1,2,3,5,7\}$.} \label{fig:GFI3}
    \end{subfigure}
    \hfill
    \begin{subfigure}[t]{0.45\textwidth}
    	\vskip 0pt
        \centering
        \includegraphics[width=1\linewidth]{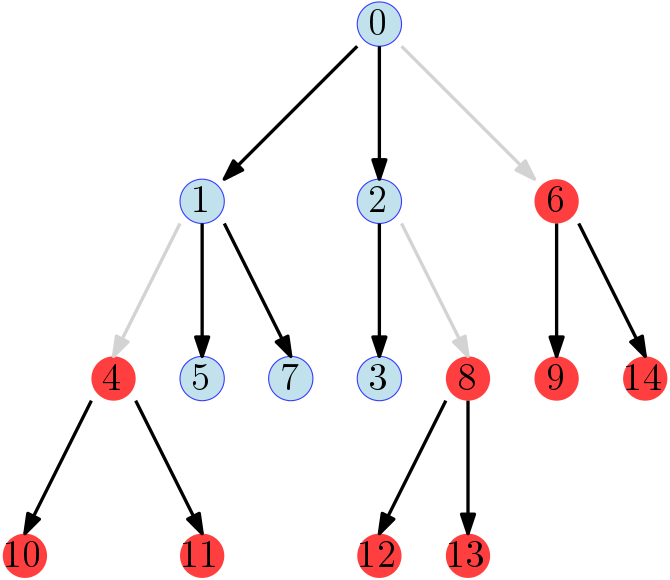} 
        \caption{The isolated vertices are no longer active, while the active vertices continue to attach new vertices.} \label{fig:GFI4}
    \end{subfigure}
    \caption{An illustration of GFI process.}\label{fig:GFI}
\end{figure}
 Because the vertices in a cluster have the same state, it is very natural to decompose the dynamical tree $G_t$ into clusters of individuals connected by open edges: for an isolated cluster, we call it \textit{inactive cluster}; otherwise, it is an \textit{active cluster}. In this paper, we use {\it isolated} and {\it inactive} interchangeably for clusters and also for vertices/patients/infected individuals. We denote by $(\XX_t, \YY_t)_{t \geq 0}$ the associated \textit{cluster process}, where $\XX_t$ is the set of active clusters and $\YY_t$ is the set of inactive clusters:
\begin{equation*}
\begin{split}
\XX_t &= \{\CC \,\vert\,\, \CC \text{ is a cluster in } G_t;\, \forall v \in \CC, \Psi_t(v) = 1\},\\
\YY_t &= \{\CC \,\vert\,\, \CC \text{ is a cluster in } G_t;\, \forall v \in \CC, \Psi_t(v) = 0\}.
\end{split}
\end{equation*}
The process $(G_t, \Psi_t, \eta_t)_{t \geq 0}$ stops evolving when $\XX_t$ is empty. We denote by $\tau$ the corresponding stopping time: 
\begin{align}
\tau := \inf\{t\geq 0 \,\vert\, \XX_t = \emptyset\},
\end{align}
and the event $\{\tau < \infty\}$ is called \textit{extinction}, while $\{\tau = \infty\}$ is called \textit{survival}. \\

In this paper, we are motivated by the following questions.
For what values of the parameters $(\beta,\theta, \gamma)$  does the epidemic reach extinction almost surely ? If the epidemic survives (with positive probability), what is the long-time behavior
of the population of active and inactive clusters ? We give some answers to these questions by first determining the asymptotic behavior of the first moment semigroup associated to the active clusters. It depends
on  the maximal eigenvalue of this semigroup, which is called Malthusian exponent.
When this exponent is negative (subcritical case) or zero (critical case), the population of active clusters reaches extinction almost surely. It corresponds to the fact that the  isolation process is strong enough to stop the epidemic. When this exponent is positive (supercritical case), survival occurs with positive probability and on this event, the growth of the population is exponential with rate given
by this Malthusian exponent.  
In that case, we also  shed some light on the genealogical structure of clusters and describe the asymptotic behavior of the empirical distribution. We prove that a.s.\
we get a collection of recursive trees whose sizes are distributed following the left eigenvector associated to the maximal eigenvalue of the semigroup, i.e.\ the Malthusian exponent. We also show similar asymptotic behaviors for the inactive clusters on the survival event.   

We give the main results and the outline of the paper in the next section. The rest of the paper is dedicated to proofs and some additional results and comments.

\section{Main results}

We first introduce the \textit{Malthusian exponent} $\lambda$ which describes the (mean) exponential growth (or decay) of the number of clusters.
This growth rate coincides for the active and inactive clusters, whose numbers at time $t$ are respectively denoted by $|\XX_t|$  and $|\YY_t|$. As expected in branching 
structures, its sign gives the global extinction and survival criterion, leading to the classification of subcritical, critical and supercritical phases. 

\begin{theorem}[Malthusian exponent]\label{thm:Malthusian}
The following limits exist and coincide and are finite 
\begin{align*}
 \lambda:=\lim_{t\rightarrow\infty} \frac{1}{t} \log(\E[| \XX_t|]) =\lim_{t\rightarrow\infty} \frac{1}{t} \log(\E[| \YY_t|])  \in (-\infty,\infty).
\end{align*}
 If $\lambda \leq 0$, then  extinction occurs a.s., i.e.\ $\P[\tau<\infty]=1$.
Otherwise, survival occurs with positive probability, i.e.  $\P[\tau=\infty]>0$. 
\end{theorem}



The Malthusian exponent $\lambda$ corresponds to the maximal eigenvalue of the first moment semigroup of $(X_t)_{t\geq 0}$ and is also called \textit{Perron's root}. The fact that the cluster size can be any positive integer leads us to use techniques for ergodic behavior in infinite dimension, where the control of large sizes is crucial. 
As usual,  irreducibility on the state of sizes  ensures that the value $\lambda$ does not depend on the initial condition (although the model starts with a single patient, i.e.\ a cluster of size $1$, the setting of initial condition with a single random recursive tree of size $n\geq 1$ will be used later, especially in the size processes). The diagram of these different phases is illustrated in Figure~\ref{fig:Phases}, for fixed $\beta>0$. 
Note that $\theta \geq \min(\beta,\gamma)$ implies a.s. extinction
since in the case $\theta \geq \beta$, individuals are isolated faster than they contaminate and in the case  $\theta \geq \gamma$, isolation is faster than fragmentation. 

We want to emphasize that Theorem~\ref{thm:Malthusian} and the forthcoming results do not depend on the initial condition (random or deterministic), see the discussion in Section~\ref{sec: initial}.  
We also refer to Section~\ref{sec:generalise} for other possible generalizations with
additional epidemic features. 


\begin{figure}
\centering
\includegraphics[scale=0.45]{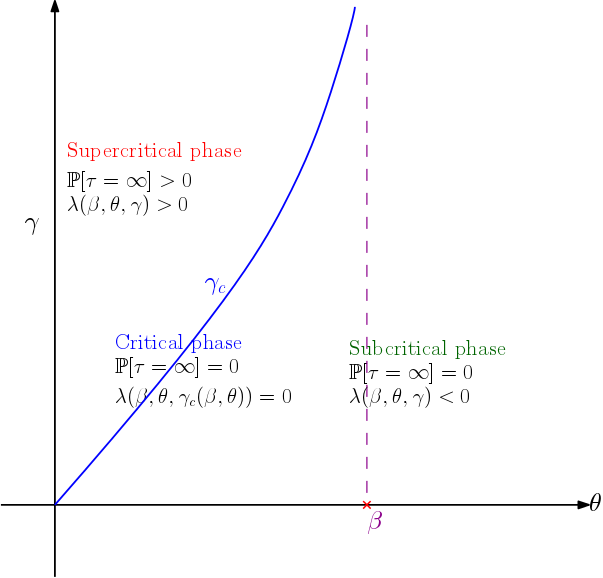}
\caption{An illustration of different phases.}\label{fig:Phases}
\end{figure}

\bigskip

To go further in the analysis of the model, we introduce the following \textit{size process} $(X_t, Y_t)_{t \geq 0}$, where the two empirical measures count the clusters of different sizes
\begin{equation}\label{eq:defEmpirical}
X_t = \sum_{\CC \in \XX_t} \delta_{\vert \CC \vert}, \qquad Y_t = \sum_{\CC \in \YY_t} \delta_{\vert \CC \vert}.
\end{equation} 
Here $\vert \CC \vert$ is the number of vertices in the cluster $\CC$, and we call it \textit{the size of cluster}. 
The process $(X_t)_{t\geq 0}$ is still
 a branching Markov process with respect to its natural filtration. This comes from the splitting 
 property of RRT, which allows to preserve RRT clusters when fragmentation occurs, 
 see the forthcoming Proposition \ref{prop:Split}.
 Moreover, for any fixed time, conditionally on cluster sizes,   all (active and inactive) clusters are independent, see Proposition~\ref{prop:ClusterRRT}. 
 We  can thus reduce the study of our GFI process $(G_t, \Psi_t, \eta_t)_{t \geq 0}$ to that of the size process $(X_t, Y_t)_{t \geq 0}$; see Figure~\ref{fig:Decom}. 

We prove the following strong law of large numbers in the supercritical case. This provides 
the asymptotic behavior of
$ \langle X_t, f \rangle=\sum_{\CC \in \XX_t} f(\vert \CC \vert)$, where $f$ has at most polynomial growth, i.e.\ there exists $p >0$ such that ${\sup_{n\geq 1} \vert f (n)\vert/ n^p < \infty}$. In particular, 
$f={\bf 1}_m$ yields the number of active clusters of size $m$, while the identity function provides the number of active individuals.


\begin{theorem}[Law of large numbers for $(X_t)_{t \geq 0}$]\label{thm:LLN} Assume that $\lambda>0$.
 Then there exists a probability distribution $\pi$ on $\N_+=\{1,2,3,\cdots\}$ and a random variable $W \geq 0$, such that for any function $f: \N_+ \to \R$ of at most polynomial growth, we have $\bracket{\pi, |f|}<\infty$ and 
\begin{align}\label{eq:LLN1}
e^{-\lambda t} \langle X_t, f \rangle \xrightarrow{t \to \infty}  W \langle \pi, f\rangle, \qquad \text{ a.s.  and  in } L^2. 
\end{align}
Besides, $\{\tau = \infty\} = \{W > 0\}$ a.s.  and on this event
\begin{align}\label{eq:LLN2}
\frac{\langle X_t, f \rangle}{\langle X_t, \mathbf{1} \rangle} \xrightarrow{t \to \infty} \langle \pi, f\rangle \quad \text{a.s.} .
\end{align}  
\end{theorem}  


Recall that $\lambda$ is defined in Theorem~\ref{thm:Malthusian}. It  gives the a.s. exponential growth of the number of active clusters and active individuals. The probability distribution $\pi$ gives the distribution of the size of  clusters for large times. This distribution is the (positive normalized) left eigenvector of the first moment semigroup   $M=(M_t)_{t\geq 0}$  associated to the size process $(X_t)_{t\geq 0}$. Equivalently, it can be characterized as the positive normalized left eigenvector of the generator $\L$
of $M$, i.e.  $\pi \L=\lambda \pi$,  where 
$$\L f(n) = \beta n (f(n+1) - f(n)) - \theta n f(n)  + \sum_{j=1}^{n-1}  \frac{\gamma n}{j(j+1)} \Ll(f(j) + f(n-j) - f(n)\Rr),$$
for real valued functions $f$ and $n\geq 0$.
Thus $(\pi(n))_{n\geq 0}$ satisfies a linear system given by the dual operator of $\L$. We refer to Section~\ref{semigene} for  rigorous 
statements and details on the  semigroup and generator.
The fact that $\bracket{\pi, |f|}<\infty$ will also be given in Proposition~\ref{Perroneigen}.

The random variable $W$ in the statement is the limit of the Malthusian martingale $e^{-\lambda t}\bracket{X_t, h}$, where $h$ is the right eigenvector of the semigroup: $M_th=e^{\lambda t}h$, for any $t\geq 0$; again see Proposition~\ref{Perroneigen}. Because the cluster size can be arbitrarily large, the classical finite-dimensional Perron--Frobenius theorem does not apply and we need a more precise analysis on the semigroup and its generator to ensure the properties of eigenvectors; see Section~\ref{sec:Root}.  

Note that the above results hold for functions of at most polynomial growth. This is inherited from our Lyapunov functions which are polynomial; see Lemma~\ref{lem:HarrisV} and Proposition~\ref{Perroneigen}. 

\smallskip

A similar result about the inactive clusters can be derived.

\begin{corollary}[Law of large numbers for $(Y_t)_{t \geq 0}$]\label{cor:LLNY}
For any function $f: \N_+ \to \R$ of at most polynomial growth, we have that 
\begin{equation*}
e^{-\lambda t} \langle Y_t, f \rangle \xrightarrow{t \to \infty}  \Ll(\frac{\theta}{\lambda}\Rr) \, W\, \sum_{n=1}^{\infty} n\pi(n)f(n) \quad
\text{ a.s. and  in } L^2,
\end{equation*}
and 
$$\frac{\langle Y_t, f \rangle}{\langle Y_t, \mathbf{1} \rangle} \xrightarrow{t \to \infty} \langle \widetilde{\pi}, f\rangle, \qquad \,\,\,\,\,\,\quad \text{ a.s. on } \{\tau = \infty\},$$
where  $\widetilde{\pi}$ is a probability law on $\N_+=\{1,2,3,\cdots\}$ 
given by
\begin{equation}\label{eq:PiTilde}
\widetilde{\pi}(n) := \frac{\pi(n) n}{\sum_{j=1}^{\infty} \pi(j) j}. 
\end{equation}
\end{corollary}
We observe  that interestingly, $\widetilde{\pi}$ is a size-biased version of $\pi$. It means that the isolated clusters are of larger size than the active clusters. This phenomenon is also recorded in the recent work of Jean Bertoin \cite[Corollary 4.3]{bertoin2022model}. The reason for this phenomenon  is that every active cluster gets isolated at a rate proportional to its size. 
 

\begin{figure}
\centering
\includegraphics[scale=0.4]{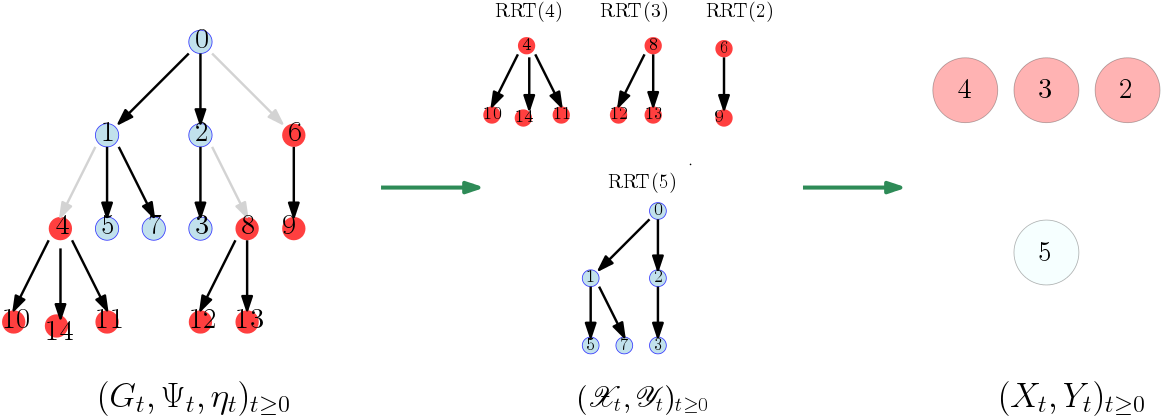}
\caption{Illustration of the reduction
of 
the study of GFI process $(G_t, \Psi_t, \eta)_{t \geq 0}$. The first simplification is to decompose the graph into clusters and study the cluster process $(\XX_t, \YY_t)_{t \geq 0}$, where every cluster is a RRT whose law only depends on its size. The second simplification is to study size process $(X_t, Y_t)_{t \geq 0}$, which provides the key   information for the law of $(\XX_t, \YY_t)_{t \geq 0}$.}\label{fig:Decom}
\end{figure}

\bigskip

Recall that every vertex $v \in \R$ is labeled by its infection time and every cluster is a  tree. We consider these clusters up to an equivalence relation, which consists in keeping the order between vertices but forgetting their infection times,  see  Section~\ref{subsec:RRT} for a rigorous definition.  
We denote by  $\TT$ the space of equivalence classes of trees of all sizes. We denote by $T_{\pi}$  the random tree whose  size is distributed as $\pi$, and conditionally on its size  the tree is a RRT. With a slight abuse of notation, for any recursive tree $\ttt$, we write   $\tt\in \TT$ for its equivalence class and
for  any function $f:\TT\mapsto \R,$ we let $f(\ttt)=f(\tt)$. We also write $\vert \tt \vert$  for the number of vertices in the equivalence class $\tt.$ 


Using the splitting property, we derive the long-time behavior of the empirical measures on (active and inactive) clusters
from the results on the size processes obtained above.
\begin{theorem}[Limit of the empirical measure of clusters]\label{thm:LLNRRT}  Consider any $f : \TT \to \R$ such that there exists  $p > 0$ satisfying $$\sup_{\tt \in \TT} \frac{\vert f(\tt) \vert}{\vert \tt \vert^p}  < \infty.$$
 Then  on the event $\{\tau = \infty\}$,
\begin{align*}
 \frac{1}{\vert \XX_t \vert}\sum_{\CC \in \XX_t} f(\CC) \stackrel{t \to \infty}{\longrightarrow} \E[f(T_\pi)],  \quad  \frac{1}{\vert \YY_t \vert}\sum_{\CC \in \YY_t} f(\CC)  \stackrel{t \to \infty}{\longrightarrow}  \E[f(T_{\widetilde{\pi}})] \quad \text{a.s.}.
\end{align*}
\end{theorem}



\bigskip

The Malthusian exponent $\lambda$ is the key characteristic of the model and is fully determined by the model parameters $(\beta,\gamma,\theta)$. For both mathematical and practical purposes, it is highly important to understand how $\lambda$ depends on $(\beta,\theta,\gamma)$ The main results are presented below. \begin{theorem}\label{thm:Regularity}
The mapping $(\beta, \theta, \gamma) \mapsto \lambda(\beta, \theta, \gamma)$ is continuous, and the sets of parameters resulting in respectively $\lambda > 0, \lambda = 0, \lambda < 0$ are non-empty. The Malthusian exponent depends on the three parameters monotonically that
\begin{enumerate}
    \item $\theta \mapsto \lambda(\beta, \theta, \gamma)$ is decreasing;
    \item $\gamma \mapsto \lambda(\beta, \theta, \gamma)$ is increasing;
    \item $\beta \mapsto \lambda(\beta, \theta, \gamma)$ is increasing if $\gamma > \theta$, constant if $\gamma = \theta$, and decreasing if $\gamma < \theta$.
\end{enumerate}
\end{theorem}

In the above statements, the continuity, the existence of phases, and also the monotonicity with respect to $\theta$ and $\gamma$, are not surprising. However, the change of monotonicity of  $\beta \mapsto \lambda(\beta, \theta, \gamma)$ seems not  obvious. Indeed,   the increment of $\beta$ makes the clusters grow faster, thus the clusters will split (fragmentation) or be detected (isolation) faster. Faster fragmentation makes the number of clusters increase while faster detection has the opposite effect. The complex competition  between the two forces leads to the simple Statement 3 in Theorem~{thm:Regularity} surprisingly. 
At the end,  the increment of the infection rate $\beta$ will not necessarily speed up the epidemics, but does make the  epidemics weaker when $\gamma < \theta$. In practical terms,  in this regime,
if the contact tracing policy  is effective so that the population is in the  subcritical regime, the emergence of a more contagious variant will make the epidemic decay faster, even if $\theta$ and $\gamma$ are unchanged (it means we use the same contact tracing practice as before). 

\bigskip

The rest of the paper is organized as follows. In Section~\ref{sec:Pre} we introduce  notations  and we  recall the key splitting property of the random recursive trees and derive the size process. The study of the first moment semigroup of $(X_t)_{t\geq 0}$ and the associated martingale and $L^2$ estimates
is carried out  in Section~\ref{sec:Root}, using
in particular Lyapunov functions. We can then prove  Theorem~\ref{thm:Malthusian}. Section~\ref{sec:Limit} is devoted to the strong convergences  and we  prove  Theorem~\ref{thm:LLN}, Corollary~\ref{cor:LLNY} and  Theorem~\ref{thm:LLNRRT}. Then we prove Theorem~\ref{thm:Regularity} in Section \ref{sec:Chara} and present related simulations in Section \ref{sec:examsimu}. Finally we discuss possible extensions and generalizations of the model in Section~\ref{Sec:Comp}.

\section{Preliminaries}\label{sec:Pre}
In this section, we  introduce  notations and explain why the size process captures the essential information of the cluster process. 

\begin{figure}
\centering
\includegraphics[scale=0.4]{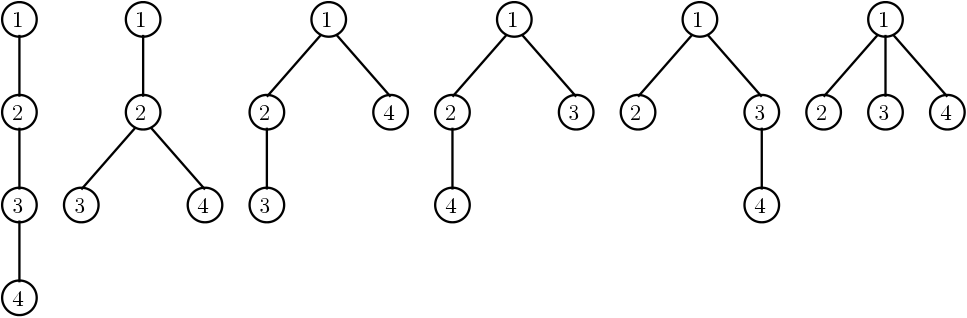}
\caption{All the recursive trees (as representatives of equivalence classes) in $\TT_4$.} \label{fig:RRT4}
\end{figure}

\subsection{Ulam-Harris-Neveu labeling of clusters}\label{UHN}
We introduce  the notation for the genealogical tree of clusters (active and inactive). Recall that every vertex is labeled by its infection time.  For any cluster $\CC$ (active or inactive),  we call the vertex with the minimum label \textit{the root} of $\CC$, and denote it by $\mathrm{root}(\CC)$. We  label every cluster using the Ulam-Harris-Neveu (UHN) notation:
\begin{align*}
    \mathcal U :=\bigcup_{n\geq 0} \{1,2\}^n,
\end{align*} 
where an  element $u\in \mathcal U$ is called a \textit{label} or \textit{word}. For the initial cluster, we use the label $\emptyset$ as convention. Then by induction, for any cluster $\CC$ labeled by a word $u \in \mcl{U}$: this label is unchanged during the growth of the cluster (infection); this
label becomes inactive  when the cluster is isolated;
this label is replaced by two labels $u1$ and $u2$ when fragmentation occurs. By convention, $u1$ is the label of the subcluster containing $\mathrm{root}(\CC)$ and is called \textit{the first child}, while $u2$ is for the other subcluster that we call \textit{the second child}. 

Every cluster except the initial one has a unique parent. There exists a partial order $\preccurlyeq$ on $\mcl U$ defined by the genealogy, i.e. for two words $u$ and $uv$ with $v \neq \emptyset$, the former is an \textit{ancestor} of the latter, while the latter is a \textit{descendant} of the former, and we denote $u \preccurlyeq uv$. For two words $u, v \in \mcl{U}$, we denote by $u \wedge v$ the most recent common ancestor of $u$ and $v$.

\begin{figure}
    \centering
    \includegraphics[scale = 0.5]{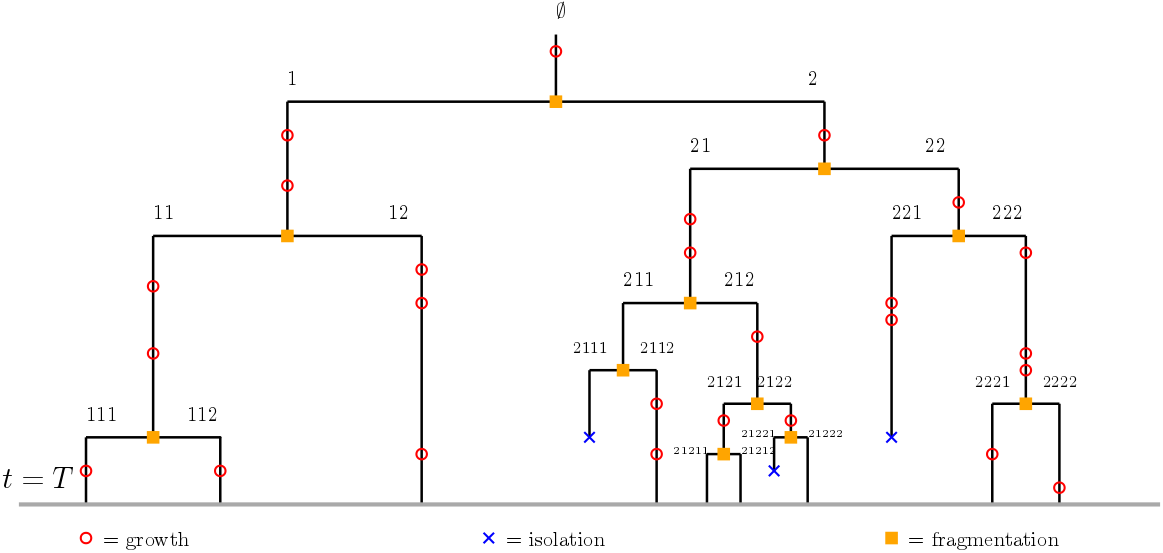}
    \caption{An illustration of the genealogy of clusters.}
    \label{fig:firstChild}
\end{figure}

We denote by $\mcl{U}_t$ the collection of labels of active clusters at time $t$, while $\mcl{U}^{\dagger}_t$ gathers the labels of inactive clusters at this time. For $u \in \mcl U$, we use  $\mcl{U}(u)$ to represent the genealogical tree rooted at $u$
\begin{align*}
    \mcl{U}(u) := \Ll\{w \in \mcl U \, \vert w \succcurlyeq u \text{ or }  w = u \Rr\}.
\end{align*}
We also use $\mcl{U}_t(u)$ (resp. $\mcl{U}^{\dagger}_t(u)$) to denote the set of labels of  clusters in $\mcl{U}(u)$ which are active (resp. inactive) at time $t$. One should notice that $\mcl{U}_t(u)$ can be either the set of labels of active descendants of $u$ at time $t$, or $u$ itself if it is still alive at that moment, or empty if it encounters isolation first at a time no later than $t$,  so we consider $\mcl{U}_t(u)$ as the ``set of labels of active clusters at $t$ issued from the cluster of the label $u$''. Finally, if $u \in \mcl{U}_t$ (resp.\ $u \in \mcl{U}^{\dagger}_t$), we denote by $\XX_t^u$ (resp.\ $\YY_t^u$) for its associated cluster, and $X_t^u$ (resp.\ $Y_t^u$) its size, i.e. $X_t^u = \vert \XX_t^u \vert$ (resp.\ $Y_t^u = \vert \YY_t^u \vert$); we write $X_t^u = 0$ for $u \notin \mcl{U}_t$. We also use $X_t(n)$ (resp.\ $Y_t(n)$) for the number of active clusters (resp.\ inactive clusters) of size $n$ at time $t$. With these notations, we have $X_t= \sum_{u \in \mcl{U}_t} \delta_{X_t^u}$ and 
for any function $f$ from $\N_+$ to $\R$:
\begin{align*}
\bracket{X_t, f} = \sum_{\CC \in \XX_t}f(\vert \CC \vert) = \sum_{u \in \mcl{U}_t}f(X_t^u) = \sum_{n = 1}^\infty  X_t(n) f(n).
\end{align*}

\subsection{Random recursive trees}\label{subsec:RRT}



The genealogy in a cluster is given by a recursive tree and here we define some notations related to the latter. Given ${V=\{a_1, \cdots, a_n\} \subset \R}$ with increasing order $a_1 < a_2 < \cdots < a_n$, a \textit{recursive tree} $\ttt$ on $V$ is a rooted tree with $V$ as the set of vertices, such that for any $a_i, 2 \leq i \leq n$, the path from $a_1$ to $a_i$ is increasing. Thus, a descendant has a larger label than that of the parent. The minimal element $a_1$ is called the \textit{root} of $\ttt$. The collection of all the recursive trees on $V$ has cardinality $(\vert V \vert-1)!$.

We define the equivalence relation $\sim$ between recursive trees on different ordering sets.  Denoting by $\ttt_1$ a recursive tree on $V_1$ and $\ttt_2$ a recursive tree on $V_2$, then $\ttt_1 \sim \ttt_2$ if and only if there exists an order-preserving function $\psi : V_1 \to V_2$, which 
induces a bijection between the trees $\ttt_1$ and $\ttt_2$. We denote by $\TT_n$ the set of recursive trees of size $n$ up to the equivalence relation $\sim$, and use the recursive trees defined on $\{1, \cdots, n\}$ to represent 
equivalence classes (canonical representation); see Figure~\ref{fig:RRT4} for an example of $\TT_4$. Finally, we define the space of finite recursive trees up to equivalence relation $\sim$
\begin{align}\label{eq:defRRTWhole}
\TT := \bigcup_{n=1}^{\infty} \TT_n.
\end{align}

A \textit{(uniform) random recursive tree (RRT)} of size $n$ is a  random element chosen uniformly in $\TT_n$. We denote by $T_n$ this random equivalence class. With a slight abuse,  RRT can refer both to  the equivalence class
or a specific labeling (for instance with the first integers).
Since $\TT$ defined in \eqref{eq:defRRTWhole} contains only countably many elements, the space $\TT$ is a Polish space under the trivial distance. We can construct more general probability measures than uniform random recursive tree on $\TT$. For example, for any $\nu$ a probability measure on $\N_+$, we use the notation $T_\nu$ to represent a random variable on $\TT$, such that we sample first the size by $\nu$, then sample an equivalence class uniformly given its size, i.e.
for any bounded function $f$ on $\TT$,
\begin{equation}\label{eq:defTnu}
\begin{split}
\E[f(T_\nu)] 
= \sum_{n=1}^{\infty} \nu(n) \E[f(T_n)] 
= \sum_{n=1}^{\infty} \nu(n) \Ll(\frac{1}{(n-1)!} \sum_{\tt \in \TT_n} f(\tt)\Rr).
\end{split}
\end{equation}
This gives the rigorous definition for $\E[f(T_\pi)]$ and $\E[f(T_{\widetilde \pi})]$ 
in Theorem~\ref{thm:LLNRRT}.\\

There are many ways to construct $T_n$. One classical construction is the recursive approach: let $T_1$ be the tree with the single vertex $1$, and construct $T_{k+1}$ by attaching the vertex labeled $(k+1)$ uniformly onto a vertex of $T_k$. This construction explains why  our infection process (Yule process),  conditioned on its size, is a RRT. Indeed
each individual contaminates a new individual
with the same rate, which amounts to attaching a new vertex 
to a uniformly chosen vertex of the tree, independent of the history of construction given the current state.\\

The key property of RRT that we need is the splitting property. We state this property below and its proof can be found in \cite{meir1974cutting} and \cite{baur2014cutting}. We will explain what role this property plays in our model in the next section.
\begin{proposition}[Splitting property]\label{prop:Split}
Let $n\geq 2$ and  $T_n$ the canonical random recursive tree of size $n$. We choose uniformly one edge in $T_n$ and remove it. Then $T_n$ is  split into two subtrees $T_n^0$ and $T_n^*$, corresponding to  two connected components, where $T_n^0$ contains the root of $T_n$
and  $T_n^*$ does not. Then
\begin{equation}\label{eq:Split}
\P\Ll[\vert T^*_n \vert = j\Rr] = \frac{n}{n - 1} \frac{1}{j(j+1)}, \qquad j = 1,2,\cdots, n-1.
\end{equation}
Furthermore, conditionally on $\vert T^*_n \vert = j$, $T_n^0$ and $T_n^*$ are two independent RRTs of size respectively $(n-j)$ and $j$.
\end{proposition}

\subsection{Reduction to the size process}

Let us explain more explicitly how  the study of $(\XX_t, \YY_t)_{t\geq 0}$ 
can be reduced to that of the size process $(X_t, Y_t)_{t \geq 0}$, with the help of the splitting property. We denote by $\mathcal M$ the finite point measures on $\N_+$ and endow it with the weak topology and corresponding Borel algebra. Then $(X_t, Y_t)_{t \geq 0}$ is a $\mcl M^2$-valued process and  we denote by $(\mcl F_t)_{t \geq 0}$ its natural filtration. In fact, $(\mcl F_t)_{t \geq 0}$ is also the natural filtration of $(X_t)_{t\geq 0}$ since $Y_t$ is a deterministic function of  $(X_s)_{0\leq s\leq t}$ for any $t\geq 0$.

\begin{proposition}\label{prop:ClusterRRT} 
Fix arbitrary $t \geq 0$.  Conditionally on $(X^u_t)_{u \in \mcl U_t}$ and $(Y^u_t)_{u \in \mcl U^\dagger_t}$, the clusters in ${\XX_t \cup \YY_t}$ are independent RRTs whose sizes are given by $(X^u_t)_{u \in \mcl U_t}$ and $(Y^u_t)_{u \in \mcl U^\dagger_t}$. Moreover, $(X_t, Y_t)_{t \geq 0}$ is a measure-valued Markov branching process in $(\mcl M^2, (\mcl F_t)_{t \geq 0}, \P)$.
\end{proposition}
\begin{proof}
The  RRT-distribution of clusters and conditional independence
is obvious at the initial time when there is only one single vertex. Let us check that this property remains valid along time and  at the same time that the size process satisfies the Markov property. The branching property of the size process is  a direct consequence of the branching property of the cluster process. 

We need to consider three events and their corresponding rates. First, a cluster is isolated with a rate depending only on its size, and then becomes inactive. Second, the growth rate of a cluster also only depends on its size, and the new vertex is added independently of the state of other clusters. Thus, after a growth,  the new cluster remains independent from the other ones (conditionally on the sizes). Moreover, by the construction of RRT by adding a vertex uniformly at random, we know that the new cluster is also distributed as a RRT.  Third, for fragmentation, we invoke the splitting property (Proposition~\ref{prop:Split}), which guarantees that the two new clusters are independent RRTs   conditionally on their sizes.
This also ensures the Markov property thanks to the absence of memory for each event.
\end{proof}

Let us sum up the dynamics of active clusters and the  transition for the size process. Independently, each active cluster of size $n$
\begin{itemize}
\item[i)] becomes an inactive cluster of size $n$ at rate $\theta n$ \\(death of the active cluster and birth of an inactive cluster);
\item[ii)] becomes a RRT of size $(n+1)$ at rate $\beta n$ (cluster size increases by 1);
\item[iii)] splits into two RRTs
 of sizes $(n-j, j)$ at rate $\gamma n\frac{1}{j(j+1)}$, for ${n \geq 2, 1 \leq j \leq n-1}$\\ (fragmentation).  Here $(n-j)$ is the size of the first child (containing the root) and $j$ is the size of the second child. 
\end{itemize}
Thus, each active cluster of size $n$
lives an exponential time of parameter
$(\beta + \theta+\gamma)n-\gamma$.
We introduce now the infinitesimal generator
$\mathcal A$ of the Markov process $(X_t, Y_t)_{t \geq 0}$ for sizes. It is defined  on a suitable subspace of measurable bounded functions on $\mathcal M^2$. Consider two functions $f,g : \N_+\rightarrow \R_+$
and  ${F: \R^2 \to \R}$ a bounded Borel function. We set 
$${F_{f,g} : (\mu,\nu)\in\mathcal M^2\rightarrow F(\bracket{\mu, f}, \bracket{\nu, g})\in \R},$$
and define
\begin{equation}\label{eq:Fxy}
\begin{split}
 \mathcal A F_{f,g}(\mu, \nu)
&= \sum_{n=1}^{\infty} \mu(\{n\}) \beta n\Ll( F( \bracket{\mu+\delta_{n+1}-\delta_n, f}, \bracket{\nu, g}) - F(\bracket{\mu, f}, \bracket{\nu, g})\Rr) \\
& \quad + \sum_{n=1}^{\infty} \mu(\{n\}) \theta n\Ll( F( \bracket{\mu-\delta_n, f}, \bracket{\nu+\delta_{n}, g}) - F(\bracket{\mu, f}, \bracket{\nu, g})\Rr) \\
& \quad + \sum_{n=1}^{\infty} \mu(\{n\}) \gamma n   \sum_{j=1}^{n-1} \Ll( \frac{1}{j(j+1)}\Rr)\\ 
& \qquad \qquad \times \Ll( F( \bracket{\mu+\delta_{j} + \delta_{n-j}-\delta_n, f}, \bracket{\nu, g}) - F(\bracket{\mu, f}, \bracket{\nu, g})\Rr).
\end{split}
\end{equation} 
Although we start from one single active individual to define our model for convenience (see Section \ref{sec: initial} for a discussion on the initial condition), 
we remark that the size process can be defined  from any finite initial state since it is a measure-valued process on integers.

\section{First moment semigroup of $(X_t)_{t\geq 0}$ and  Perron's root }\label{sec:Root}
In this part, we study the first moment semigroup associated to the process $(X_t)_{t \geq 0}$. We will establish the existence of Perron's eigenelements and speed of convergence and prove Theorem~\ref{thm:Malthusian}.

\subsection{Semigroup and generator}\label{semigene}
Thanks to Section~\ref{subsec:RRT}, the study of the model
is reduced to the  long-time behavior of the measure-valued Markov branching process $(X_t, Y_t)_{t \geq 0}$. We denote the first moment semigroup associated to $(X_t)_{t\geq 0}$ by  $M=(M_t)_{t\geq 0}$.  
It is defined for any non-negative function $f$ by
\begin{align}\label{mtf}
M_tf(n) := \E_{\delta_n}[\langle X_t, f\rangle], \quad \forall t\geq 0, n\geq 1, 
\end{align}
where  $\P_{\delta_n}$  stands for the size process with  initial condition $(X_0, Y_0) = (\delta_n, 0)$ and  $\E_{\delta_n}$ is its associated expectation. 
In particular we consider 
for any $n,m\in \mathbb N_+$,
\begin{align*}
M_t(n,m) :=M_t\mathbf 1_m(n)= \E_{\delta_n}[\langle X_t, \mathbf 1_m\rangle]=\E_{\delta_n}\Big[\#\{ \CC \in \XX_t  \, : \, \vert \CC\vert =m\}\Big],
\end{align*}
which is  the mean number of clusters of size $m$ at time $t$ descending from one single RRT of size $n$ at time $0$.

We use the notation $\x{p}$ for the polynomial function such that $\x{p}(n) = n ^p, n\geq 1$. When $p=1$ it is just the identity function $\x{} (n) = n$.  We introduce $\mcl{B}$ the set of functions from $\mathbb N_+$ to $\R$ with at most polynomial growth:
\begin{align}\label{eq:defB}
\mcl{B} := \left\{ f :  \mathbb N_+ \rightarrow \R, \, \exists p>0 \text{ such that } \sup_{n\geq 1} |f(n)|/n^p <\infty\right\}.    
\end{align}
Let $\mcl{B}_p$ be the set of functions with $p>0$ fixed
\begin{align}\label{eq:defBp}
\mcl B_p :=\Ll\{  f: \N_+ \rightarrow \R, \, \sup_{n\geq 1}|f(n)|/n^p<\infty \Rr\}.
\end{align}

Notice that $\B_p\subset \B_{p'}\subset \B$ for $0<p\leq p'$, and $\B=\bigcup_{p>0}\B_p.$ We then extend now the first moment semigroup to the space $\mcl{B}$. 
\begin{lemma}\label{lemma:duhamel}
\begin{enumerate}[label=(\roman*)]
    \item For any $p\geq 1, t\geq 0, n\geq 1$, we have 
    \begin{align*}
        M_t([x^p])(n) \leq e^{(2^{p-1} p \beta  - \theta) t} n^p.
    \end{align*}
    
    \item For any  $f\in\mcl{B} $, setting $f_+$ (resp.\ $f_-$) to be the positive part (resp.\ negative part) of $f$, the functions $t\in [0,\infty)\rightarrow M_tf_+$ 
and $t\in [0,\infty)\rightarrow M_tf_-$  are well defined  and finite.  We can thus set
for any $t\geq 0$ and $n\in \mathbb N_+$,
$$M_tf(n)=\E_{\delta_n}\big[\langle X_t,f \rangle\big]:=M_tf_+(n)-M_tf_-(n).$$
    
    \item $(M_t)_{t\geq 0}$ is a  positive semigroup on $\mcl{B}$ and for any  $f\in \mcl{B}$, we have for $t\geq0$ and $n\geq 1$
\begin{align}\label{eq:defGenerator}
\frac{\d}{\d t} M_t f(n) = M_t (\L f)(n),
\end{align}
where the linear operator $\L : \mcl{B} \rightarrow \mcl{B}$ is defined for any $n\geq 1$ by  
\begin{multline}\label{eq:Generator}
\L f(n) 
= \underbrace{\beta n (f(n+1) - f(n))}_{\text{``growth''}} \underbrace{- \theta n f(n)}_{\text{``isolation''}} \\ + \underbrace{\sum_{j=1}^{n-1}  \frac{\gamma n}{j(j+1)} \Ll(f(j) + f(n-j) - f(n)\Rr)}_{\text{``fragmentation''}}.
\end{multline}
\end{enumerate}
\end{lemma}
\begin{proof}
Linear operator $\L$ yields the generator of the first moment semigroup of the size process.
We focus on $(i)$ and study $\L (\x{p})$ at first . Notice that for $p \geq 1$ and any $x,y>0$, $(x+y)^p \geq x^p + y^p$, 
so the contribution of the fragmentation term in $\L([x^p])$ is negative. Thus we have 
\begin{align*}
\L ([x^p])(n) \leq \beta n ((n+1)^p - n^p) - \theta n^{p+1}.
\end{align*} 
 We then apply convexity for $n \in \N_+$
\begin{align*}
(n+1)^p - n^p \leq p (n+1)^{p-1} \leq p 2^{p-1}n^{p-1}.
\end{align*}
It gives us 
\begin{align}\label{eq:lpBound}
\L ([x^p]) \leq (2^{p-1}p \beta - \theta \x{}) \x{p} \leq (2^{p-1}p \beta - \theta ) \x{p}.
\end{align}
Here we use simply the fact that $\x{}(n)=n \geq 1$ for all $n$ for the isolation term.

The rest of the proof follows classical arguments of localization, see e.g.\ Theorem 1 in  \cite{MT}, and we give here the main lines only. We assume that $X_0=\delta_n$ for any given $n\geq 1.$ 
We consider the stopped process 
$(X^m_t,Y^m_t)_{t \geq 0}$ defined by $X_t^m=X_{t\wedge \tau_m}, Y_t^m= Y_{t\wedge \tau_m}$, where 
$$\tau_m := \inf\{t \geq 0\,  : \,  \bracket{X_t,[x]}\geq m\}.$$
Notice that on the event $\{\tau_m\geq t\}$, we have $\bracket{X_t,[x]}\leq m$ and $\bracket{X_t,[x^p]}\leq m^p.$ The process $(X^m_t,Y^m_t)_{t \geq 0}$ lives on a finite state space and has bounded transition rates. 
 Consider three  functions ${f,g : \N_+\rightarrow \R_+}$
and  ${F: \R^2 \to \R}$ a bounded 
Borel function and recall that $F_{f,g}(\mu,\nu)
=F(\bracket{\mu, f}, \bracket{\nu, g})$. 
We get by  Dynkin's formula $$\E_{\delta_n}[F_{f,g}(X_t^m,Y_t^m)]=\E_{\delta_n}\left[F_{f,g}(X_0^m,Y_0^m)\right]+\E_{\delta_n}\left[\int_0^{t\wedge \tau_m} \mathcal A F_{f,g}(X_s,Y_s) \, \d s \right], $$
where $\mathcal A$   is defined in
\eqref{eq:Fxy}. 
 We apply this equation 
with $F(x,y)=x\wedge m^p$ and $f=[x^p]$ and $g = 0$ to obtain 
\begin{equation}\label{ke}
    \E_{\delta_n}[\langle X_t^m, [x^p]\rangle]= \E_{\delta_n}[\langle X_0^m, [x^p]\rangle]+\E_{\delta_n}\left[\int_0^{t\wedge \tau_m} \langle X_s, \L [x^p]\rangle  \, \d s\right].\end{equation}
Using the above display and $\eqref{eq:lpBound}$ yields 
$$\E_{\delta_n}[\langle X_t^m, [x^p]\rangle]\leq \E_{\delta_n}[\langle X_0^m, [x^p]\rangle]+(2^{p-1}p \beta - \theta)\E_{\delta_n}\left[\int_0^{t\wedge \tau_m} \langle X_s, [x^p]\rangle  \, \d s\right].$$
Since the process $(X,Y)$ is non explosive,  $\tau_m$ tends a.s.\ to infinity as $m$ tends to infinity. Moreover, $\tau_m$ is increasing in $m$.  Applying Fatou's lemma on the left hand side and monotone convergence 
on the right hand side, the above inequality yields 
$$M_t[x^p](n) \leq M_0[x^p](n)+(2^{p-1}p \beta - \theta)\int_0^t M_s[x^p](n) \, \d s.$$
Gr\"onwall's lemma then ensures that $(i)$ and  $(ii)$ are immediate consequences.

For $(iii)$, thanks to the fact $\B_p\subset \B_{p'}\subset \B = \cup_{p > 0} \B_p$ for any $p'\geq p>0$, we only need to show that for any $f\in \B_p$ with $p\geq 1$, we have 
\begin{equation}\label{bf}\E_{\delta_n}[\langle X_t, f\rangle]= \E_{\delta_n}[\langle X_0, f\rangle]+\E_{\delta_n}\left[\int_0^{t} \langle X_s, \L f\rangle  \, \d s\right].\end{equation}
Let $C>0$ be such that $|f(n)|\leq Cn^{p}, |\L f(n)|\leq Cn^{p+1}$ for all $n$. 
Similarly to \eqref{ke}, we have 
\begin{equation}\label{kef}\E_{\delta_n}[\langle X_t^m, f\rangle]= \E_{\delta_n}[\langle X_0^m, f\rangle]+\E_{\delta_n}\left[\int_0^{t\wedge \tau_m} \langle X_s, \L f\rangle \, \d s\right].\end{equation}
We will show the limits of the left and right terms in above display. We first study the left term which is the sum of the two terms below 
\begin{equation}
\label{2terms}\E_{\delta_n}[\langle X_t^m, f\rangle\Ind{t<\tau_m}]\,\,\text{ and }\,\,\E_{\delta_n}[\langle X_t^m, f\rangle\Ind{t\geq \tau_m}].
\end{equation} 
For the first term in \eqref{2terms}, we have 
$$\langle X_t^m, f\rangle\Ind{t<\tau_m}\xrightarrow{m \to \infty}  \langle X_t, f\rangle,\quad\text{ almost surely, } $$
and 
$$|\langle X_t^m, f\rangle\Ind{t<\tau_m}|\leq \langle X_t, |f|\rangle\leq C\langle X_t, [x^p]\rangle.$$
By $(i)$, the last term in the above display has a finite mean. Using dominated convergence  theorem, we obtain 
$$\lim_{m\to\infty}\E_{\delta_n}[\langle X_t^m, f\rangle\Ind{t<\tau_m}]=\E_{\delta_n}[\langle X_t, f\rangle].$$
We prove now that the second term in \eqref{2terms} 
 vanishes. We use a coupling argument below
$$\E_{\delta_n}[|\langle X_t^m, f\rangle\Ind{t\geq \tau_m}|]\leq m^p\P_{\delta_n}[\tau_m\leq t]\leq m^p\P_{\delta_n}[\langle \widetilde X_t, [x]\rangle\geq m],$$
where $(\widetilde X_t)_{t\geq 0}$ is the size process with only growth term (i.e.\ $\beta>0, \theta=\gamma=0$) and $\widetilde X_0=\delta_n$. Then $(\langle \widetilde X_t, [x]\rangle)_{t\geq 0}$ is a Yule process with initial value $n$. Thus for fixed $t$, $\langle \widetilde X_t, [x]\rangle$ follows the negative binomial distribution with parameters $n,1-e^{-\lambda}$.  Using the above display, we get 
$$\E_{\delta_n}[|\langle X_t^m, f\rangle\Ind{t\geq \tau_m}|]\leq m^p\frac{\E_{\delta_n}[(\langle \widetilde X_t, [x]\rangle)^{2p}]}{m^{2p}}\xrightarrow{m \to \infty} 0.$$ Combining the analysis of the two terms in \eqref{2terms},
the term on the left hand side  in  \eqref{kef} converges to
$\E_{\delta_n}[\langle X_t, f\rangle]$ as $m\to\infty.$

Now we turn to the right hand side of \eqref{kef}. Note that 
$$\left|\int_0^{t\wedge \tau_m} \langle X_s, \L f\rangle  \, \d s\right|\leq \int_0^{t} \langle X_s, |\L f|\rangle  \, \d s\leq C\int_0^{t} \langle X_s, [x^{p+1}]\rangle \, \d s,\quad \forall m \in \N_+.$$
Due to $(i)$, the last term has finite mean. Moreover 
${\int_0^{t\wedge \tau_m} \langle X_s, \L f\rangle  \, \d s\longrightarrow \int_0^{t} \langle X_s, \L f\rangle  \, \d s},$
 almost surely as $m\to\infty$.
Applying dominated convergence theorem
yields 
$$\E_{\delta_n}\left[\int_0^{t\wedge \tau_m} \langle X_s, \L f\rangle \, \d s\right]\xrightarrow{m \to \infty} \E_{\delta_n}\left[\int_0^{t} \langle X_s, \L f\rangle  \, \d s\right].$$
Letting  $m\to\infty$ in \eqref{kef} gives 
 \eqref{bf} and  ends the proof.  
\end{proof}

\subsection{Perron's root and eigenvectors}

In this part, we study the asymptotic behavior of the first moment semigroup $(M_t)_{t\geq 0}$ of $(X_t)_{t\geq 0}.$ Under general assumptions extending the Perron--Frobenius theory in finite dimension, the ergodic behavior of the positive semigroup is given by the unique triplet of eigenelements corresponding to the maximal eigenvalue. We refer
in particular to \cite{bansaye2019non,MS} and references therein for general statements and applications to  growth fragmentation.
In this work, we apply a general statement of 
\cite{bansaye2019non} on the ergodic behavior of positive semigroups. It allows us to exploit practical sufficient conditions which are satisfied by our process:  irreducibility properties of the dynamic of the cluster sizes  \eqref{eq:Irreducible}, and the fast splitting or isolation of large clusters which provides a Lyapunov function for a typical cluster \eqref{eq:Lyapunov}. The fact that splitting is very asymmetric (such that one child cluster is close to the parent in size) and the fact that a typical active cluster at a given time has avoided isolation make the  proof 
of the required lower bound  delicate.  It  involves a subtle compensation of fragmentation and isolation terms.  Moreover, we will show the exponential speed of convergence of the semigroup. This will
be useful  in particular for the proof of the a.s. convergences in the next section.  

In what follows, notation ``$f \leq g$'' means the point-wise comparison for functions (including constants). In this case, we say $f$ is upper bounded by $g$ or $g$ is lower bounded by $f$.  We also use $f\sim g$ to denote the fact that $\lim_{n\to\infty}\frac{f(n)}{g(n)}=1$.  We prove at first Lemma~\ref{lem:HarrisV}, the key technical ingredient  for  Proposition~\ref{Perroneigen}, the latter is the main result of this subsection. The following space of sublinear functions is useful to control the harmonic function of the semigroup:
\begin{equation}\label{eq:Sublinear}
\begin{split}
    \mcl S := \Big\{f:\N_+ \to &[1,\infty), \text{ such that } \\
    &\,\, a) \, f \text{ is increasing and } \lim_{n\to \infty} f(n) = \infty \\
    &\,\, b) \, f \text{ is sublinear } \frac{f(n+1)}{n+1} \leq \frac{f(n)}{n}, \text{ and } C_f := \sum_{j=1}^{\infty} \frac{f(j)}{j(j+1)} < \infty\Big\}.
\end{split}    
\end{equation}

\begin{lemma}\label{lem:HarrisV}
There exists a positive function $\psi$ defined on $ \N_+$,  such that
\begin{align*}
{0<\inf_{\N_+} \psi<\sup_{\N_+}  \psi \leq 1}    
\end{align*}
 and for every ${V \in \mcl{S} \cup \{\x{p}: p \geq 1\}}$:
\begin{enumerate}[label=(\roman*)]
    \item There exist real constants $a < b$ and $\zeta>0$ 
 such that    
\begin{align}\label{eq:Lyapunov}
\L V \leq a V + \zeta \psi, \qquad   b \psi \leq  \L \psi \leq \xi \psi.
\end{align}

    \item For $R$ large enough,  the set   $K = \{ x\in \mathbb N_+ :   \psi(x) \geq V(x)/R\}$ is a non-empty finite set and
 for any $x,y \in K$ and $t_0>0$,
\begin{align}\label{eq:Irreducible}
 M_{t_0}(x,y) > 0.
\end{align}
\end{enumerate}
\end{lemma} 
\begin{proof}
 To find the Lyapunov-type functions in $(i)$, the main difficulty is to find the lower bound of $\L \psi$ in \eqref{eq:Lyapunov}. As we can see in \eqref{eq:Generator}, the isolation term $- \theta n f(n)$ cannot be bounded from below uniformly in $n$ by $f$ times a constant. The  strategy is to use the growth term and fragmentation term to compensate the isolation term.

\textit{Step 1: Construction of $\psi$ - setup.} We set
\begin{align}\label{psi}
\psi(n)  := A - (A-B)q^{n-1},
\end{align}
with $A,B \in (0,\infty)$ and $q \in (0,1)$ to be chosen later. Then $\psi$ is bounded between $A$ and $B$ and $\lim_{n \to \infty}\psi(n) = A$. 
We decompose $\L \psi$ as follows
\begin{multline}\label{eq:psi1}
\L \psi(n) = \underbrace{\beta n (\psi(n+1)-\psi(n))}_{\mathbf{I}} - \underbrace{(\theta + \gamma)n\psi(n)}_{\mathbf{II}} + \gamma \psi(n) \\ + \underbrace{\gamma n\sum_{j=1}^{n-1} \frac{1}{j(j+1)} \Ll(\psi(j) + \psi(n-j)\Rr)}_{\mathbf{III}}.
\end{multline}
Firstly 
\begin{equation}\label{c1}
\vert  \mathbf{I} \vert = \vert \beta(A-B) n (q^{n-1} - q^{n}) \vert \leq C_1 \psi(n),
\end{equation}
for some $C_1>0$, since $\psi\geq \min\{A,B\}>0$.
Secondly, we observe that
\begin{align*}
\lim_{n \to \infty} \frac{\sum_{j=1}^{n-1} \Ll(\frac{1}{j(j+1)} \psi(j)\Rr)}{\psi(n)}
&= \frac{\sum_{j=1}^{\infty} \frac{1}{j(j+1)} \Ll( A - (A-B)q^{j-1}  \Rr)   }{A}\\
&=1-\Ll(1-\frac{B}{A}\Rr)q^{-1}\Big(1+(q^{-1} - 1) \ln(1-q)\Big) =: C_q^{A,B}.
\end{align*}
Here we used $\sum_{j=1}^{\infty} \frac{1}{j(j+1)}=1$ and $\sum_{j=1}^\infty \frac{q^{j-1}}{j(j+1)}=q^{-1}\Big(1+(q^{-1} - 1) \ln(1-q)\Big)$. Moreover, 
\begin{align*}
\lim_{n \to \infty} \frac{\sum_{j=1}^{n-1} \frac{1}{j(j+1)}  \psi(n-j)}{\psi(n)} &= \frac{\lim_{n \to \infty}\sum_{j=1}^{n-1} \frac{1}{j(j+1)}\Ll( A - (A-B)q^{n-j-1}  \Rr)}{A} = 1,
\end{align*}
since 
$\sum_{j=1}^{n-1} \frac{1}{j(j+1)}  q^{n-j-1}$
goes to $0$ as $n\rightarrow \infty$. 
Combining the above two displays, we obtain 
$$\mathbf{III} \sim (1+C_q^{A,B}) \gamma n \psi(n),\quad \text{ as } n\to\infty.$$

\smallskip

\textit{Step 2: Construction of $\psi$ - choice of parameters.}
We add and subtract the term ${(1+C_q^{A,B}) \gamma n \psi(n)}$ 
and reformulate \eqref{eq:psi1} as 
\begin{equation*}
\L \psi(n) = \gamma \psi(n) + \underbrace{(C_q^{A,B}\gamma  - \theta)n\psi(n)}_{\mathbf{II'}}  + \underbrace{R_1(n,q) + R_2(n,q) + \beta n (\psi(n+1)-\psi(n))}_{\mathbf{III'}},
\end{equation*}
where the term $\mathbf{III'}$ is the remainder term and 
\begin{equation}\label{eq:psiRemainder}
\begin{split}
R_1(n,q) &:= \gamma n \Ll(\sum_{j=1}^{n-1} \Ll(\frac{1}{j(j+1)} \psi(j)\Rr) -  C_q^{A,B} \psi(n)\Rr),\\
R_2(n,q) &:= \gamma n \Ll(\sum_{j=1}^{n-1} \Ll(\frac{1}{j(j+1)} \psi(n-j)\Rr) - \psi(n)\Rr).
\end{split}
\end{equation}
We choose  $q, A, B$ such that the term $\mathbf{II'}$ is $0$ (i.e.\ $C_q^{A,B}  = \theta/\gamma$) and $0<A,B\leq 1$ (then $0<\inf_{\N_+} \psi<\sup_{\N_+}  \psi \leq 1$). More precisely, we distinguish three cases:
\begin{itemize}
\item If $\gamma = \theta$, we can choose $A=B=1$.
\item If $\gamma > \theta$, we can choose $q$  close to $1$ such that ${q^{-1}\Big(1+(q^{-1} - 1) \ln(1-q)\Big) \in \Ll(1 - \frac{\theta}{\gamma}, 1\Rr)}$ and then choose $0<B<A\leq 1$ such that $C_q^{A,B}  = \theta/\gamma$.
\item If $\gamma < \theta$, it suffices to fix some $q \in (0,1)$ and then choose $0 < A < B \leq 1$  such that $C_q^{A,B}  = \theta/\gamma$.
\end{itemize}
Besides, the convergences in Step 1 ensure that there exists $C_2 \in (0, \infty)$ such that
\begin{align}\label{eq:psiTerm2}
\sup_{n \in \N_+}\Ll\{ \Ll\vert \frac{R_1(n,q)}{\psi(n)} \Rr\vert + \Ll\vert\frac{R_2(n,q)}{\psi(n)}\Rr\vert \Rr\} \leq C_2.
\end{align}
Together with \eqref{c1},
we obtain that
\begin{align*}
(\gamma - C_1 - C_2) \psi \leq \L \psi \leq (\gamma + C_1 + C_2) \psi.
\end{align*}
 This guarantees that the last two inequalities
 of  \eqref{eq:Lyapunov} hold with the following choice of parameters:
\begin{align*}
b:= \gamma - C_1 - C_2, \qquad \xi:= \gamma + C_1 + C_2.
\end{align*}

\smallskip

\textit{Step 3: Find $a, \zeta$.} 
For $V = \x{p}$ with $p\geq 1$, we  pick a real number $a$ such that $$ a < \min\{ 2^{p-1}p \beta - \theta, b\}.$$
Using the first inequality in \eqref{eq:lpBound} and distinguishing if $2^{p-1}p \beta - \theta n$ is larger than $a$ or not, 
 we can write
\begin{align*}
\L [x^p](n) 
&\leq  a n^p + (2^{p-1}p \beta - \theta n  - a) n^p\Ind{2^{p-1}p \beta - \theta n   \geq a}\\
&\leq a [x^p](n) + \zeta \psi(n),
\end{align*}
with $\zeta\in (0,\infty)$. The above result holds because $\psi$ is bounded and 
 there exist only finitely many  $n$ satisfying ${2^{p-1}p \beta - \theta n \geq a}$. This ends the proof of $(i)$ for $p\geq 1$.
 Besides,  for any large $R$,  the set $K $ is finite and non-empty. The combination of growth, fragmentation and isolation ensures the irreducibility of $(X_t)_{t\geq 0}$ which allows one to end the proof of  $(ii)$. So both $(i)$ and $(ii)$ are proved for $V=[x^p]$ with $p\geq 1.$

Now we treat the case $V \in \mcl S$ and verify the condition $(i)$  and $(ii)$ with  $\psi$ given in Step 1.  For condition $(ii)$, we only have to show that $K$ is non-empty and finite, which is straightforward to see since $V$ is increasing to infinity while $\psi$ is a bounded function. For condition $(i)$, we only have to show the first half of \eqref{eq:Lyapunov}. We calculate $\L V$ and use the decomposition in \eqref{eq:Generator}. Note that $\frac{V(n+1)}{n+1} \leq \frac{V(n)}{n}$ implies for the growth term that
\begin{align*}
    \mathbf{I} = \beta n (V(n+1) - V(n)) \leq \beta n \Ll(\frac{n+1}{n} V(n) - V(n)\Rr) \leq \beta V(n).
\end{align*}
Then the fact that  $V$ increases
and $C_V$ in \eqref{eq:Sublinear} is finite yields for the fragmentation term  $\mathbf{III}$: 
\begin{align*}
  \gamma (n-1)\sum_{j=1}^{n-1} \frac{n}{n-1} \frac{1}{j(j+1)} \Ll(V(j) + V(n-j) - V(n)\Rr) \leq \gamma n\sum_{j=1}^{n-1} \frac{V(j)}{j(j+1)} < C_V \gamma n.
\end{align*}
The above two displays entail that 
\begin{align}\label{eq:LV}
    \L V(n) \leq (\beta - \theta n) V(n) +  C_V \gamma n. 
\end{align}
Now we pick a real number $a$ such that  $a < \min\{\beta - \theta, b\}.$
Using $\lim_{n \to \infty}V(n) = \infty$, we notice that
$E := \{n \in \N_+ :  (\beta - \theta n) V(n) +  C_V \gamma n > a V(n)\}$
is a non-empty finite set. Then distinguishing the cases whether $n$ belongs to $E$ or not in  \eqref{eq:LV} yields
\begin{align*}
    \L V(n) 
    &\leq a V(n) \Ind{n \in E^c} + ( (\beta - \theta n) V(n) +  C_V \gamma n) \Ind{n \in E} \\
    &= a V(n) + ((\beta - a - \theta n) V(n) +  C_V \gamma n) \Ind{n \in E} \\
   &\leq a V(n) + \zeta \psi(n). 
\end{align*}
Here the constant $\zeta$ is defined by 
\begin{align*}
    \zeta := \max_{n \in E} \frac{(\beta - a - \theta n) V(n) +  C_V \gamma n}{\psi(n)}\in(0,\infty).
\end{align*}
We conclude that both conditions $(i)$ and $(ii)$ are verified. The whole proof is complete and finished.
\end{proof}

Now we come to the main result of this subsection which gives the existence of eigenelements and asymptotic behavior of the semigroup, based on Theorem 2.1 in \cite{bansaye2019non}.

\begin{proposition}\label{Perroneigen}
There exists a  unique  triplet $(\lambda, \pi, h)$ where $\lambda \in \R$ and ${\pi=(\pi(n))_{n\in \mathbb N_+}}$ is a positive vector of probability distribution
and  $h : \N_+\rightarrow  (0,\infty)$ is a positive function, such that for all $t\geq 0$,
$$ \pi M_t = e^{\lambda t} \pi, \qquad M_t h =  e^{\lambda t} h,$$
 and $0<\inf_{n\geq 1} h(n) \leq  \sup_{n\geq 1} h(n) <\infty$  and $\sum_{n\geq 1}\pi(n) = \sum_{n\geq 1}  \pi(n) h(n) =1.$

Besides, for every $p > 0$  there exists $C, \omega > 0$ such that for any $n,m\geq 1$, $t\geq 0$,
\begin{align}\label{eq:Spectral}
  \big\vert e^{-\lambda t} M_t(n,m) -  h(n) \pi(m) \big\vert  &\leq C n^{p} m^{-p} e^{-\omega t}, \qquad \sum_{n\geq 1} \pi(n) n^p <\infty.
\end{align}
\end{proposition}

\begin{proof}  Lemma \ref{lem:HarrisV} together with Lemma \ref{lemma:duhamel} (iii)
ensures that the semigroup $M$   satisfies
the drift and irreducibility conditions  given in \cite{bansaye2019non} (Propositions 2.2 and 2.3 therein).  More precisely, these conditions
are met with $V=[x^p]$ for $p > 0$ and $\varphi=\psi$ (defined in \eqref{psi}), while  $\psi \leq V$ is guaranteed by the fact that $\psi\leq 1$. Using these conditions, 
we can apply Theorem 2.1 in \cite{bansaye2019non}, which yields the proposition, except the boundedness of $h$. In particular,  \eqref{eq:Spectral} is obtained  by specifying the initial condition $\mu=\delta_n$ and using the test function $1_m$. 

We prove at first that $h$ is upper bounded.
Lemma  3.4  in \cite{bansaye2019non} ensures
that $h$ is upper bounded by $V$ times a constant (i.e.\ $h\lesssim V$ in their notation). Adding that
 Lemma~\ref{lem:HarrisV} guarantees that  
we can pick a $V \in \mcl{S}$ that increases arbitrarily slowly, we obtain that
 $h$ is upper bounded.
Finally, we justify that $h$ is lower bounded. Indeed $h(n)=e^{\lambda}M_1h(n)\geq ch(1)$ for all $n\geq 1$, where $c>0$. This is because, due to \eqref{eq:Split}, the probability that a cluster of size $n$ produces (by fragmentation) a cluster of size one before unit time $1$ and that this latter stays unchanged in the remaining time within the unit time  is lower bounded by a positive constant independent of its size $n$. 
\end{proof}
Equation \eqref{eq:Spectral}  ensures that for any $f\in \B_p,$ 
  \begin{align}\label{eq:fSpectral}
 \big\vert e^{-\lambda t} M_tf(n) -  h(n) \bracket{\pi,f} \big\vert  &\leq C n^{p+2} \parallel f\parallel_p e^{-\omega t},
\end{align}
where ${\parallel f\parallel_p :={\sum_{m\geq 1}m^{-(p+2)}|f(m)|<\infty}}$. This result will be useful later. 

The fact that the eigenvector $h$ is lower and upper bounded in $(0,\infty)$ ensures that the impact of  the size of the initial cluster on the first order approximation of the profile 
remains bounded. 
Furthermore, we expect that restricting the set of test functions to bounded functions, the current result \eqref{eq:Spectral} can be enhanced as uniform ergodic convergence.
Indeed, we may  apply
 Theorem  3.5 in \cite{BCG} with
$\nu=\delta_1$, using again that  large clusters produce with high probability  clusters of small sizes at fragmentation and  are fast isolated. The remaining difficulty lies in controlling uniformly $M_t{\bf 1}(n)/M_t{\bf 1}(1)$ in time and size. \\



At this point one may want to apply \cite{AsmussenHering} to
prove strong convergence using the asymptotic behavior of the first moment semigroup $(M_t)_{t\geq 0}$. But \cite{AsmussenHering} requires stronger assumptions than what is obtained in \eqref{eq:Spectral}, in particular in terms of 
the stationary distribution $\pi$. 
Besides,  we are interested in finer and more quantitative estimates, with    motivations coming from inference
and epidemiology. We thus  follow  another approach via $L^2$ estimates and control of fluctuations.

\subsection{$L^2$ martingale}
Using the first moment semigroup, we can compute the second moment of $\langle X_t, f\rangle$ for $f\in \B$, which consists in the so-called \textit{formula for forks}  or \textit{many-to-two formula}, see e.g. \cite{BDMT, Marguet} and references therein. The idea is to use the most recent common ancestor of two individuals  to decouple their values.
\begin{lemma}
\label{lem:L2}
For any $x\in \N_+$ and   $f\in\B$, we have
 \begin{multline*}
 \E_{\delta_x}\left[ \langle X_t,f\rangle^2 \right] \\
 = M_t(f^2)(x) +2\int_0^t \sum_{n\geq 1} M_s(x, n) \Ll( \sum_{ 1\leq j \leq n-1} \kappa(n,j)  M_{t-s}f(j)M_{t-s}f(n-j)  \Rr) \,   \d s,
 \end{multline*} 
 where $\kappa(n,j) = \frac{\gamma n}{j(j+1)}$ is the rate at which a cluster of size $n$ breaks into two clusters of sizes $(n-j)$ (first child) and $j$ (second child).
\end{lemma}

\begin{remark}Combining this identity with the estimates on the semigroup $M$ obtained in the previous subsection gives $L^2$ convergence of the empirical measure of clusters, with exponential speed and size dependency; see next section.\end{remark}
\begin{proof}
Recalling notations in Section \ref{UHN},
we have
$$\langle X_t,f\rangle=\sum_{u\in \mcl{U}_t} f(X_t^u).$$
Recall also that, for any $u, v \in \mathcal{U}$,   $u\wedge v$ is the label of the most recent common ancestor of $u$ and $v$, and $\mcl U_t(u)$ as the active clusters at $t$ issued from $u$. We first notice that
\begin{equation}\label{eq:mart1}
\begin{split}
 \langle X_t,f\rangle^2 =\sum_{u,v\in \mcl{U}_t} f(X_t^u)f(X_t^v) &=\sum_{u\in \mcl{U}_t} f^2(X_t^u)+\sum_{\substack{   w \in \mcl U }} \sum_{\substack{u,v \in \mcl U_t, \\ u \neq v,  u \wedge v = w}} f(X_t^u)f(X_t^v) \\
 &=\sum_{u\in \mcl{U}_t} f^2(X_t^u)+\sum_{\substack{   w \in \mcl U }} \Ind{b(w)<t}  I_t(w),
\end{split}
\end{equation}
where for any $w \in \mathcal U$, $b(w)$ is the time at which the cluster labeled by $w$  branches  (i.e. the time  when it splits into two clusters, labeled $w1$ and $w2$; potentially infinite if that does not happen due to isolation) and
\begin{align*}
I_t(w)= \sum_{\substack{i,j \in \{1,2\},\, i\ne j\\
 u\in \mcl U_t(wi),\, v\in  \mcl U_t(wj)  } } f(X_t^u)f(X_t^v)=2
 \left( \sum_{u\in \mcl U_t(w1)} f(X_t^u) \, \times\,
\sum_{v\in \mcl U_t(w2)} f(X_t^v)\right).
\end{align*}
Thus $I_t(w)$ is the cross term between the active clusters on the two genealogical subtrees rooted at $w1$ and $w2$.

Concerning the equation \eqref{eq:mart1}, we have firstly
\begin{align*}
\E_{\delta_x}\left[ \sum_{u\in \mcl{U}_t} f^2(X_t^u)\right]= M_t(f^2)(x).
\end{align*}
Secondly, we deal  with $\E_{\delta_x}\left[\sum_{  w \in \mathcal U } \Ind{b(w)<t}\,  I_t(w)\right]$.
For any $w\in \mathcal U$ and for any $i\in\{1,2\}$, we use strong Markov property to get
\begin{align*}
\Ind{b(w)<t}\, \E_{\delta_x}\left[\sum_{u\in \mcl U_t(wi)} f(X_t^u) \, \Big\vert  \, b(w), X_{b(w)}^{wi}\right]=\Ind{b(w)<t}\, M_{t-b(w)}f(X_{b(w)}^{wi}).
\end{align*}  
For any  $w \in \mathcal U$, the branching property then yields 
\begin{align*}
\Ind{b(w)<t} 
\E_{\delta_x}\left[I_t(w) \, \big\vert  \mathcal F_{b(w)}, b(w) \right]
  = 2 \Ind{b(w)<t} \, M_{t-b(w)}f(X_{b(w)}^{w1})M_{t-b(w)}f(X_{b(w)}^{w2}).
\end{align*}
 Combining  these identities, we obtain
 \begin{align*}
 \E_{\delta_x}\left[\sum_{ w\in \mathcal U}
 \Ind{b(w)<t}\, I_t(w) \right] &= 2 \E_{\delta_x}\left[  \sum_{w\in \mathcal{U}} \Ind{b(w)<t}  \, M_{t-b(w)}f(X_{b(w)}^{w1})M_{t-b(w)}f(X_{b(w)}^{w2})   \right]\\
 &=2\E_{\delta_x}\left[  \sum_{w\in \mathcal{U}} \Ind{b(w)<t}  \,  g (X^{w}_{b(w)-},b(w))\right],
\end{align*} 
where we introduce
\begin{align*}
g(X^{w}_{b(w)-},b(w)) :=  \E_{\delta_x}\left[ M_{t-b(w)} f(X_{b(w)}^{w1})M_{t-b(w)}f(X_{b(w)}^{w2}) \, \big\vert X^{w}_{b(w)-},b(w)  \right].
\end{align*} 
This function involves the fragmentation event and can be written explicitly by recalling that, when a  cluster of size $n$ splits, the probability that the size of the first child is $(n-j)$ and the second child is $j$ is $n/((n-1)\cdot j\cdot (j+1))$. That is, we can obtain
\begin{align}\label{eq:defg}
g(n,s) &= \sum_{1\leq j\leq n-1} \frac{n}{n-1} \frac{1}{j(j+1)} M_{t-s} f(j)M_{t-s}f(n-j).\
\end{align}
Adding that  the branching rate of a cluster of size $n$ is $\gamma(n-1)$, we obtain
\begin{align*}
&\E_{\delta_x}\left[\sum_{w \in \mathcal U} \Ind{b(w)<t}\,  g(X^{w}_{b(w)-},b(w))\right]\\ 
&=\int_{0}^t  \sum_{w \in \mathcal U, n\geq 1} g(n,s) \P_{\delta_x}[ w\in \mathcal U_{s-},\,  X^{w}_{b(w)-}=n,  \, b(w) \in \d s]\\
&=\int_{0}^t  \sum_{w \in \mathcal U, n\geq 1} g(n,s) \P_{\delta_x}[ w\in \mathcal U_{s-}, X^{w}_{b(w)-}=n]\, \gamma(n-1) \d s \\
&=\int_{0}^t  \sum_{n\geq 1} g(n,s)  \gamma(n-1)   M_s(x,n) \, \d s.   
\end{align*}
This equation and \eqref{eq:defg} give us the expression of $\kappa$. The proof is thus completed.
\end{proof}

With the help of this $L^2$ expression, we can deal with the martingale associated to the harmonic function $h$.
\begin{proposition}\label{prop:Mt}
The process $(\MM_t)_{t \geq 0}$ defined as 
\begin{align}\label{eq:Mt}
\MM_t = e^{-\lambda t} \bracket{X_t, h},
\end{align}
is a non-negative  martingale, which converges almost surely to a non-negative finite random variable $W$ as $t$ tends to infinity. Moreover, if $\lambda > 0$, $(\MM_t)_{t \geq 0}$  converges in the $L^2$ norm to $W$.
\end{proposition}
\begin{proof}
The martingale property is classical and the proof is given for the sake of completeness. Recall the notation $\mathcal U_t$ and $X_t^u$ introduced in Section \ref{UHN}. For any $u\in \mathcal U_t,$ recall that $\mathcal U_{t+s}(u)$ is the set of labels of all the clusters active at time $(t+s)$ issued from the cluster labeled by $u$ active at time $t$.
Then we have
\begin{align*}
\E\Ll[\MM_{t+s} \,\vert \mcl F_t\Rr] & =  e^{- \lambda(t+s)} \E\Ll[\sum_{u \in \mcl U_{t+s}}h(X^u_{t+s}) \, \Big\vert \mcl F_t  \Rr] \\
& = e^{- \lambda(t+s)} \sum_{u \in \mcl U_t}\E_{\delta_{X^u_t}}\Ll[\sum_{v \in \mcl U_{t+s}(u)}h(X^{v}_{t+s}) \, \Big\vert \mcl F_t  \Rr] \\
&= e^{- \lambda(t+s)} \sum_{u \in \mcl U_t} M_s h(X^u_t) =\MM_t,
\end{align*}
since $M_s h= e^{\lambda s}h$.
As $\MM$ is non-negative,  it converges almost surely to a finite random variable. 

Let us now prove the $L^2$ convergence. We apply Lemma \ref{lem:L2}  with $x=1$ and obtain 
\begin{align*}
 &\E\left[ \langle X_t,h\rangle^2 \right] \\
 &= M_t(h^2)(1) + 2\int_0^t \sum_{n\geq 1} M_s(1, n) \Ll( \sum_{ 1\leq j \leq n-1} \kappa(n,j)  M_{t-s}h(j)M_{t-s}h(n-j)  \Rr) \,   \d s \nonumber\\
 &= M_t(h^2)(1) +  2 e^{2\lambda t} J_t,
\end{align*} 
where
\begin{align*}
J_t = \int_0^t \sum_{n\geq 1} e^{-2\lambda s} M_s(1, n) \Ll( \sum_{ 1\leq j \leq n-1} \kappa(n,j) h(j)h(n-j)  \Rr) \,  \d s.
\end{align*} 
Using that
$\kappa(n,j) = \gamma n/(j(j+1))$  for all $n\geq 1$ and $1\leq j \leq n-1$
and that 
 $h$ is bounded given in Proposition~\ref{Perroneigen}, we get
 that $\sum_{ 1\leq j \leq n-1} \kappa(n,j) h(j)h(n-j)$ grows at most linearly with $n$.
Moreover  we can apply
 \eqref{eq:fSpectral} to control the gap
 between 
 $e^{-\lambda s }M_s(1, n)$ and $h(1)\pi(n)$.
Combining these estimates ensures that  for any $p>2$,  there exists $C>0$ such that
\begin{align*}
0\leq J_t \leq C \int_0^t e^{-\lambda s}\sum_{n\geq 1} n \gamma(h(1)\pi(n)+ n^{-p}) \d s, \quad \forall n\geq 1, t\geq 0,
\end{align*}
which is uniformly upper bounded for all $n\geq 1, t\geq 0$. Adding that $\pi(n)$ decreases to $0$ faster than $n^{-3}$ ensures
that 
$\sup_{t\geq 0 }J_t<\infty.$
Finally
$$\E\left[(\MM_t)^2\right]=e^{-2\lambda t}\E\left[ \langle X_t,h\rangle^2 \right]=e^{-2\lambda t}M_t(h^2)(1)+ J_t,$$
and we apply \eqref{eq:fSpectral} to conclude that $\sup_{t\geq 0}\E\left[(\MM_t)^2\right]<\infty.$ Then by the martingale convergence theorem, we obtain that $\E[W^2]<\infty$ and $(\MM_t)_{t\geq 0}$ converges in the $L^2$ norm to $W$. 
\end{proof}

\begin{remark}\label{rmk:W}
In fact, Proposition~\ref{prop:Mt}  holds for  $(X_t)_{t \geq 0}$ under $\P_{\delta_n}$ for any $n\geq 1$, with the limit $W$ depending on $n$. The proof is essentially the same as for $n=1$.   We state here the result under $\P=\P_{\delta_1}$ so that the limit $W$ is consistent with Theorem~\ref{thm:LLN} and Corollary~\ref{cor:LLNY}.
\end{remark}

\subsection{Proof of Theorem~\ref{thm:Malthusian}}

With the help of Proposition~\ref{Perroneigen} and Proposition~\ref{prop:Mt}, we are now ready to prove our  Theorem~\ref{thm:Malthusian}.
 
\begin{proof}[Proof of Theorem~\ref{thm:Malthusian}]
We  use the classical notation $\liminf$ and $\limsup$ respectively for the limit inferior and limit superior of a sequence defined on discrete or continuous time.  We notice that 
$\E[\vert \XX_t \vert] = \E[\bracket{X_t, 1}] = \sum_{j=1}^{\infty}M_t(1,j)$
and  we apply \eqref{eq:fSpectral} with $f \equiv \mathbf{1}$ (constant function) and  $n = p = 1$. This ensures
 that ${\lim_{t\rightarrow\infty}  \log(\E[| \XX_t|])/t = \lambda}$. 

To study the  limit of $ \log(\E[| \YY_t|])/t$, we use   Kolmogorov's equation. 
More precisely,  following the localization argument in the proof of Lemma \ref{lemma:duhamel} $(ii)-(iii)$, we check that  ${F_{f,g}(\mu, \nu) = \bracket{\nu, 1}}$ belongs to the domain of the extended generator of $(X_t,Y_t)_{t\geq 0}$ (see \eqref{eq:Fxy}). 
Then we get
\begin{align*}
\E[| \YY_t|] = \E[\bracket{Y_t, 1}] = \int_0^t \E[\bracket{X_s, \theta \x{}}] \, \d s = \int_0^t M_s (\theta \x{})(1) \, \d s,
\end{align*}
and we conclude using  \eqref{eq:fSpectral}.

Lastly, we study the the survival probability $\P[\tau = \infty]$. 
\begin{itemize}
    \item In the subcritical phase $\lambda < 0$,   \eqref{eq:fSpectral} and the classical first moment estimate prove that extinction is almost sure.
    \item In the supercritical phase $\lambda > 0$,  we use the $L^2$ martingale 
of Proposition~\ref{prop:Mt} and  the optional stopping theorem to get
\begin{align*}
h(1) = \E\Ll[\lim_{t \to \infty} e^{-\lambda (t \wedge \tau)}\bracket{X_{t \wedge \tau}, h}\Rr] = \E\Ll[ W \Ind{\tau = \infty}\Rr]. 
\end{align*}
Adding that $h > 0$ from Proposition~\ref{Perroneigen}  implies that $\P[\tau = \infty] > 0$ and $\P[W>0]>0$. 
    \item In the critical phase $\lambda = 0$, we first observe that
the probability of extinction 
   starting from one cluster, within a unit time, is greater than a positive constant (uniformly
 with respect to the cluster size that we started with).
 Besides, $\liminf_{t\rightarrow \infty}
 | \XX_t|<\infty$ a.s.\ since Fatou's lemma ensures that
$$
    \E[\liminf_{t\rightarrow \infty}\vert \XX_t\vert] = \E[\liminf_{t\rightarrow \infty}\bracket{X_t,\mathbf{1}}] \leq \lim_{t\rightarrow \infty} \E[\bracket{X_t, h}]/\inf_{n\geq 1}h(n)=h(1)/\inf_{n\geq 1}h(n)<\infty.
$$
  This ensures that extinction occurs a.s.\ in finite time by a classical argument for Markov processes with accessible absorbing points. Indeed,
for any $K\geq 1$, on the event ${\liminf_{t\rightarrow \infty} | \XX_t| \leq  K}$, extinction occurs a.s. since we can construct an infinite sequence of stopping times $T_n$ (separated at least by a unit time)
 such that $| \XX_{T_n}|\leq K$ and for each $n$, extinction occurs with a positive (lower bounded) probability during $[T_n,T_{n}+1]$. 
\end{itemize}
 \end{proof}

\section{Strong convergences}\label{sec:Limit}
The Perron's root $\lambda \in \R$ and associated eigenelements have been characterized in Proposition \ref{Perroneigen}. The sign of $\lambda$ determines if the first moment semigroup goes to $0$ or infinity. We turn now to trajectorial results and find first an equivalent statement of being extinct (Kesten--Stigum). 
Then we focus on the supercritical regime $\lambda>0$ and prove strong law of large numbers for the distribution of clusters.

\subsection{Kesten--Stigum limit theorem}

A  fundamental and classical question is whether $\{W>0\}$ coincides with survival event $\{\tau = \infty\}$ or not. This is one part of the Kesten--Stigum theorem in branching processes, see e.g. \cite{kesten1966limit, kurtz1997conceptual}. In our case, the $L^2$ computation ensures that $\P[W>0]>0$ 
and we will get a positive answer to the question posed. 
\begin{proposition}\label{prop:KS} Assume $\lambda>0$. Then  $\P[W>0]>0$, and
$\{ W >0 \}=\{\tau = \infty \} \, \text{a.s.}$. 
\end{proposition}
\begin{proof} 
The fact that $\P[ W > 0] > 0$ 
comes from  the $L^2$ martingale convergence in Proposition~\ref{prop:Mt}. 
Besides, it is straightforward to see that
$\{ W >0 \} \subset \{\tau = \infty \}$. Thus if $\P[ W > 0] = \P[\tau = \infty ]$, the proof is done.  
The  lines of the proof are classical, even though the sizes of clusters being unbounded requires some specific arguments. 

First, we use the fact that any cluster can be isolated (before any other event happening to it) during a unit time interval, with the isolation probability lower bounded by a positive value for all sizes. As a result, the number of clusters has to tend to infinity to survive:
$$\{\tau=\infty\} =\{ \lim_{t \to \infty} \vert \XX_t \vert = \infty\} \quad \text{a.s.} .$$
Second, we deduce from the above result that the number of clusters of size $1$ tends to infinity on the survival event.
Indeed, during a unit time interval, 
clusters of size one have a positive probability to stay unchanged and other clusters have a positive probability to create (by fragmentation) one cluster of size one, and the latter probability is lower bounded uniformly with respect to the size $n\geq 2$. 
By independence of clusters and Markov inequality, this ensures that
\begin{align*}
    \limsup_{t\rightarrow \infty} X_t(1)=\infty \quad \text{a.s.} \text{ on }\{\tau = \infty\}.
\end{align*}
On the event $\{\tau=\infty\}$, we can thus  define a sequence of stopping times for $N\geq 1$  
\begin{align*}
\tau_N := \inf \{t :  X_t(1) \geq N\}.
\end{align*}
We obtain for $t \geq \tau_N$ 
\begin{align*}
e^{- \lambda t}\bracket{X_t, h} &\geq e^{- \lambda \tau_N} \sum_{u \in \mcl A_N} e^{-\lambda (t - \tau_N)} \sum_{v \in \mcl U_t(u)} h(X^v_t), 
\end{align*}
 where $\mcl A_N := \{ u \in \mcl U_{\tau_N}: X^u_{\tau_N} = 1\}.$
By Proposition~\ref{prop:Mt},  ${e^{-\lambda (t - \tau_N)} \sum_{v \in \mcl U_t(u)} h(X^v_t)}$  converges to a non-negative
 random variable denoted by $W(u)$ which is equal in law to $W$. Besides,
 $\{W(u)\}_{u \in \mcl A_N}$ are i.i.d.\ random variables. Thus we have 
\begin{align*}
\P[W = 0, \tau = \infty] 
&\leq \P[W = 0, \tau_N < \infty] \\
&\leq \P[\{\tau_N < \infty\} \cap \{W(u)=0, \forall u \in\mcl A_N\}] \leq  (\P[W = 0])^N.
\end{align*}
As a result, $\P[W = 0]=1$ or $\P[W = 0, \tau = \infty]=0$ (by letting $N \to \infty$). Only the latter is possible since we know that $\P[W>0]>0.$ The proof is thus finished. 
\end{proof}

\subsection{Strong law of large numbers for the size process of active clusters}\label{subsec:LimitSize}
In this part, we prove Theorem~\ref{thm:LLN} using the estimates of the first moment semigroup, the $L^2$ estimates and the martingale associated to the harmonic function. The  $L^2$ estimates ensure weak convergence, and the convergence speed obtained  entails strong convergence of subsequences. 
The strategy is then to  control fluctuations to prove the strong convergence along $t \in \R_+$. To that purpose, we follow the idea from \cite{athreya1968some}.  We divide the proof  into three steps.
\begin{proof}[Proof of Theorem~\ref{thm:LLN}]
\textit{Step 1: $L^2$ convergence.}
We prove first the $L^2$ convergence of $e^{- \lambda t}\bracket{X_t, f}$ to $W \bracket{\pi, f}$ for any $f\in \B_p$ with $p>0$. We develop the difference as follows 
\begin{multline}\label{eq:Functional}
 e^{- \lambda t}\bracket{X_t, f} -  W \bracket{\pi, f} \\
= \underbrace{e^{- \lambda t}\bracket{X_t, f} - e^{- \lambda t}\bracket{X_t, h}\bracket{\pi, f}}_{{\mathbf{I}}} + \underbrace{e^{- \lambda t}\bracket{X_t, h}\bracket{\pi, f} - W \bracket{\pi, f}}_{{\mathbf{II}}}. 
\end{multline}
The second term $\mathbf{II}$ is nothing but $(\MM_t - W)\bracket{\pi, f}$, which converges in $L^2$ to $0$ by Proposition~\ref{prop:Mt}. We only have to prove the $L^2$ convergence of the term $\mathbf{I}$ to $0$. Denoting  $g:= f - \bracket{\pi, f}h$,   Lemma~\ref{lem:L2} yields
\begin{align}\label{g2j}
e^{2\lambda t}\E[\vert \text{$\mathbf{I}$} \vert^2]&= \E\left[\bracket{X_t, g}^2 \right] = M_t(g^2)(1) + J_t,
\end{align}
where 
\begin{align*}
J_t := 2\int_0^t \sum_{n\geq 1} M_s(1, n) \Ll( \sum_{ 1\leq j \leq n-1} \kappa(n,j)  M_{t-s}g(j)M_{t-s}g(n-j)  \Rr) \,   \d s.
\end{align*}
Recall that $\parallel f \parallel_p = \sum_{m\geq 1}\vert f (m)\vert m^{-(p+2)}\in(-\infty,\infty)$. Observe that 
$g\in  \mcl B_p$ and  let ${p'> 2p+8}$. By \eqref{eq:fSpectral}, there exists $C'>0$ such that for any $n\in \N_+$ and $s,t\geq 0,$ 
\begin{align*}
\vert e^{-\lambda t }M_tg(n)-h(n)\bracket{\pi,g} \vert &\leq C'n^{p+2} \parallel g \parallel_p
e^{-wt}, \\
\vert e^{-\lambda s }M_s(1,n)-h(1)\pi(n) \vert &\leq C'n^{-p'}e^{-ws}.
\end{align*} 
Since $\bracket{\pi, g}=0$ and $\kappa(n,j)\leq \gamma n$, using the above two displays, there exists $C_1>0$ such that 
 \begin{align*}
 \vert J_t \vert 
 &\leq C_1\parallel g \parallel_p e^{2(\lambda-\omega)t}\int_0^t e^{(2\omega-\lambda)s} \sum_{n\geq 1}n^{2p+6}(h(1)\pi(n)+n^{-p'}) \, \d s.
 \end{align*}
Using the second statement in  \eqref{eq:Spectral} and $p'>2p+8$, the sum $\sum_{n\geq 1}n^{2p+6}(h(1)\pi(n)+n^{-p'})$ in the above display is finite and there exists $C_2>0$ such that 
 $$|J_{t}|\leq C_2\parallel g \parallel_p e^{2(\lambda -\omega)t}\int_0^te^{(2\omega-\lambda)s} \d s,\quad \forall n\geq 1, t\geq 0.$$
Moreover, by \eqref{eq:fSpectral}, $e^{-2\lambda t}M_t(g^2)(1)\leq \bracket{\pi,g^2}h(1)e^{-\lambda t}+C\parallel g^2 \parallel_pe^{-(\lambda+\omega) t}, $ for any $t\geq 0.$ Plugging  in these estimates to \eqref{g2j}, we see that there exists $C_3>0$ such that 
\begin{equation}\label{eq:L2Error}
\E[\vert \mathbf{I} \vert^2]\leq C_3\left(\bracket{\pi,g^2}+\parallel g^2 \parallel_p+\parallel g \parallel_p\right)te^{-(\lambda \wedge 2\omega)t}, \quad \forall t\geq 0.
\end{equation}
Note that $C_1,C_2,C_3$ do not depend on $f$ and $t$.  Then step 1 is finished.


\begin{remark}
A byproduct of \eqref{eq:L2Error} and Proposition~\ref{prop:Mt} is that, for the case $\lambda > 0$ there exists a constant $C_0 > 0$ and an exponent $\sigma \in (0, \lambda)$, such that for any $f \in \mcl B_p$, 
\begin{align}\label{eq:XL2Bound}
\E\Ll[ \bracket{X_t, f}^2\Rr] \leq  C_0 e^{2\lambda t}\Big(\vert \bracket{\pi, f}\vert^2 + \left(\bracket{\pi,g^2}+\parallel g^2 \parallel_p+\parallel g \parallel_p\right)  e^{ - \sigma t}\Big),\quad \forall t\geq 0, 
\end{align}
with the notation $g = f - \bracket{\pi, f}h$. 
\end{remark}

\smallskip

\textit{Step 2: Almost sure convergence for one type.} The main idea is to extend an elegant argument from \cite{athreya1968some} to our countable-type branching process. 

First, we  establish  an almost sure  convergence for a discrete scheme, using the speed of convergence obtained from the $L^2$ estimates.
 We can pick a step size $\Delta > 0$  and apply the decomposition \eqref{eq:Functional}. Then the martingale part $\mathbf{II}$ converges to $0$ almost surely, and for the term $\mathbf{I}$,   \eqref{eq:L2Error} yields
\begin{align}
\E\Ll[ \vert e^{- \lambda k \Delta}\bracket{X_{k\Delta}, f} - e^{- \lambda k \Delta}\bracket{X_{k\Delta }, h}\bracket{\pi, f} \vert^2 \Rr] \leq C k \Delta e^{- (\lambda \wedge 2 \omega) k\Delta}, \quad \forall k\geq 0.
\end{align}
By Borel--Cantelli lemma, we get
\begin{equation}\label{nconv}
{e^{- \lambda k \Delta}\bracket{X_{k\Delta}, f}  \xrightarrow{k \to \infty}  W \bracket{\pi, f} },  \qquad \text{ almost surely. }
\end{equation} 
Recall that on the event $\{W=0\}$, by Proposition~\ref{prop:KS}, the extinction occurs a.s.. So 
we focus on the event $\{W > 0\}$.
Let us first prove that
\begin{align}\label{eq:psXn}
e^{-\lambda t} X_t(n) \xrightarrow{t\to\infty} W\pi(n), \qquad \text{almost surely}.
\end{align}
Given the almost sure convergence in discrete times, we need to control the fluctuations in the intervals $[k\Delta, (k+1)\Delta)$. A nice observation in \cite{athreya1968some} is that we only need to prove the following sufficient  (and necessary) condition
\begin{align}\label{eq:psKeyLem}
\liminf  \limits_{t \to \infty} e^{-\lambda t} X_t(n) \geq W\pi(n), \qquad \text{almost surely for all } n\geq 1.
\end{align}
We first show that \eqref{eq:psKeyLem}  implies  \eqref{eq:psXn} using
 that the martingale convergence controls the dissipation of mass. To this purpose, for any fixed $n\geq 1$, consider any sequence of (random) times $(t_k)_{k \in \N_+}$ such that
 $\limsup  \limits_{t \to \infty} e^{-\lambda t} X_t(n)=\lim_{k \to \infty} e^{-\lambda t_k} X_{t_k}(n)$.
  Proposition~\ref{prop:Mt} and Fatou's lemma and 
  \eqref{eq:psKeyLem} ensure 
 \begin{equation}\label{eq:limsupTrick}
     \begin{split}
          \limsup  \limits_{t \to \infty} e^{-\lambda t} X_t(n)h(n)&=\lim_{k\rightarrow\infty} \left(\sum_{i\geq 1} e^{-\lambda t_k} X_{t_k}(i)h(i) -\sum_{i\geq 1, i\ne n} e^{-\lambda t_k} X_{t_k}(i)h(i)\right)\\
 & \leq W-\sum_{i\geq 1, i\ne n} \liminf  \limits_{k \to \infty} e^{-\lambda t_k} X_{t_k}(i)h(i)\\
 & \leq W -\sum_{i\geq 1, i\ne n} W\pi(i)h(i)=W\pi(n)h(n).
     \end{split}
 \end{equation}
Then together with \eqref{eq:psKeyLem} we obtain  \eqref{eq:psXn}.

We need now to prove \eqref{eq:psKeyLem} 
following the argument in \cite{athreya1968some}. Recall  $\Delta > 0$ is the time step size. The proof relies on the following lower bound: 
\begin{align}\label{eq:psFluctuation}
\forall t\in [k\Delta, (k+1)\Delta), \qquad    X_t(n) \geq X_{k\Delta}(n) - N_{ k,\Delta}(n),
\end{align}
where $N_{k,\Delta}(n)$ is the number of active clusters of size $n$ at time  $k\Delta$  that will encounter at least one event within $(k\Delta, (k+1)\Delta)$. Indeed, to prove \eqref{eq:psKeyLem}, we can find a lower bound for $\liminf_{t\to\infty}e^{-\lambda t}X_t(n)$ using the above display. 
More precisely,
\begin{align*}
\liminf  \limits_{t \to \infty} e^{-\lambda t} X_t(n)\geq \liminf  \limits_{k \to \infty} e^{-\lambda (k+1)\Delta } X_{k\Delta} (n) -
\limsup \limits_{k \to \infty} e^{-\lambda k \Delta } N_{ k,\Delta}(n).
\end{align*}
Using \eqref{nconv} for the first term on the right hand side, we obtain
$$\liminf  \limits_{t \to \infty} e^{-\lambda t} X_t(n)\geq  e^{-\lambda \Delta } \pi(n) W -
\limsup \limits_{k \to \infty} e^{-\lambda k \Delta } N_{ k,\Delta}(n).$$
It suffices to prove that  $\lim_{k \to \infty} e^{-\lambda k\Delta } N_{ k,\Delta}(n)=0$ a.s. and then let $\Delta$ go to $0$.
We introduce  
\begin{align*}
D_k =D_{\Delta, n,k,\epsilon}
:=\left\{N_{k,\Delta}(n)>\epsilon X_{k\Delta }(n), \quad X_{k\Delta}(n)>k\right\}, \quad k\geq 1.
\end{align*}
By branching property, we know that
\begin{align*}
N_{k,\Delta}(n)   \eqdist \sum_{i=1}^{ X_{k\Delta}(n)} \xi_i,
\end{align*}
where $\{\xi_i\}_{i \geq 1}$ are i.i.d. Bernoulli random variables, independent of $ X_{k\Delta}(n)$ and 
$${\P[\xi_i = 0 ]= 1 -\P[\xi_i = 1] = \exp(- r_n \Delta)}, \qquad  {r_n = (\beta + \theta + \gamma)n - \gamma}.$$
Indeed, $r_n$ is the total jump rate
 of an active cluster of size $n$.
Choose $\Delta$ small such that ${\P[\xi_i = 1]<\epsilon.}$ Then, 
\begin{align*}
\sum_{k\geq 1}\P [D_k] \leq \sum_{k\geq 1}\P\left[N_{k,\Delta}(n) >\epsilon X_{k\Delta }(n)\,\,|\,\, X_{k\Delta }(n)>k\right]<\infty,
\end{align*}
using that $\P[\sum_{i=1}^k \xi_i >\varepsilon k]$ decreases exponentially as $k$ grows thanks to  Hoeffding inequality. Borel--Cantelli lemma then ensures that a.s.\ $D_k$ happens a finite number of times.


Recalling now from \eqref{nconv} that  $X_{k\Delta }(n)$ grows exponentially on the event $\{W>0\}$, so $\{X_{k\Delta}\leq k\}$ also happens a.s.\ a finite number of times. As a result, a.s.\ on the event 
$\{W>0\}$, we have $N_{k,\Delta}(n) \leq \epsilon X_{k\Delta }(n)$ for $k$ large enough. We conclude that $\lim_{k \to \infty} e^{-\lambda k\Delta } N_{ k,\Delta}=0$ a.s.\ on the event $\{W>0\}$ since $\varepsilon$ can be arbitrarily small. This ends the proof of \eqref{eq:psKeyLem} and we obtain \eqref{eq:psXn}. We have thus also proved Theorem~\ref{thm:LLN} for functions with bounded  support. 

\smallskip

\textit{Step 3: Almost surely convergence - general test function.} 
It suffices to prove the results for functions in 
$\B_p$ with $p\geq 1$, thanks to the fact $\B = \cup_{p > 0} \B_p$. For convenience, we consider instead
\begin{align}\label{eq:defBpbar} \overline\B_p :=\Ll\{  f: \N_+ \rightarrow \R, \, \sup_{n\geq 1}|f(n)|/n^p<1 \Rr\}, \quad p\geq 1. \end{align}
Indeed, $\overline \B_p\subset \B_p$ and for any $f\in\B_p$ there exists $g\in \overline \B_p$ and $c\in\R$ such that $f=cg.$

We define the cutoff operator at some level $K \in \N_+$
\begin{align}\label{eq:cutoff}
f_{\leq K}(n) := f(n)\Ind{n \leq K}, \qquad f_{>K}(n) := f(n)\Ind{n > K}. 
\end{align}
First, using \eqref{eq:psXn},  we obtain 
\begin{align*}
\sup_{f \in \overline \B_p} \vert e^{-\lambda t} \bracket{X_t, f_{\leq K}} - W \bracket{\pi, f_{\leq K}} \vert
& \leq K^p \sum_{n=1}^K   \vert e^{-\lambda t} X_t(n) - W \pi(n)\vert \xrightarrow{t \to \infty} 0,\quad a.s..
\end{align*}
Second, recall the definition of $[x^p]$ and $[x]$ introduced above \eqref{eq:defB}. Then 
\begin{align*}
\sup_{f \in \overline \B_p} \vert e^{-\lambda t} \bracket{X_t, f_{>K}} \vert \leq  e^{-\lambda t} \bracket{X_t, \x{p}_{>K}}, \qquad  \sup_{f \in \overline \B_p} \vert W \bracket{\pi, f_{>K}} \vert \leq W \bracket{\pi, \x{p}_{>K}}.
\end{align*}
Combining these estimates and 
\begin{multline*}
\vert e^{-\lambda t} \bracket{X_t, f} - W \bracket{\pi, f} \vert \\ 
 \leq  \vert e^{-\lambda t} \bracket{X_t, f_{>K}} \vert +\vert e^{-\lambda t} \bracket{X_t, f_{\leq K}} - W \bracket{\pi, f_{\leq K}} \vert + \vert W \bracket{\pi, \x{p}_{>K}}  \vert,    
\end{multline*}
it yields
\begin{align*}
\limsup_{t\rightarrow \infty}    
  \sup_{f \in \overline\B_p}  \vert e^{-\lambda t} \bracket{X_t, f} - W \bracket{\pi, f} \vert 
 \leq  \limsup  \limits_{t \to \infty}  e^{-\lambda t} \bracket{X_t, \x{p}_{>K}}+W \bracket{\pi, f_{>K}},
\end{align*}
for any $K\geq 1$.
To show that the right hand side in the above display goes to $0$ as $K$ goes to infinity, it suffices to prove:
\begin{align}\label{eq:TailV}
\lim_{K \to \infty}\limsup  \limits_{t \to \infty}  e^{-\lambda t} \bracket{X_t, \x{p}_{>K}} = 0.
\end{align}
Since \eqref{nconv} ensures
\begin{align}\label{5.45+}
\lim_{k \to \infty}  e^{-\lambda k \Delta} \bracket{X_{k \Delta}, \x{p}_{>K}} = W \bracket{\pi, \x{p}_{>K}}, \end{align}
we just need to control what happens on the time intervals  $[k \Delta, (k+1)\Delta)$.

For that purpose, we use a coupling argument. On every interval $[k\Delta, (k+1)\Delta)$, we consider a size process $\widetilde{X}_t$
starting at time $k\Delta$
with the same value $\widetilde{X}_{k\Delta}:=X_{k\Delta}$; we let the rates of fragmentation and isolation be zero in $\widetilde{X}_t$, while for any cluster at $k\Delta$, 
the growth events occurring to it on $(k\Delta, (k+1)\Delta)$ are constructed by the common exponential clocks in $\widetilde{X}_t$ and $X_t$, until it gets isolated or fragmented in the latter. Notice that the isolation events make negative contribution to $\LL([x^p]), p \geq 1$, so are the fragmentation events because  ${(a+b)^p \geq a^p + b^p}$ for all $a,b>0, p \geq 1$. Therefore, we obtain
\begin{align*}
&\sup_{t \in [k\Delta, (k+1)\Delta)} \bracket{X_{t}, \x{p}_{>K}} \leq \sup_{t \in [k\Delta, (k+1)\Delta)} \bracket{\widetilde{X}_{t}, \x{p}_{>K}}. 
\end{align*}
The term on the right hand side is monotone in $t$ and we get
\begin{align*}\label{eq:BdCoupling}
&\sup_{t \in [k\Delta, (k+1)\Delta)} \bracket{X_{t}, \x{p}_{>K}} \leq  \bracket{\widetilde{X}_{(k+1)\Delta-}, \x{p}_{>K}}. 
\end{align*}
As a result, setting
$$B_k=B_{\Delta, n,k}^K:=\left\{ \bracket{\widetilde X_{(k+1)\Delta-}, \x{p}_{>K}}>2\bracket{X_{k\Delta}, \x{p}_{>K}}\right\},$$
it suffices to prove
\begin{equation}\label{iobk}
\P[\{\text{i.o.\ }B_k\} \cap \{W > 0\}]=0,
\end{equation}
to get that $\limsup  \limits_{t \to \infty}  e^{-\lambda t} \bracket{X_t, \x{p}_{>K}}\leq 2
\lim_{k \to \infty}  e^{-\lambda k \Delta} \bracket{X_{k \Delta}, \x{p}_{>K}}$ a.s.. Together with \eqref{5.45+}, we can conclude that \eqref{eq:TailV} holds. 

To this purpose, we use  a truncation technique.
Define 
$$C_k=C^K_{\Delta, n,k,\epsilon}:=\Ll\{e^{-\lambda k\Delta}\bracket{X_{k\Delta}, \x{p}_{>K}} \geq \epsilon\Rr\} \cap  \Ll\{e^{-\lambda k\Delta}\bracket{X_{k\Delta}, \x{2p}}\leq 1/\epsilon \Rr\},$$
for   $\epsilon>0$. We do a split 
\begin{align}\label{ckdev}\P[\{\text{i.o.\ }B_k\} \cap \{W > 0\}] \leq \P[\{\text{i.o.\ }B_k \cap C_k\} ] + \P[\{\text{i.o.\ }B_k \cap (C_k)^c\} \cap \{W > 0\}].\end{align}
The second term on the right hand side has the following upper bound thanks to  \eqref{nconv} and dominated convergence theorem
$$ \P\Big[W\bracket{\pi,\x{p}_{>K}}\in (0,2\epsilon)\Big] + \P\Big[ W\bracket{\pi,\x{2p}}>1/(2\epsilon)\Big],$$
which converges to $0$ as $\epsilon\to 0.$

We now deal with the first term on the right hand side in \eqref{ckdev}. Defining $$
Z_{k,p,K}:= \bracket{\widetilde{X}_{(k+1)\Delta-}, \x{p}_{>K}} - \bracket{\widetilde{X}_{k\Delta}, \x{p}_{>K}},$$ 
and using Markov inequality,
\begin{equation}\label{eq:bcIntegralBound}
\begin{split}
\P[B_k \,\, \vert \,\,  \mcl F_{k\Delta}] &= \P\Ll[\bracket{\widetilde X_{(k+1)\Delta-}, \x{p}_{>K}} - \bracket{\widetilde X_{k\Delta}, \x{p}_{>K}}> \bracket{\widetilde X_{k\Delta}, \x{p}_{>K}} \,\, \Big\vert \,\,  \mcl F_{k\Delta}\Rr] \\
&\leq  \frac{\var[Z_{k,p, K} \vert \,\,  \mcl F_{k\Delta}]}{\left( \bracket{\widetilde X_{k\Delta}, \x{p}_{>K}}  - \E[Z_{k,p, K} \vert \mcl F_{k\Delta}]\right)^2}.
\end{split}
\end{equation}
We need now to evaluate the conditional expectation and variance.
We will apply the following lemma whose proof will be provided in Appendix~\ref{varianceApp}. 
\begin{lemma}\label{appen} For any $k, K \in \N_+$ and $p \geq 1$, we have
\begin{align}\label{eq:BdExptConditional}
\E[Z_{k,p,K}\,\,\vert\,\, \mcl F_{k\Delta}] & \leq C_{\Delta}\bracket{X_{k\Delta},[x^p]}, 
\end{align}
 where $C_{\Delta}=e^{2^{p-1} p\beta  \Delta} - 1+ (1-e^{-\beta  \Delta K}) K^p$ and
 \begin{align}\label{eq:BdVarConditional}
\var\Ll[Z_{k,p,K} \,\, \vert  \,\, \mcl F_{k\Delta}\Rr] \leq 2\beta \Delta(4^p p+K^{2p+1})\bracket{X_{k\Delta},\x{2p}}.
\end{align}
\end{lemma}

Plugging in the estimates \eqref{eq:BdExptConditional} and \eqref{eq:BdVarConditional} in Appendix to  \eqref{eq:bcIntegralBound}, we  obtain
\begin{align*}
1_{C_k}\P[B_k \,\, \vert \,\,  \mcl F_{k\Delta}] 
&\leq \frac{2\beta \Delta(4^p p+K^{2p+1}) \times \epsilon^{-1} e^{\lambda k \Delta}}{\Big(\epsilon e^{\lambda k \Delta} - \beta \Delta(2^p p+K^{p+1}) \times \epsilon^{-1} e^{\lambda k \Delta}\Big)^2}.
\end{align*}
We pick $\Delta =\epsilon^2$ with $\epsilon$ small enough. Then we obtain $\P[B_k \,\, \vert \,\,  \mcl F_{k\Delta}] \leq C e^{-\lambda k \Delta}$ conditional on $C_k$.
Adding that $\P[B_k \,\, \vert \,\,  C_k] = \E\Big[\P[B_k \,\, \vert \,\,  \mcl F_{k\Delta}]\,\, \Big\vert \,\,  C_k \Big]$, we obtain 
$\sum_{k\geq 1}\P[B_k \,\, \vert \,\,  C_k]<\infty$.
By Borel--Cantelli lemma, $\P[\text{i.o.\ }B_k\cap C_k]=0$ for $\epsilon$ small enough. Therefore we have proved the term on the right hand side in \eqref{ckdev} is equal to $0$. This implies 
\eqref{iobk} and the proof for general test functions is finished. 
\end{proof}

\subsection{Strong law of large numbers for the size process of inactive clusters}\label{subsec:LimitSizeY}

We prove Corollary~\ref{cor:LLNY} in this part. A heuristic argument to obtain the asymptotic limit is to use the generator \eqref{eq:Fxy} and the convergence of $X_t$ in Theorem~\ref{thm:LLN} : 
\begin{align*}
\lim_{s \searrow t}\frac{\E[\langle Y_s, f\rangle - \langle Y_t, f\rangle \vert \mcl F_t]}{s-t}   &= \theta \langle X_t, \x{} f \rangle 
\sim_{t\rightarrow \infty}  \theta e^{\lambda t} W \langle \pi, \x{}\rangle \langle \widetilde{\pi},  f \rangle,
\end{align*}
with $\widetilde{\pi}$ defined in \eqref{eq:PiTilde}. 
In the sequel, we prove the result, with a suitable set of test functions using in particular  martingale analysis.

\begin{proof}[Proof of Corollary~\ref{cor:LLNY}]
Let $f \in \mcl B_p$ for some fixed $p > 0$ throughout the proof.
The proof can be divided into 3 steps. In Step 1, we control the value $\bracket{Y_t, f}$. In Step 2 we prove the result with a specific function $f = h/\x{}$ which gives us a martingale. In Step 3, we generalize this result to general $f \in \mcl B_p$. Let  $C_0$ be a constant, independent of $f$, which may change from
line to line.

\textit{Step 1: $L^2$ estimate.}
We  will use again the estimation \eqref{eq:XL2Bound}. 
We can check that
$F_{g,f}(\mu,\nu)=\bracket{ \nu, f}$ belongs to the domain of extended generator of $(X_t,Y_t)_{t\geq 0}$  (see \eqref{eq:Fxy}) using the same localization argument as in the proof of Lemma \ref{lemma:duhamel}. So we get
\begin{align*}
\frac{\d}{\d t} \E\Ll[\bracket{ Y_t, f}^2\Rr] &= 2\theta\E\Ll[\bracket{ Y_t, f}\bracket{X_t, \x{} f}\Rr] + \theta \E\Ll[ \bracket{ X_t, \x{} f^2} \Rr],\quad \forall  t\geq 0.
\end{align*}
Using Young's inequality with  $\alpha > 0$ (to be chosen later), 
$$\theta\E\Ll[\bracket{ Y_t, f}\bracket{X_t, \x{} f}\Rr]\leq \alpha \E\Ll[\bracket{ Y_t, f}^2\Rr] + \Ll(\frac{\theta^2}{\alpha}\Rr)\E\Ll[\bracket{ X_t, \x{}f}^2\Rr].$$ 
Then we use Gr\"onwall's lemma to get
\begin{align*}
\E\Ll[\bracket{ Y_t, f}^2\Rr]  \leq   \int_0^t e^{\alpha(t-s)}\Ll( \Ll(\frac{\theta^2}{\alpha}\Rr)\E\Ll[\bracket{ X_s, \x{}f}^2\Rr] +  \theta \E\Ll[ \bracket{ X_s, \x{} f^2} \Rr]\Rr)\, \d s.
\end{align*}
Combining  the $L^2$ estimate of $\bracket{ X_s, \x{}f}$ obtained  in \eqref{eq:XL2Bound} and the $L^1$ estimate of $\bracket{ X_t, \x{} f^2}$ in
\eqref{eq:fSpectral}, we get
\begin{multline*}
\E\Ll[\bracket{ Y_t, f}^2 \Rr]  \leq   C_0\int_0^t e^{\alpha(t-s)}  \Ll(\frac{\theta^2}{\alpha}\Rr) \Ll(  \bracket{\pi, \x{}f}^2  e^{2\lambda s} + \parallel f\parallel_p e^{(2\lambda - \sigma)s} \Rr)\, \d s\\
+    \int_0^t e^{\alpha(t-s)}  \theta\Ll(  \bracket{\pi, \x{}f^2}  e^{\lambda s} +C \parallel f \parallel_p e^{(\lambda - w)s}\Rr)\, \d s.
\end{multline*}
We choose $\alpha \in (0, \lambda - \max(\sigma/2,w))$ and conclude that there exists $C'>0$ such that
\begin{align}\label{eq:YL2}
\E[\bracket{ Y_t, f}^2]  \leq  C'  \Ll(\bracket{\pi, \x{}f}^2  e^{2\lambda t} + \bracket{\pi, \x{}f^2}  e^{\lambda t} + \parallel f\parallel_p \Ll(e^{(2\lambda - \sigma)t} + e^{(\lambda - w)t}\Rr)\Rr).
\end{align}

\smallskip

\textit{Step 2: A martingale for $Y$ which tends to $0$.}
We  introduce the  function $$F_{h,h/\x{}}(\mu,\nu)=\bracket{\mu,h}-\Ll(\frac{\lambda}{\theta}\Rr) \bracket{\nu, h/\x{}}.$$
We can again check that it belongs to the domain of the extended generator of $(X_t,Y_t)_{t\geq 0}$. 
It turns out to be a harmonic function: $\mathcal A F_{h,g}=0$.
Then we obtain that
\begin{align}\label{eq:MtY}
H_t := \bracket{X_t, h} - \Ll(\frac{\lambda}{\theta}\Rr) \bracket{Y_t, h/\x{}}.
\end{align}
is a martingale with respect to $(\FF_t)_{t\geq 0}$.
Let us prove that $e^{-\lambda t}H_t$ converges to $0$ as $t\rightarrow \infty$ and thus   Corollary \ref{cor:LLNY} holds for the specific test function $h/\x{}$. 
This vanishing property is due to the fact that the two parts in $H$ compensate each other.
We prove first $L^2$ convergence
using \eqref{eq:Fxy}: 
\begin{align*}
&\frac{\d}{\d t}\E[\vert H_t \vert^2]\\
&= \E\Ll[\sum_{n=1}^{\infty} X_t(n) \beta n\Ll( \vert H_t + h(n+1) - h(n) \vert^2 - \vert H_t \vert^2\Rr) \Rr] \\
& \quad + \E\Ll[\sum_{n=1}^{\infty} X_t(n) \theta n\Ll( \Ll\vert H_t - h(n) - \Ll(\frac{\lambda}{\theta}\Rr) h(n)/n \Rr\vert^2 - \vert H_t \vert^2\Rr) \Rr] \\
& \quad + \E\Ll[\sum_{n=1}^{\infty} \Ll(X_t(n) \gamma n \sum_{j=1}^{n-1} \Ll( \frac{1}{j(j+1)}\Rr) \Ll( \vert H_t + h(j) + h(n-j) - h(n) \vert^2 - \vert H_t \vert^2\Rr)\Rr) \Rr].
\end{align*}
We develop this equation and recognize  the generator $\L$ defined in \eqref{eq:Generator}.  As
$\L h = \lambda h$, we get 
\begin{align*}
&\frac{\d}{\d t}\E[\vert H_t \vert^2]\\
&= 
   \E\Ll[\sum_{n=1}^{\infty} X_t(n) \Ll(  \beta n \vert h(n+1) - h(n) \vert^2 + \theta n \Ll\vert  h(n) + \Ll(\frac{\lambda}{\theta}\Rr) h(n)/n \Rr\vert^2 \Rr)\Rr] \\
&\qquad + \E\Ll[\sum_{n=1}^{\infty} X_t(n) \Ll( \gamma n \sum_{j=1}^{n-1} \Ll( \frac{1}{j(j+1)}\Rr) \Ll\vert h(j) + h(n-j) - h(n)\Rr\vert^2 \Rr)\Rr].
\end{align*}
Since $h$ is bounded (see Proposition~\ref{Perroneigen}), we obtain 
\begin{align}\label{eq:HtL2}
\E[\vert H_t \vert^2] \leq C_0 \int_0^t \E\Ll[\bracket{X_s, \x{}}\Rr] \, \d s  \leq C_0 e^{\lambda t}.
\end{align}
This implies 
the  $L^2$ convergence of $e^{-\lambda t} H_t$ to $0$ as $t\rightarrow \infty$.
For the almost sure convergence, we set the step size $\Delta > 0$. As $(H_t)_{t\geq 0}$ is a martingale, for any $\epsilon > 0$, we combine Markov's inequality, Doob's inequality for $H_t$ and the estimate \eqref{eq:HtL2} to obtain that
\begin{align*}
\P\Ll[\sup_{t \in [k\Delta, (k+1)\Delta)} \vert e^{-\lambda t} H_t \vert > \epsilon \Rr] 
&\leq \P\Ll[e^{-\lambda k\Delta} \sup_{t \in [k\Delta, (k+1)\Delta)} \vert H_t \vert > \epsilon \Rr] \\
&\leq \epsilon^{-2} e^{-2 \lambda  k\Delta} \E\Ll[  \Ll( \sup_{t \in [k\Delta, (k+1)\Delta)} \vert H_t \vert  \Rr)^2 \Rr]\\ 
&\leq 4\epsilon^{-2} e^{-2 \lambda k\Delta} \E\Ll[   \vert H_{(k+1)\Delta} \vert^2 \Rr]\\
&\leq C_0\epsilon^{-2} e^{-\lambda  k\Delta}.
\end{align*}
By  Borel--Cantelli lemma, 
we obtain the a.s.\
convergence of $e^{-\lambda t} H_t$
to $0$.

\smallskip

\textit{Step 3: Convergence for general test functions.} Now we need to obtain the result for a general test function $f \in \mcl B_p$. 
The idea is similar: we define 
\begin{align}\label{eq:MtYf}
H^f_t := \bracket{X_t, f} - \Ll(\frac{\lambda}{\theta}\Rr) \bracket{Y_t, f/\x{}}=\bracket{\pi, f} H_t+ A_t+B_t,
\end{align}
where $A_t=\bracket{X_t, f - \bracket{\pi, f}h}$ and $B_t=\Ll(\frac{\lambda}{\theta}\Rr)\bracket{Y_t, (f - \bracket{\pi, f}h)/\x{}}$.
We apply $L^2$ estimates from  
\eqref{eq:XL2Bound} and \eqref{eq:YL2} to ensure that
\begin{align*}
e^{-2\lambda t}\E\Ll[A_t^2 + B_t^2\Rr] \leq C_0\left(1+\sup_{n\geq 1}\frac{|f(n)|}{n^p}\right)^2 \Ll(e^{-\lambda t} + e^{-\sigma t} + e^{-(\lambda + \sigma/2)t}\Rr).
\end{align*}
Here $C_0$ may depend on $p$ but not on (the specific choice of) $f.$ By Borel--Cantelli lemma, this implies  the convergence of $e^{-\lambda t}H^f_{t}$ along a subsequence $\{k\Delta\}_{k \geq 1}$ with $\Delta > 0$
\begin{align*}
e^{-\lambda k\Delta} H^f_{k\Delta} \xrightarrow{k \to \infty} 0, \qquad \text{ in } L^2 \text{ and almost surely}.
\end{align*}
We get
\begin{align*}
\lim_{k \to \infty}e^{-\lambda k\Delta}\bracket{Y_{k\Delta}, f/\x{}} = \lim_{k \to \infty}\Ll(\frac{\theta}{\lambda}\Rr)e^{-\lambda k\Delta} \bracket{X_{k\Delta}, f} = \Ll(\frac{\theta}{\lambda}\Rr) \bracket{\pi, f}W, \quad \text{ in } L^2 \text{ and a.s.}.
\end{align*}
Finally, to obtain the convergence along $t \in \R_+$, we decompose $f$ as the difference of two positive functions $f = f^+ - f^-$ and use  that $Y_t$ is increasing in $t$:  \begin{align*}
\forall t \in [k\Delta, (k+1)\Delta), \qquad e^{-\lambda (k+1)\Delta}\bracket{Y_{k\Delta}, f^+} \leq e^{-\lambda t} \bracket{Y_t, f^+} \leq e^{-\lambda k\Delta}\bracket{Y_{(k+1)\Delta}, f^+}.
\end{align*} 
We obtain then 
\begin{multline*}
 e^{-\lambda \Delta}\Ll(\frac{\theta}{\lambda}\Rr) \bracket{\pi, \x{}f^+}W \leq \lim_{k \to \infty} e^{-\lambda (k+1)\Delta}\bracket{Y_{k\Delta}, f^+} \leq \liminf_{t \to \infty} e^{-\lambda t} \bracket{Y_t, f^+} \\
\leq \limsup_{t \to \infty} e^{-\lambda t} \bracket{Y_t, f^+}  \leq \lim_{k \to \infty} e^{-\lambda k\Delta}\bracket{Y_{(k+1)\Delta}, f^+} = e^{\lambda \Delta} \Ll(\frac{\theta}{\lambda}\Rr) \bracket{\pi, \x{}f^+}W.
\end{multline*}
We take $\Delta \searrow 0$ to get 
\begin{align*}
\lim_{t \to \infty} e^{-\lambda t} \bracket{Y_t, f^+} = \Ll(\frac{\theta}{\lambda}\Rr) \bracket{\pi, \x{}f^+}W,\quad a.s.
\end{align*}
A similar argument also works for $f^-$. We combine these two terms and use $\widetilde{\pi}$ to prove the almost sure convergence in Corollary~\ref{cor:LLNY}. The $L^2$ convergence can be done similarly and we skip the details. 
\end{proof}

\subsection{Limit on recursive trees}\label{subsec:LimitTree}
In this part, we prove the convergence of the empirical measure on clusters.

\begin{proof}[Proof of Theorem~\ref{thm:LLNRRT}]
We will prove only the convergence on active clusters $\XX_t$, as the proof for $\YY_t$ follows in the same manner. The main idea is similar to the size process $(X_t, Y_t)_{t \geq 0}$, which involves one type convergence and the cut-off argument. Without loss of generality, we suppose that for any $\tt \in \TT$, $\vert f(\tt) \vert\leq \vert \tt \vert^p$ for some $p > 0$.

\textit{Step 1: Cut-off argument.}
We do the following decomposition on the event $\XX_t$ being not empty.
\begin{equation}\label{eq:LimitTreeDecom}
\begin{split}
&\frac{1}{\vert \XX_t \vert}\sum_{\CC \in \XX_t} f(\CC) - \E[f(T_\pi)]=\underbrace{\sum_{n=1}^{\infty} A_t(n)}_{\mathbf{I}}+\underbrace{\sum_{n=1}^{\infty} B_t(n)}_{\mathbf{II}},\\
\end{split}
\end{equation}
where, if we write $g(n) = \E[f(T_n)]$,
$$A_t(n)=\frac{X_t(n)}{\langle X_t, 1\rangle} \Ll( \frac{1}{X_t(n)} \sum_{\CC \in \XX_t, \vert \CC \vert = n} f(\CC) -g(n)\Rr), \quad B_t(n)=\Ll(\frac{X_t(n)}{\langle X_t, 1\rangle} - \pi(n)\Rr)g(n).$$
Theorem~\ref{thm:LLN} implies the a.s.\ convergence  of the second term $\mathbf{II}:$ 
\begin{align*}
\lim_{t \to \infty} \sum_{n=1}^{\infty} B_t(n) = \lim_{t \to \infty} \Ll(\frac{\langle X_t, g\rangle}{\langle X_t, 1\rangle} - \langle \pi, g\rangle\Rr) = 0,\text{ on }\{\tau=\infty\}.
\end{align*}
For the first term $\mathbf{I}$, we use a cut-off argument and  $\vert f( \CC )\vert \leq \vert \CC \vert^p$:
\begin{equation}\label{eq:LimitTreeDecom2}
\Ll\vert \sum_{n=1}^{\infty} A_t(n) \Rr\vert 
 \leq \Ll\vert \sum_{n=1}^{K} A_t(n)\Rr\vert +  2\Ll\vert \sum_{n= K+1}^{\infty} \frac{X_t(n) n^p}{\langle X_t, 1\rangle} \Rr\vert.
\end{equation}
For the first term on the right hand side in \eqref{eq:LimitTreeDecom2}, we admit the following convergence (to be proved in Step 2)
\begin{align}\label{eq:ClusterOneType}
\forall n \in \N_+, \quad \lim_{t \to \infty} \frac{1}{X_t(n)} \sum_{\CC \in \XX_t, \vert \CC \vert = n} f(\CC)  = \E[f(T_n)], \qquad \text{ almost surely on } \{\tau = \infty\},
\end{align}
which can also be seen as a generalized law of large numbers. By Theorem \ref{thm:LLN}, for any $n \in \N_+$, $\frac{X_t(n)}{\langle X_t, 1\rangle}$ converges a.s.\ as $t\to\infty$. Then the above display ensures that $
{\lim_{t \to \infty} \Ll\vert \sum_{n=1}^{K} A_t(n) \Rr\vert  = 0}$, almost surely on $\{\tau=\infty\}$. The second term on the right hand side in \eqref{eq:LimitTreeDecom2} also converges a.s.\ due to Theorem~\ref{thm:LLN}: 
\begin{align*}
\lim_{t \to \infty} \Ll\vert \sum_{n= K+1}^{\infty} \frac{X_t(n) n^p}{\langle X_t, 1\rangle} \Rr\vert = \langle \pi, \x{p}_{>K}\rangle, \qquad  \text{ almost surely on }\{\tau=\infty\},
\end{align*}
where $\x{p}_{>K}(n) = n^p \Ind{n > K}$. We put these results back in to \eqref{eq:LimitTreeDecom} and obtain that
\begin{align*}
\lim_{t \to \infty} \Ll\vert \frac{1}{\vert \XX_t \vert}\sum_{\CC \in \XX_t} f(\CC) - \E[f(T_\pi)] \Rr\vert \leq 2\langle \pi, \x{p}_{>K}\rangle, \qquad  \text{ almost surely on }\{\tau=\infty\}.
\end{align*} 
Then we let $K \to \infty$ and prove Theorem~\ref{thm:LLNRRT}. 

\smallskip

\textit{Step 2: Convergence for one type.} It remains to prove \eqref{eq:ClusterOneType}. 
We follow the same spirit as in Step 2 in the proof of Theorem~\ref{thm:LLN}, with some minor technical differences. We recall $\TT_n$ the space of equivalence classes of RRT of size $n$, and denote by 
\begin{align*}
\forall \tt \in \TT_n, \qquad X_t(\tt) := \sum_{\CC \in \XX_t} \Ind{\CC \sim \tt},
\end{align*}
the number of active clusters of type $\tt$. Because the space $\TT_n$ is finite ($\vert \TT_n \vert = (n-1)!$), it suffices to prove that 
\begin{equation}\label{eq:LLNClusterOneType}
\forall n \in \N_+, \forall \tt \in \TT_n, \qquad \lim_{t \to \infty} e^{-\lambda t}X_t(\tt) = W \frac{\pi(n)}{(n-1)!},  \quad a.s.,
\end{equation} 
and this can be dealt with using the same technique in the proof of Theorem~\ref{thm:LLN}: we only need to prove the following 
\begin{align}\label{eq:LLNClusterOneTypeInf}
\forall n \in \N_+, \forall \tt \in \TT_n, \qquad \liminf \limits_{t \to \infty} e^{-\lambda t}X_t(\tt) \geq W \frac{\pi(n)}{(n-1)!}.
\end{align}
This is similar to proving \eqref{eq:limsupTrick}. Let $(t_k)_{k \in \N_+}$ be any (random) time sequence such that ${\limsup \limits_{t \to \infty} e^{-\lambda t}X_t(\tt) = \lim_{k \to \infty} e^{-\lambda t_k}X_{t_k}(\tt)}$. Then we have
 \begin{equation*}
     \begin{split}
          \limsup  \limits_{t \to \infty} e^{-\lambda t} X_t(\tt) &=\lim_{k\rightarrow\infty} \left(\sum_{\tt' \in \TT_n} e^{-\lambda t_k} X_{t_k}(\tt') -\sum_{\tt' \in \TT_n, \tt' \neq \tt} e^{-\lambda t_k} X_{t_k}(\tt')\right)\\
 & \leq W\pi(n)-\sum_{\tt' \in \TT_n, \tt' \neq \tt} \liminf  \limits_{k \to \infty} e^{-\lambda t_k} X_{t_k}(\tt')\\
 & \leq W\pi(n) -\sum_{\tt' \in \TT_n, \tt' \neq \tt} W \frac{\pi(n)}{(n-1)!} =W \frac{\pi(n)}{(n-1)!}.
     \end{split}
 \end{equation*}
Here from the first line to the second line, we use Fatou's lemma, and from the second line to the third line we use \eqref{eq:LLNClusterOneTypeInf} . This equation controls the upper bound of $\limsup  \limits_{t \to \infty} e^{-\lambda t} X_t(\tt)$. Using also \eqref{eq:LLNClusterOneTypeInf}, we conclude that \eqref{eq:LLNClusterOneType} is true.

Finally, we prove \eqref{eq:LLNClusterOneTypeInf}, which requires convergence along discrete time sequences and the control of fluctuation. We calculate the $L^2$ moment as follows
\begin{align*}
&\E\Ll[ \Ll(e^{-\lambda t} \sum_{\CC \in \XX_t, \vert \CC \vert = n}\Ll(\Ind{\CC \sim \tt} - \frac{1}{(n-1)!}\Rr)\Rr)^2 \Rr] \\
&= e^{-2\lambda t} \E\Ll[\sum_{\CC \in \XX_t, \vert \CC \vert = n} \E\Ll[\Ll(\Ind{\CC \sim \tt} - \frac{1}{(n-1)!}\Rr)^2\,\,\Big\vert \,\, \mcl F_t\Rr]\Rr]\\
&\leq e^{-2\lambda t} \E\Ll[X_t(n)\Rr]=O(e^{-\lambda t}) \xrightarrow{t \to \infty} 0.
\end{align*}
In the first line, we use the fact that the clusters are i.i.d.\ RRTs of size $n$, see Proposition~\ref{prop:ClusterRRT}. In the second line, we use \eqref{eq:XL2Bound}. Notice that the $L^2$ moment 
decreases at least exponentially, we can thus take a discrete time sequence $\{k\Delta\}_{k \geq 1}$ with $\Delta > 0$ and use Borel--Cantelli lemma to obtain that for any $\Delta > 0$
\begin{align*}
\forall n \in \N_+, \forall \tt \in \TT_n, \qquad \lim_{k \to \infty} e^{-\lambda k\Delta}X_{k\Delta}(\tt) = W \frac{\pi(n)}{(n-1)!}. 
\end{align*}
Let us now control fluctuations and  let $N_{k,\Delta}(\tt)$ be the number of active clusters of type $\tt$ at time $k\Delta$, on which occurs some 
event during $(k\Delta, (k+1)\Delta)$. Then like \eqref{eq:psFluctuation}, we have
\begin{align*}
\forall t\in [k\Delta, (k+1)\Delta), \qquad    X_t(\tt) \geq X_{k\Delta}(\tt) - N_{ k,\Delta}(\tt),
\end{align*}
and it suffices to prove $\lim_{k \to \infty}e^{-\lambda k\Delta}N_{ k,\Delta}(\tt) = 0$ to conclude \eqref{eq:LLNClusterOneTypeInf}. We skip the details as it follows exactly the same lines starting from \eqref{eq:psFluctuation} in Step 2 in the proof of Theorem~\ref{thm:LLN}. 
\end{proof}

\section{Characterization of phases and regularity}\label{sec:Chara}
Recall the three parameters $\beta,\theta,\gamma$ introduced when describing the GFI process in Section~\ref{sec:Introduction}. The Malthusian exponent $\lambda=\lambda(\beta,\theta,\gamma)$ allows us to define the phases in terms of the sign of $\lambda$ (i.e. $\lambda>0, =0, <0$). So far, we have not shown how many phases exist and how they are related. The main objective of this section is to justify the phase diagram shown in Figure~\ref{fig:Phases} and prove Theorem~\ref{thm:Regularity}. The proof strategy is organized as follows: in Section~\ref{subsec:Regularity1}, we prove some regularity properties about the mapping $(\beta, \theta, \gamma) \mapsto \lambda(\beta, \theta, \gamma)$. In Section~\ref{subsec:ExitPhases},  we justify Figure~\ref{fig:Phases}. In  Section~\ref{subsec:Regularity2}, we give a brief proof for the change of monotonicity of $\beta\mapsto \lambda(\beta.\theta,\gamma)$ and finally prove Theorem \ref{thm:Regularity}.

\subsection{Regularity of Perron's root}\label{subsec:Regularity1}
The occurrence of different phases relies on the parameters $(\beta,\theta,\gamma).$ We will study how the Perron's root $\lambda(\beta, \theta, \gamma)$ depends on the three parameters in the following proposition. Recall that $(\beta, \theta, \gamma)\in \R_+^3.$
\begin{proposition}[Regularity of Perron's root]\label{prop:Regularity}
For the Perron's root $\lambda$ associated to $\L$ in \eqref{eq:Generator}, the mapping $(\beta, \theta, \gamma) \mapsto \lambda(\beta, \theta, \gamma)$ satisfies the following properties: 
\begin{enumerate}
\item Monotonicity: $\lambda(\beta, \theta, \gamma)$ is increasing on $\gamma$ and decreasing on $\theta$.
\item Regime: 
\begin{itemize}
\item if $\theta \geq \beta$, or $\theta \in (0, \beta)$ and $0 < \gamma < \theta\Ll(2^{1 - \frac{\theta}{\beta}} - 1\Rr)^{-1}$, then $\lambda  < 0$;
\item if $\theta \in (0, \beta)$ and $\gamma > \max \Ll\{ 2\theta\Ll(\frac{1 + \frac{\theta}{\beta}}{1 - \frac{\theta}{\beta}}\Rr), \frac{3}{2}\theta \Ll(1 + \frac{\theta}{\beta}\Rr)  \Rr\}$, then $\lambda > 0$.
\end{itemize}
\item Homogeneity: for any $\rho > 0$, we have $\lambda(\rho\beta, \rho\theta, \rho\gamma) = \rho \lambda(\beta, \theta, \gamma)$.
\item Continuity: $\lambda(\beta, \theta, \gamma)$ is continuous with respect to the three parameters.
\end{enumerate}
\end{proposition}

\begin{proof}
\textit{Monotonicity.} If the clusters split faster, then the number of active clusters will increase. If the clusters become inactive faster, then the number of active clusters will decrease. This intuition is easy to prove by providing a coupling argument for each of the parameters. Then we have the monotonicity property.

\textit{Regime.} 
To find a regime where $\lambda < 0$, we make a test with $f=[x^p]$ with ${p \in (0, 1)}$. Jensen's inequality implies that 
\begin{align*}
\frac{x^p + y^p}{2} \leq \Ll(\frac{x + y}{2}\Rr)^p,\quad \forall x,y\geq 0.
\end{align*}
This entails that
\begin{align*}
(j)^p + (n-j)^p - n^p  \leq \Ll(2^{1-p} - 1\Rr) n^p.
\end{align*}
We apply the above inequality into $\L f$ and get that
\begin{equation}\label{eq:TestPoly}
\begin{split}
\L f(n) &\leq \beta n ((n+1)^p - n^p) - \theta n^{p+1} + \gamma \Ll(2^{1-p} - 1\Rr)(n-1) n^p \\ 
& \leq \Ll(p\beta - \Ll(2^{1-p} - 1\Rr) \gamma\Rr) f(n) + \Ll(\Ll(2^{1-p} - 1\Rr) \gamma - \theta\Rr)nf(n)\\
&=\Ll(p\beta-\theta + \gamma(2^{1-p} - 1-\theta)(n-1) \Rr) f(n).
\end{split} 
\end{equation} 
Now, we distinguish two situations.
If  $\theta \in (0, \beta)$,  the condition $\gamma < \theta/\Ll(2^{1 - \frac{\theta}{\beta}} - 1\Rr)$ allows us to find some $p<\theta/\beta$ 
such that
$\L f(n) \leq (p \beta - \theta) f(n)$ for all $n$. Consequently 
$ M_t f \leq \exp((p \beta - \theta)t)f$. Then  Proposition \ref{Perroneigen} entails $\lambda < 0$.
If  $\theta \geq \beta$, it suffices to choose some $p$ close to $1$ such that $\gamma(2^{1-p} - 1)- \theta < 0$  and we get also $\lambda < 0$.\\

To find a regime where $\lambda > 0$, we consider the test function 
\begin{align*}
f(n) = \Ind{n=1} + \kappa\Ind{n \geq 2},
\end{align*}
with constant $\kappa \in (0,2)$ to be fixed. Recalling \eqref{eq:Generator},  it is clear that 
\begin{align*}
n = 1, \quad \L f(1) & = \beta(\kappa - 1) - \theta,\\
n = 2, \quad \L f(2) &= -2 \theta\kappa + \gamma(2 - \kappa),\\
n \geq 3, \quad \L f(n) & \geq - n \theta  \kappa + \gamma(n-1).
\end{align*}
For  $$\gamma > \max \Ll\{ 2\theta\Ll(\frac{1 + \frac{\theta}{\beta}}{1 - \frac{\theta}{\beta}}\Rr), \frac{3}{2}\theta \Ll(1 + \frac{\theta}{\beta}\Rr)  \Rr\},$$
we get $\L f(n) > 0$ for any $n\geq 1$. Taking  $\kappa = \frac{1}{2}\Ll(\Ll(1 + \frac{\theta}{\beta}\Rr) + \Ll(\frac{2\gamma}{3\theta}\Rr)\Rr)$, we obtain ${\L f(n) \geq C f(n)}$ for all $n$, with $C>0$. It implies $\lambda >0$.

\textit{Homogeneity.} This is due to the fact that the generator $\L$ is linear with respect to the three parameters. Let $\widetilde{\L}$ and $(\widetilde{M}_t)_{t\geq 0}$ be respectively the generator and semigroup associated to $(\rho\beta, \rho\theta, \rho\gamma)$, while we denote by $\lambda, h, \pi$ the eigenelements associated to $\L$. Then by the definition of $(X_t)_{t\geq 0}$, it is clear that $\widetilde{\L} = \rho \L$, so $\widetilde{M}_t = M_{\rho t}$. This also implies 
\begin{align*}
\pi \widetilde{\L} = \rho \pi \L = \rho \lambda \pi , \qquad \widetilde{\L} h = \rho \L h = \rho \lambda h,
\end{align*} 
so $\rho\lambda, h, \pi$ are  eigenelements associated to $\widetilde{\L}$. Then by Proposition \ref{Perroneigen}, $\rho \lambda$ is Perron's root for $\widetilde{\L} = \rho \L$.


\textit{Continuity.} We will use three steps to prove it.

\textit{Step 1: $\lambda(\beta,\theta,\gamma)$ is locally bounded.} For any $\epsilon>0$ such that $$\frac{1}{1+\epsilon}<\frac{\beta_0}{\beta}<1+\epsilon,\quad\frac{1}{1+\epsilon}<\frac{\theta_0}{\theta}<1+\epsilon,\quad\frac{1}{1+\epsilon}<\frac{\gamma_0}{\gamma}<1+\epsilon,$$
we have \begin{align*}
&|\lambda(\beta_0,\theta_0,\gamma_0)|\leq &\\
&\quad\quad\quad\quad (1+\epsilon)\max\left\{\big|\lambda(\beta, (1+\epsilon)^{2}\theta,(1+\epsilon)^{-2}\gamma)\big|,\big|\lambda(\beta, (1+\epsilon)^{-2}\theta,(1+\epsilon)^2\gamma)\big| \right\}&.
\end{align*}
The above result is due to the homogeneity of $\lambda$ on all parameters and the monotonicity of $\lambda$ on $\theta$ and $\gamma$. Then the local boundedness of $\lambda$ is proved. 

\smallskip

\textit{Step 2: $\bracket{\pi(\beta,\theta,\gamma),[x]}$ is locally bounded.} We prove it by contradiction. Assume that $(\beta_n,\theta_n,\gamma_n)$ converges to $(\beta_0,\theta_0,\gamma_0)$. Let $(\beta,\theta,\gamma)=(\beta_0,\theta_0/2,2\gamma_0)$. Let $M_t,\pi,h,\lambda$ be associated to $(\beta,\theta,\gamma)$ and $\widetilde M_t,\widetilde \pi,\widetilde h,\widetilde \lambda$ to $(\beta_n,\theta_n,\gamma_n)$. We assume that $\bracket{\widetilde \pi,[x]}$ converges to infinity as $n\to\infty$. 
Due to homogeneity of $\lambda$ on its parameters, we can assume $\beta_n=\beta_0$ without loss of generality. 

By the assumption above, for any $K>0,$ let $n_K>0$ such that for any $n\geq n_K$, we have $\bracket{\widetilde \pi,[x]}>K$ and $\theta_n>3\theta_0/4$ and $\gamma_n<7\gamma_0/4$. By Proposition~\ref{Perroneigen}, there exist $0<c<C$ such that $c\leq h(n)\leq C$ for any $n\geq 1$. Then using (iii) in Lemma~\ref{lemma:duhamel},  we have 
\begin{equation}\label{eq:mupbound0}
\begin{split}
\frac{\d}{\d t} \widetilde{M}_t h &= \widetilde{M}_t (\widetilde{\L} h) \\
&\leq \widetilde{M}_t (L h) - c(\theta_n-\theta_0/2) \widetilde{M}_t (\x{})+2c(\gamma_n-2\gamma_0)\widetilde M_t([x])\\
&= \lambda \widetilde{M}_th - c(\theta_n-\theta_0/2) \widetilde{M}_t (\x{})+2c(\gamma_n-2\gamma_0)\widetilde M_t([x]).
\end{split}
\end{equation} 
Using again Proposition~\ref{Perroneigen}, as $t\to\infty$,
\begin{equation}\label{eq:mtopi}\widetilde M_t h(n)\sim \widetilde h\bracket{\widetilde \pi,h}e^{\widetilde \lambda t},\quad \widetilde M_t ([x])(n)\sim \widetilde h\bracket{\widetilde \pi,[x]}e^{\widetilde \lambda t}.\end{equation}
Then for any fixed $n\geq n_K$, and $t$ large enough, we have 
$$\widetilde M_t([x])(n)\geq \frac{K}{2C}\widetilde M_t h(n).$$
Then we obtain that 
\begin{equation}\label{eq:mupbound}
\begin{split}
\frac{\d}{\d t} \widetilde{M}_t h(n) &\leq  \lambda\widetilde{M}_th(n) - c(\theta_n-\theta_0/2) \widetilde{M}_t (\x{})(n)+2c(\gamma_n-2\gamma_0)\widetilde M_t([x])(n)\\
&\leq  \left(\lambda-cK\frac{\theta_0+2\gamma_0}{8C}\right)\widetilde{M}_th(n).
\end{split}
\end{equation} 
The above display implies that 
$$\widetilde \lambda\leq \lambda-cK\frac{\theta_0+2\gamma_0}{8C}.$$
As $K$ can be arbitrarily large, we obtain that $\widetilde \lambda \to -\infty$ as $n\to\infty$, which is contradictory to the fact that $\lambda$ is locally bounded as proved in Step 1. Therefore we conclude that $\bracket{\pi(\beta,
\theta,\gamma),[x]}$ is locally bounded.  

\smallskip

\textit{Step 3: $\lambda$ is continuous with respect to the three parameters.} The approach is very similar to that in Step 2. Assume that $(\beta_n,\theta_n,\gamma_n)$ converges to $(\beta,\theta,\gamma)$. Let $M_t,\pi,h,\lambda$ be associated to $(\beta,\theta,\gamma)$ and $\widetilde M_t,\widetilde \pi,\widetilde h,\widetilde \lambda$ to $(\beta_n,\theta_n,\gamma_n)$. We use again that there exist $0<c<C$ such that $c\leq h(n)\leq C$ for any $n\geq 1$. Due to homogeneity, we can assume that $\gamma_n=\gamma$. We prove next $\lim_{n\to\infty}\widetilde \lambda=\lambda.$

By the assumption above, for any $K>0,$ let $n_K>0$ such that for any $n\geq n_K$, we have 
\begin{align*}
\vert \beta_n-\beta \vert \leq \beta/K, \quad  \vert \theta_n-\theta \vert \leq \theta/K.    
\end{align*}
Since $\bracket{\widetilde{\pi},[x]}$ is locally bounded, using \eqref{eq:mtopi} and $c\leq h\leq C$, there exists $l=l(\beta,\theta,\gamma)>1$ such that for fixed $n\geq n_K$ and $t$ large enough, we have 
$$\widetilde M_t h(n)/l\leq \widetilde M_t([x])(n)\leq  l\widetilde M_t h(n).$$
Then similarly to \eqref{eq:mupbound}, we have 
\begin{equation}\label{}
\begin{split}
\frac{\d}{\d t} \widetilde{M}_t(h) (n)&\leq \left(\lambda+Cl\frac{\beta+\theta}{K}\right)\widetilde{M}_t (h)(n).
\end{split}
\end{equation} 
Since $K$ can be arbitrarily large, we obtain that $\limsup_{n\to\infty}\widetilde \lambda\leq \lambda.$ The other direction can be proved similarly. Then the proof for continuity is finished.  
\end{proof}

\subsection{Existence of phases}\label{subsec:ExitPhases}

The main objective of this part is to prove the following result, which describes the phases and justifies  Figure~\ref{fig:Phases}. 

\begin{proposition}[Classification of phases]\label{prop:Curve}
The regime of the three phases of the GFI process depends on the parameters $(\beta, \theta, \gamma)$ in the following way.
\begin{enumerate}
\item Critical phase: the regime is a critical surface defined as
\begin{align*}
\{ \lambda(\beta, \theta, \gamma) = 0\} = \Ll\{(\beta, \theta, \gamma) \in \R^3_+ \,\vert 0 < \theta < \beta, \gamma = \gamma_c(\beta,\theta)\Rr\}.
\end{align*}
Here $\gamma_c(\beta,\theta)$ is a function, such that for fixed $\beta$, the mapping $\theta \mapsto \gamma_c(\beta, \theta)$ is strictly increasing, continuous, and satisfies
\begin{align*}
\lim_{\theta \searrow 0} \gamma_c(\beta, \theta) = 0, \qquad \lim_{\theta \nearrow \beta} \gamma_c(\beta, \theta) = \infty.
\end{align*}

\item Subcritical phase: the regime stays below the critical surface 
\begin{align*}
\{\lambda(\beta, \theta, \gamma) < 0\} &= \Ll\{(\beta, \theta, \gamma) \in \R^3_+ \,\vert \theta \geq \beta\Rr\}\\
& \qquad \qquad \cup \Ll\{(\beta, \theta, \gamma) \in \R^3_+ \,\vert 0 < \theta < \beta , 0 < \gamma < \gamma_c(\beta, \theta)\Rr\}
\end{align*}
\item Supercritical phase: the regimes stays above the critical surface
\begin{align*}
\{\lambda(\beta, \theta, \gamma) > 0\} = \Ll\{(\beta, \theta, \gamma) \in \R^3_+ \,\vert 0 < \theta < \beta, \gamma > \gamma_c(\beta, \theta)\Rr\}.
\end{align*}
\end{enumerate}
\end{proposition}

The key step in the proof of Proposition~\ref{prop:Curve} relies on some strict increment estimates which are proved in the following two lemmas.

\begin{lemma}\label{lem:StrictTheta}
Perron's root $\lambda$ associated to the generator $\L$ in \eqref{eq:Generator} satisfies the following estimates
\begin{equation}\label{eq:StrictTheta}
\begin{split}
\forall \delta > 0, \quad \lambda(\beta, \theta+\delta, \gamma) \leq \lambda(\beta, \theta, \gamma) - \delta,\\
\forall \delta \in (0, \theta), \quad \lambda(\beta, \theta-\delta, \gamma) \geq \lambda(\beta, \theta, \gamma) + \delta.
\end{split}
\end{equation}
\end{lemma}
\begin{proof}
Let $\widetilde{\L}$ and $(\widetilde{M}_t)_{t\geq 0}$ be the generator and semigroup associated to $(\beta, \theta+\delta, \gamma)$, let $h$ be the eigenvector associated to $\L$ of parameters $(\beta, \theta, \gamma)$. Using a similar development like \eqref{eq:mupbound0} and using also $[x]\geq 1$, we obtain 
\begin{align*}
\frac{\d}{\d t} \widetilde{M}_t h &= \widetilde{M}_t (\L h) - \delta \widetilde{M}_t (\x{} h) = \lambda \widetilde{M}_t h - \delta \widetilde{M}_t (\x{} h) \leq (\lambda-\delta) \widetilde{M}_t h.
\end{align*}
Using Gr\"onwall's inequality, the above display implies that $\widetilde{M}_t h \leq e^{(\lambda-\delta)t} h$, which proves the first statement in \eqref{eq:StrictTheta}. For the second, the proof is similar and we skip it. 
\end{proof}

\begin{lemma}\label{lem:hRegularity}
Let $h$ be the eigenvector of Perron's root associated to the generator $\L$ in \eqref{eq:Generator}. Then in the critical phase $\lambda = 0$, $h$ is subadditive, i.e. 
\begin{align}\label{eq:hRegularity}
\forall n,m \in \N_+, \quad h(m+n) \leq h(m) + h (n).
\end{align}
\end{lemma}
\begin{proof}
The proof relies on the probabilistic representation of $h$. Let $N_t^n$ be the number of clusters at time $t$ with the initial state being a single cluster of size $n$. Then we have
\begin{align*}
\E[N_t^n]=\bracket{\delta_n M_t,  1}.
\end{align*}
Using \eqref{eq:fSpectral}, in the case $\lambda = 0$, 
we have 
\begin{align}\label{eq:hProba1}
\lim_{t \to \infty}\E[N_t^n] = \lim_{t \to \infty}\bracket{\delta_n M_t,  1} = h(n) \bracket{\pi, 1}.
\end{align}
Then we think of a coupling as follows: consider two processes of $(\XX_t)_{t\geq 0}$, one with  initial state being a single RRT of size $(m+n)$, and the other one with the same initial state but has a uniformly selected edge removed which results in two smaller clusters of sizes respectively $m$ and $n$. By Proposition \ref{prop:Split}, these two clusters are independent RRTs with sizes respectively $m,n$.  Assume that the two processes have the same random events in their life times, except that a splitting may happen in the particular edge that exists in the first process and does not exist in the second process.  In this coupling, it is easy to see that the number of clusters at any time $t>0$ in the first process will not be larger than the number in the second process. Since the clusters evolve independently, we have \begin{align}\label{eq:hProba2}
\E[N_t^{n+m}] \leq \E[N_t^{n}]  + \E[N_t^{m}].
\end{align}
Combining \eqref{eq:hProba1} and \eqref{eq:hProba2}, we obtain the desired result \eqref{eq:hRegularity}.
\end{proof}
\begin{remark}\label{rmk:NonTrivial}
In the critical case, since $\L h=0h=0,$ we solve $\L h(1) = 0$ to obtain that $h(2) = \Ll(1 + \frac{\theta}{\beta}\Rr)h(1)$. The criticality also implies $\beta > \theta$. Therefore
\begin{align*}
h(1) + h(1) - h(2) = \Ll(1 - \frac{\theta}{\beta}\Rr) h(1) > 0.
\end{align*} 
So the $``="$ in \eqref{eq:hRegularity} is not always established.
\end{remark}

\begin{proof}[Proof of Proposition~\ref{prop:Curve}]
We divide the proof into 4 steps. Steps 1-3 are about the critical phase, and the Step 4 proves the regime of the other phases.

\textit{Step 1: Unique critical point.} Fix $\beta > 0$ and $0 < \theta < \beta$. Using the regimes discovered for subcritical and critical phases in Proposition~\ref{prop:Regularity} and also the continuity and monotonicity property of $\lambda$ therein, we know that $\{\gamma: \lambda(\beta, \theta, \gamma) = 0\}$ is a non-empty closed interval. We aim to prove that this interval is a single point by contradiction. Suppose that there exist $\gamma>0, \delta>0$ such that ${\lambda(\beta, \theta, \gamma) = \lambda(\beta, \theta, \gamma+\delta) = 0}$. Let $\widetilde{\L}$ and $(\widetilde{M}_t)_{t\geq 0}$ be the generator and semigroup associated to $(\beta, \theta, \gamma+\delta)$, while we denote by $\L$ the generator and $h$ the eigenvector for Perron's root associated to $(\beta, \theta, \gamma)$. We also define the function $G$ 
\begin{align*}
G(n) :=  \sum_{j=1}^{n-1} \frac{n}{j(j+1)}\Ll(h(j) + h(n-j) - h(n)\Rr). 
\end{align*}
Since $h$ is subadditive (Lemma \ref{lem:hRegularity}) and non-negative (Proposition \ref{Perroneigen}), we have 
\begin{align*}
G(n)\leq \sum_{j=1}^{n-1} \frac{n}{j(j+1)} \Ll(j h(1) + (n-j)h(1)\Rr) \leq n^2h(1).  
\end{align*}
Moreover by Remark~\ref{rmk:NonTrivial}, $G$ is a non-zero positive function.  Then by a similar argument like \eqref{eq:mupbound0} and using $\L h = \lambda h = 0$, we have 
\begin{align}\label{eq:CriticalIntegration}
\frac{\d}{\d t} \widetilde{M}_t h =  \widetilde{M}_t (\widetilde{\L} h) =\widetilde{M}_t (\L h+\delta G)= \widetilde{M}_t (\L h) + \delta  \widetilde{M}_t G = \delta  \widetilde{M}_t G.
\end{align}
By Proposition \ref{Perroneigen}, there exist Perron's eigenelements $\widetilde{h}, \widetilde{\pi}$, and also a constant $\widetilde{\omega} > 0$, corresponding to $\widetilde \L$, and the estimate \eqref{eq:fSpectral} holds with the corresponding terms. Then we apply \eqref{eq:fSpectral} and $G(n)\leq n^2 h(1)$ to obtain
\begin{align*}
\widetilde{M}_t G \geq \bracket{\widetilde{\pi}, G}\widetilde{h} - C e^{- \omega t}.
\end{align*} 
Here $C>0$ is a constant that does not depend on $t$. 
We plug in the above display to  \eqref{eq:CriticalIntegration} and obtain
\begin{align*}
\frac{\d}{\d t}\widetilde{M}_t h \geq \delta \Ll(\bracket{\widetilde{\pi}, G} \widetilde{h} - C e^{- \omega t}\Rr),
\end{align*}
If $\lambda(\beta, \theta, \gamma+\delta) = 0$, then by 
\eqref{eq:fSpectral}, $\widetilde{M}_t h$, as a vector, converges to $ \bracket{\widetilde \pi,h} \widetilde h$ as $t\to\infty$ for every element. However since $G$ is a non-zero positive function (which implies $\bracket{\widetilde{\pi}, G} >0$), the above display shows that $\widetilde{M}_t h$ increases at least linearly to infinity as $t\to\infty$ for every element. This is a contradiction, so by the monotonicity of $\lambda$ in $\gamma$, we have $\lambda(\beta, \theta, \gamma+\delta)>0$. We conclude that for fixed $0<\theta<\beta$, there exists only one $\gamma$ which makes $\lambda(\beta,\theta,\gamma)=0.$ We denote it by $\gamma_c=\gamma_c(\beta,\theta)$.  

\smallskip

\textit{Step 2: Strictly increasing and continuous function.} Using Lemma~\ref{lem:StrictTheta} (first inequality), we have 
\begin{align*}
\lambda(\beta, \theta+\delta, \gamma_c(\beta, \theta)) \leq \lambda(\beta, \theta, \gamma_c(\beta, \theta)) - \delta = -\delta.
\end{align*}
Since $\lambda(\beta,\theta,\gamma)$ is increasing with respect to $\gamma$,  we obtain that ${\gamma_c(\beta, \theta + \delta) > \gamma_c(\beta, \theta)}$.

Next we prove the continuity of ${\theta \mapsto \gamma_c(\beta, \theta)}$. 
By the monotonicity of $\lambda$ in $\gamma$, we have  for any $\epsilon > 0$,  $\lambda(\beta, \theta, \gamma_c(\beta, \theta)+\epsilon) > 0$. Then by the continuity of $\lambda$ on $\theta$, there exists some $\delta_0 > 0$ such that for all $\delta \in (0, \delta_0)$, $\lambda(\beta, \theta + \delta, \gamma_c(\beta, \theta)+\epsilon) > 0$. This implies, using the uniqueness of $\gamma_c$ and monotonicity of $\lambda$ on $\gamma$, that for any $\delta \in (0, \delta_0)$
\begin{align*}
\gamma_c(\beta, \theta + \delta) < \gamma_c(\beta, \theta) + \epsilon,
\end{align*}
which proves the right continuity of $\gamma_c(\beta,\theta)$ on $\theta$. The left continuity can be proved exactly in the same way.

\smallskip

\textit{Step 3: Asymptotic behavior.} By the monotonicity property of $\lambda$ on $\gamma$ and regimes for subcriticality and supercriticality given in Proposition~\ref{prop:Regularity}, we have: 
\begin{align}\label{eq:BoundCriticalPhase}
\theta\Ll(2^{1 - \frac{\theta}{\beta}} - 1\Rr)^{-1} \leq \gamma_c(\beta, \theta) \leq  \max \Ll\{ 2\theta\Ll(\frac{1 + \frac{\theta}{\beta}}{1 - \frac{\theta}{\beta}}\Rr), \frac{3}{2}\theta \Ll(1 + \frac{\theta}{\beta}\Rr)  \Rr\}.
\end{align}
Then it suffices to let $\theta \searrow 0$ and $\theta \nearrow \beta$ to prove the asymptotic limits. 

\smallskip

\textit{Step 4: Regime.} Once we have proved that the critical phase is the single point for $\beta, \theta$ fixed in an admissible domain, we apply the monotonicity about $\gamma$ in Proposition~\ref{prop:Regularity} and justify the regime for other phases.
\end{proof}

\subsection{Monotonicity of $h$  and $ \lambda$}\label{subsec:Regularity2}
The dependence of $\lambda$ in function of $\theta, \gamma$ has been discussed in the previous subsections, while the dependence on $\beta$ is more complicated. Indeed, in terms of number of  active clusters,  isolation and fragmentation have opposite effect and increasing the cluster size accelerate both. This problem is  closely related to the monotonicity of the eigenfunction $h$.  The main result in this part is the following one.

\begin{proposition}\label{prop:betaMono}
The monotonicity of the mappings $h(., \theta, \gamma)$ and $\lambda(., \theta, \gamma)$ depends on the values of $\gamma, \theta$:
\begin{itemize}
    \item if $\gamma > \theta$, the two functions are increasing;
    \item if $\gamma = \theta$, the two functions are constants;
    \item if $\gamma < \theta$, the two functions are decreasing.
\end{itemize}
\end{proposition}

The case $\gamma = \theta$ is  special: using \eqref{eq:Generator},  we find that $h \equiv 1$ is an eigenvector and  ${\lambda = - \theta}$. The other cases are more delicate.  Roughly, increasing $\beta$ makes the infection progress faster and accelerates the process. If  $\gamma > \theta$, this benefits the fragmentation more and thus makes the number of clusters and $\lambda$  increase. Otherwise, isolation is more impacted and the converse happens.  

This effect of  acceleration due to the increase of $\beta$ can be understood via the following simplified problem. Let us consider a linear birth and death process $(S_t)_{t \geq 0}$ starting from a single particle: i.e. each particle (representing a cluster) lives during an exponential time of parameter $1/\rho$  and is replaced by two particles  (resp. zero) with probability $p_2=\gamma/(\theta+\gamma)$ (resp. $p_0=\theta/(\theta+\gamma)$), corresponding respectively to fragmentation and isolation.  This birth and death process $(S_t)_{t \geq 0}$ satisfies $ \E[S_t] = \exp(\lambda_{BD} t )$ with 
\begin{align*}
\lambda_{BD}=\rho \Ll(p_2-p_0\Rr)=\rho\Ll( \frac{\gamma-\theta}{\gamma+\theta}\Rr). 
\end{align*}
The monotonicity of $\lambda_{BD}$ in function of $\rho$ thus depends on the sign of $(\gamma-\theta)$, i.e. on the fact that the process is subcritical or supercritical.

Such a dichotomy also exists in our model, and will be clear in a modified process,  which is close to our original GFI process and the Malthusian coefficients are explicitly linked. This modification yields a simpler process, since the probability for a cluster to be isolated or fragmented does not depend on its size any longer and $\gamma=\theta$ provides the critical case for this latter. This is the key property, which will allow us to make a coupling for the modified processes with different $\beta$ and then verify the monotonicities. 

 We divide the proof into three steps. Firstly, we introduce a \textit{modified GFI process},  by slightly changing the isolation term.  Secondly, we obtain coupling  results for times and sizes of our modified GFI process. Finally, we use this coupling to prove Proposition~\ref{prop:betaMono}. We skip some technical details in the second and the third step, which can be found in the Appendix B of the long version \cite{bansaye2021growth} of this paper on arXiv.

\subsubsection{A modified GFI process}\label{subsubsec:mGFI}

In this part, we introduce our modified GFI process. Its dynamic of growth and fragmentation is the same as the original one.  For the dynamic of isolation, we let each edge (instead of each vertex) have rate $\theta$ to be detected independently and then the cluster is isolated. This minor modification of the isolation rate does not change the splitting property, thus the study of the modified GFI process can also be reduced to its size process $(\overline{X}_t, \overline{Y}_t)_{t \geq 0}$. For the parameters $(\beta, \theta, \gamma)$, the associated generator $\overline{\L}$ of the first moment is defined as 
\begin{multline}\label{eq:Generator2}
\overline{\L} f(n) 
= \beta n (f(n+1) - f(n)) - \theta (n-1) f(n) \\ + \sum_{j=1}^{n-1}  \frac{\gamma n}{j(j+1)} \Ll(f(j) + f(n-j) - f(n)\Rr).
\end{multline}    
Compared to the generator $\L$ defined in  \eqref{eq:Generator} with the same parameters, the isolation rate is $\theta(n-1)$ instead of $\theta n$ for a cluster of size $n$, and it is clear that we have
\begin{align}\label{eq:GeneratorEquation}
    \overline{\L} = \L - \theta Id.
\end{align}
Therefore, Lemma~\ref{lem:HarrisV} also applies to $\overline{\L}$. The couple $(\pi, h)$ associated to  $\L$  is also the eigenelement of $\overline{\L}$, with the Perron's root / Malthusian exponent
\begin{align}\label{eq:LL}
    \overline{\lambda} h = \overline{\L} h = (\lambda - \theta) h \Longrightarrow \overline{\lambda} = \lambda - \theta.
\end{align}
For fixed $\gamma, \theta > 0$, the mapping $\beta \mapsto \lambda(\beta, \theta, \gamma)$ shares the same monotonicity with ${\beta \mapsto \overline{\lambda}(\beta, \theta, \gamma)}$, so it suffices to study the the latter one.

The modified GFI process provides some advantages. We observe that for a cluster of size $n$, the fragmentation rate and the isolation rate are respectively $\gamma(n-1)$ and $\theta(n-1)$. So the probability that isolation happens before fragmentation (and conversely) does not depend on its size and is equal
to   $\theta/(\gamma + \theta)$ (resp.  $\gamma/(\gamma + \theta))$. This nice property has inspired the idea of coupling in the next step. 

We finish this part by proving the following lemma.
\begin{lemma}\label{lem:modiGFIphase}
The modified GFI process has the following criteria of phases: it is supercritical ($\overline{\lambda} > 0$) when $\gamma > \theta$, critical ($\overline{\lambda} = 0$) when $\gamma = \theta$ , and subcritical $\overline{\lambda} < 0$ when $\gamma < \theta$.
\end{lemma}
\begin{proof}
For the original GFI process, we know that the case $\gamma = \theta$ gives us $\lambda = - \theta$. Moreover, the monotonicity of the function $\gamma \mapsto \lambda(\beta, \theta, \gamma)$ in Lemma~\ref{lem:StrictTheta} tells us 
\begin{align*}
    \gamma > \theta \Longrightarrow \lambda > -\theta, \qquad \gamma < \theta \Longrightarrow \lambda < -\theta.
\end{align*}
Combing this with \eqref{eq:GeneratorEquation} proves Lemma~\ref{lem:modiGFIphase}.
\end{proof}

\subsubsection{Monotone coupling for the modified GFI processes}\label{subsubsec:MonotoneCoupling}

Let us fix the rates $\theta$ and $\gamma$, and consider an initial cluster of size $n$ in the modified GFI process with the infection rate $\beta$. This cluster may grow by infection for a random time $\tau_n^{\beta}$, called \textit{lifetime}, when  (independently of this time) it makes fragmentation with probability ${\gamma/(\gamma+\theta)}$ or  isolation with probability ${\theta/(\gamma+\theta)}$. When it splits, it leaves two child clusters whose sizes are distributed as the random variable 
$(Z^\beta_{n,1},Z^\beta_{n,2})$ following Proposition~\ref{prop:Split}. That is, $Z^\beta_{n,1}$ is the size of the first child, and $Z^\beta_{n,2}$ is for the second child. Here we write $(Z^\beta_{n,1},Z^\beta_{n,2}) = (0,0)$ by convention if this cluster is isolated.  


We say there is a coupling for two random variables $A$ and $B$ if we can find 
$A'$ distributed as $A$, $B'$  distributed as $B$,  and  $A', B'$ are defined in a common probability space. We sometimes skip the superscripts for these coupled random variables when the context of coupling is clear. 
Then we have the following lemma for $(Z^\beta_{n,1},Z^\beta_{n,2})$ and $\tau_n^{\beta}$ defined above.

\begin{lemma}\label{lem:DomSto}
For any $1\leq n\leq n'$ and $0< \beta\leq \beta'$, there exists a coupling for $\Ll(\tau_n^{\beta}, Z^\beta_{n,1},Z^\beta_{n,2}\Rr)$ and $\Ll(\tau_{n'}^{\beta'}, Z^{\beta'}_{n',1},Z^{\beta'}_{n',2}\Rr)$ such that 
\begin{align*}
\tau_n^{\beta} \geq \tau_{n'}^{\beta'}, \qquad   Z^\beta_{n,1} \leq Z^{\beta'}_{n',1}, \qquad Z^\beta_{n,2} \leq  Z^{\beta'}_{n',2}  \qquad a.s. .
\end{align*}
and in this coupling the two clusters either both reach fragmentation or both reach isolation. 
\end{lemma}

This lemma implies that the larger cluster size or infection rate will result in shorter lifetime and larger sizes for child clusters. Concerning the second statement, since the probability of either event is independent of the cluster size, we can use a common Bernoulli random variable to determine whether they split or are isolated in both processes. For \eqref{lem:DomSto}, it is also natural as larger cluster size or infection rate speed up the growth events.

We now consider two modified GFI processes with common $\theta,\gamma$ but different infection rates $\beta \leq \beta'$ and initial cluster sizes $n_0 \leq n'_0$. We denote the active cluster processes by  $(\overline{\msc{X}}^\beta_{n_0,t})_{t\geq 0}, (\overline{\msc{X}}^{\beta'}_{n'_0,t})_{t\geq 0}$ and their size processes by  $(\overline{X}^{\beta}_{n_0, t})_{t \geq 0}$ and $(\overline{X}^{\beta'}_{n'_0, t})_{t \geq 0}$. 

We can construct a coupling of the two modified GFI processes in a common probability space as follows.  Notice that Lemma~\ref{lem:DomSto} allows us to couple the lifetimes and cluster sizes of any two clusters. In particular, the lemma tells us that either fragmentation or isolation occurs to both of them. We can continue to apply the coupling to their first child clusters and second child clusters if the fragmentation takes place.  If we start from the initial clusters of both process,  we can couple the whole processes in a common probability space denoted by   $\Ll(\Omega, \mcl G, \P^{\beta, \beta'}_{n_0, n'_0}\Rr)$. 
In particular, we can use the UHN labelling to describe the one-to-one mapping or coupling between clusters from the two processes. For instance, the label $ 12$, if it exists, maps the second child of the first child of the initial cluster from one process to another. 

In this coupling, due to Lemma~\ref{lem:DomSto}, any active cluster in $(\overline{\msc{X}}^\beta_{n_0,t})_{t\geq 0}$  has a longer lifetime and smaller child cluster sizes (if they exist) compared to the corresponding cluster in $(\overline{\msc{X}}^{\beta'}_{n'_0,t})_{t\geq 0}$, almost surely.

\subsubsection{Proof of Proposition~\ref{prop:betaMono}}\label{subsubsec:CouplingProof}
Let us prove the monotonicity result on the modified GFI, which will immediately give the result on the original process using the link \eqref{eq:LL} between the maximal eigenvalues.

To compare the Malthusian exponent and the eigenfunction, it suffices to compare $\E[\langle \overline{X}^{\beta}_{n_0, T}, 1 \rangle]$ and $\E[\langle \overline{X}^{\beta'}_{n'_0, T}, 1 \rangle]$ for a large $T$ thanks to the Malthusian behavior in \eqref{eq:Spectral}. 

We compare them in the common probability space $\Ll(\Omega, \mcl G, \P^{\beta, \beta'}_{n_0, n'_0}\Rr)$ constructed previously. 
It can be done using a stopping line as follows. Let us write $\overline{\mathcal U}^{\beta}_T$ 
the set of UHN labels of the active clusters in $\overline{\msc{X}}^{\beta}_{n, T}$. The corresponding clusters in $(\overline{\msc{X}}^{\beta’}_{n'_0, t})_{t \geq 0}$ labeled by $\overline{\mathcal U}^{\beta}_T$  
are active or have been active before time $T$ due to the coupling. Conversely, the active clusters in $\overline{\msc{X}}^{\beta’}_{n'_0, T}$ are issued from active clusters whose labels are in 
$\overline{\mathcal U}^\beta_T$, which allows us to use the branching property (see Figure \ref{fig:Coupling}):
\begin{align*}
\langle \overline{X}^{\beta'}_{n'_0, T},1\rangle=\sum_{u \in\overline{\mathcal U}^\beta_T} N_T^{\beta’}(u).    
\end{align*}
Here $N_T^{\beta’}(u)$ is the number of active clusters in $\overline{\msc{X}}^{\beta’}_{n'_0, T}$ that are issued from 
the cluster labeled by $u$ (either its descendants, or itself if it is still alive up to $T$)
in $(\overline{\msc{X}}^{\beta’}_{n'_0, t})_{0<t\leq T}$,  and we also notice that $\langle \overline{X}^{\beta}_{n_0, T}, 1 \rangle = \sum_{u \in\overline{\mathcal U}^\beta_T} 1$.

The comparison of $\langle \overline{X}^{\beta'}_{n'_0, T},1\rangle$ and $\langle \overline{X}^{\beta}_{n_0, T}, 1 \rangle$  is reduced to that between the mean of $N_T^{\beta’}(u)$ and $1$. For any active cluster labeled by $u$ that has evolved to time $T$ in $(\overline{\msc{X}}^{\beta}_{n_0, t})_{t \geq 0}$, the coupled evolution of the corresponding active cluster in $(\overline{\msc{X}}^{\beta’}_{n'_0, t})_{t \geq 0}$ is at a random time $\tau_T(u)<T$. Thus $N_T^{\beta’}(u)$ is the number of clusters issued from the cluster labeled $u$ in $(\overline{\msc{X}}^{\beta’}_{n'_0, t})_{t \geq 0}$ that evolves for the extra time $(T-\tau_T(u))$. We recall that each active cluster either fragments or be isolated with fixed probabilities whose difference is  $(\gamma-\theta)/(\gamma+\theta)$. Then, if $\gamma > \theta$ (resp.  $\gamma < \theta$), $N_T^{\beta’}(u)$ as the mean number of clusters issued from the cluster labeled by $u$ during the extra time $(T-\tau_T(u))$ is larger than $1$ (resp. smaller than $1$). This finishes the proof of Proposition~\ref{prop:betaMono}.  

\begin{figure}
\label{stoptheline}
    \centering
    \includegraphics[scale = 0.4]{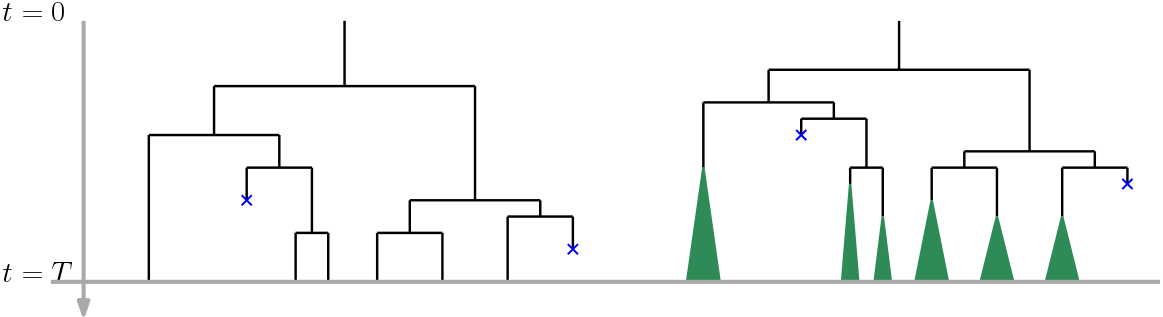}
    \caption{An illustration of the monotone coupling argument and the proof of Proposition~\ref{prop:betaMono}. The blue crosses represent isolation. The figure on the left is for $(\overline{\msc{X}}^{\beta}_{n_0, t})_{0 \leq t\leq T}$, and the one on the right is for $(\overline{\msc{X}}^{\beta'}_{n'_0, t})_{0 \leq t\leq T}$. With a monotone coupling, the evolution on the right is quicker than the one on the left, so each cluster at time $T$ on the left can find its coupled one on the right at a random moment before $T$. The rest of evolution on the right (the green cones) is not coupled with the left up to time $T$, thus the monotonicity depends on whether the extra time will make the population grow or diminish. That is, the monotonicity depends on the phase (supercricial, critical and subcritical) that the population belongs to. }
    \label{fig:Coupling}
\end{figure}

\subsubsection{Proof of Theorem~\ref{thm:Regularity}} We combine Proposition~\ref{prop:Regularity}, Proposition~\ref{prop:Curve} and Proposition~\ref{prop:betaMono} to obtain Theorem~\ref{thm:Regularity}. 

\section{Examples and simulations}\label{sec:examsimu}

In this part, we give some examples about how our model and theorems can be applied to perform the numerical simulation of epidemics. The key quantity of our model is the Malthusian exponent $\lambda$. It dictates whether the infection will grow or vanish, and moreover serves as the growth or decay rate. We described how $\lambda$ depends on $(\beta,\theta,\gamma)$ in Proposition~\ref{prop:Regularity} and ~\ref{prop:betaMono}. However, the function $\lambda(\beta,\theta,\gamma)$ is not explicit, so it is useful to investigate its values numerically. We use Finite Difference Method applied to \eqref{eq:defGenerator} to approximate $M_t$ and then use \eqref{eq:Spectral} to approximate $\lambda$. 

We simulated $\lambda$ for different values of the three parameters in Figure~\ref{fig:4}. Figure~\ref{fig:3d} gives a surface of the function $\lambda$ with $\beta=1$, which confirms the monotonicity of $\lambda$ on $\theta$ and $\gamma$ (see also Figure~\ref{fig:theta} and Figure~\ref{fig:gamma}). Thus to reduce infection, we should conduct more contact tracing (increase $\theta$) and avoid loss of contact information (decrease $\gamma$). These two approaches are natural and intuitive.

Figure~\ref{fig:beta}  confirms the monotonicity of $\lambda$ on $\beta$, see Proposition~\ref{prop:betaMono}. When $\theta=0.03,0.06,0.09$ (thus $\theta<\gamma$), $\lambda$ is increasing on $\beta$; when $\theta=0.12,0.15$ (thus $\theta>\gamma$), $\lambda$ is a decreasing function of $\beta$. 
\begin{figure}
    \begin{subfigure}[t]{0.5\textwidth}
    	\vskip 0pt
        \centering
        \includegraphics[width=1\linewidth]{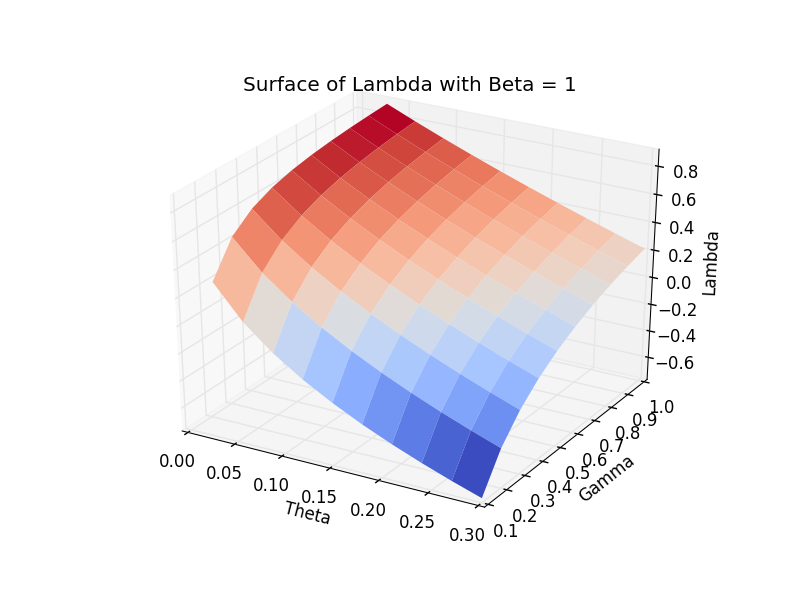} 
        \caption{The values of $\lambda$ with $\beta=1$.} \label{fig:3d}
    \end{subfigure}
    \hfill
    \begin{subfigure}[t]{0.45\textwidth}
     	\vskip 0pt
        \centering
        \includegraphics[width=1\linewidth]{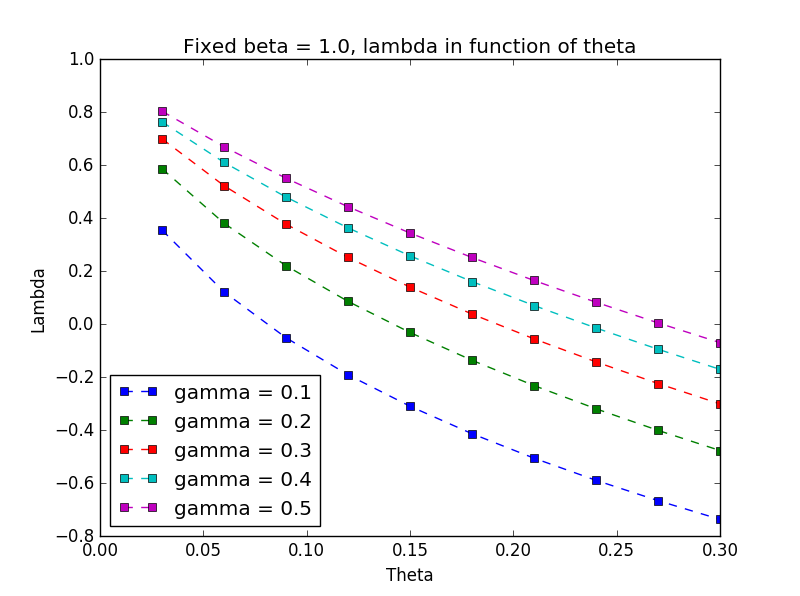} 
        \caption{Each curve describes the dependence of $\lambda$ on $\theta$, with $\gamma, \beta$ fixed.} \label{fig:theta}
    \end{subfigure}

    \begin{subfigure}[t]{0.45\textwidth}
     	\vskip 0pt
        \centering
        \includegraphics[width=1\linewidth]{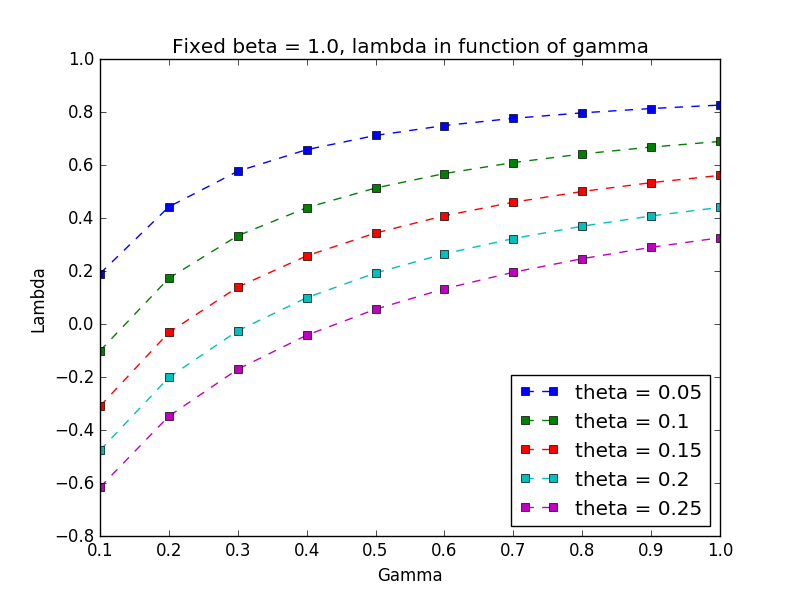} 
        \caption{Each curve describes the dependence of $\lambda$ on $\gamma$, with $\theta, \beta$ fixed.} \label{fig:gamma}
    \end{subfigure}
    \hfill
\begin{subfigure}[t]{0.45\textwidth}
    	\vskip 0pt
        \centering
        \includegraphics[width=1\linewidth]{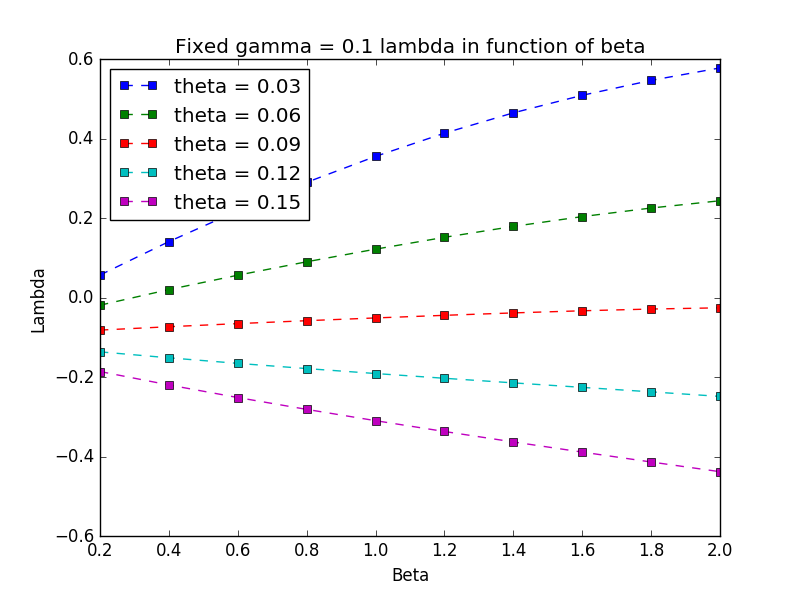} 
        \caption{Each curve describes the dependence of $\lambda$ on $\beta$, with $\theta, \gamma$ fixed.} \label{fig:beta}
    \end{subfigure}
    \caption{The dependence of $\lambda$ on $(\beta,\theta,\gamma)$.}
    \label{fig:4}
\end{figure}

\section{Complementary results and further discussions}\label{Sec:Comp}
In this part, we discuss some other properties and generalizations of our model.

\subsection{Fragmentation by removing vertices}
One can also consider a similar GFI process where the fragmentation is generated by removing vertices. We call the new process \textit{v-GFI process} for short. More precisely, this process has the same dynamics of growth and isolation, but every vertex rings independently with an exponential clock of parameter $\gamma$, and once it rings, the associated vertex is removed from the cluster to generate several subclusters. 

For the model of fragmentation by removing vertices on RRT (without isolation), some interesting properties about the size distribution have been discussed in \cite{kalay2014fragmentation}. Especially, a generalized splitting property should hold for our v-GFI process, so we can also follow the same strategy to study its size process $(X_t, Y_t)_{t \geq 0}$. Moreover, by \cite[eq.(7)]{kalay2014fragmentation} we deduce that the generator for the first moment semigroup of $(X_t)_{t\geq 0}$ is 
\begin{equation}
\begin{split}
\L_{\bullet} f(n) &= \beta n (f(n+1) - f(n)) - \theta n f(n) \\
& \qquad \qquad \qquad - \gamma n f(n) + \gamma n \sum_{j=1}^{n-1} \Ll(\frac{1}{j(j+1)} + \frac{1}{(n-j)(n-j+1)}\Rr) f(j).
\end{split}
\end{equation}
Compared to \eqref{eq:Generator}, we have 
$\L_{\bullet} f = \L f - \gamma f$, which implies that ${e^{\L_{\bullet}t} = e^{-\gamma t}e^{\L t} = e^{-\gamma t}M_t}$. Therefore we conclude that, quite surprisingly, $\L$ and $\L_{\bullet}$ have
exactly the same eigenelements $(\pi, h)$ given in Proposition~\ref{Perroneigen}, with a modification of Perron's root $\lambda \mapsto \lambda - \gamma$. We can expect to recover other similar properties for the v-GFI process.

\subsection{Role of the   initial condition}\label{sec: initial}
Although the main results are proved with the initial condition $G_0 = \{0\}$ (a patient zero), the proofs and results will not change much if the initial cluster is a RRT of any size. Indeed the reduction of the study to a growth fragmentation isolation process for sizes works similarly. 

Besides,  extension of the results to an initial condition
given by a collection of RRTs 
can be achieved using that each initial cluster evolves independently by branching property.

Finally, one may wish to start from one (or a collection of) deterministic finite tree(s).We believe that our work can be  adapted or used in that purpose.
For instance using the stopping line when the clusters have size $1$
to exploit the results of this paper.

\subsection{Model generalizations}\label{sec:generalise}

We can use the same method to deal with  generalizations such as allowing a cluster to lose several edges at the same time, or infect several individuals at the same time. More delicate generalizations could be interesting too: \begin{itemize}
\item Adding recovering;
\item Breaking the Markov property (adding an age would register the time for infection and would affect the individual rates in the epidemics: contamination,  loss of contact information, and recovery);
\item Improving and enriching the tracing procedure;
\item Adding heterogeneity within the population (for tracing and infection).
\end{itemize}
The reduction we have exploited (from collections of trees to collection of sizes, due to the
nice splitting property of RRT) should fail in general. We hope there is still some nice
probabilistic constructions (even if more complex) to find but it is still speculative.
We are confident on the extension of asymptotic  results on the Malthusian growth and
 empirical measures. It could be achieved without going through a
reduction of the state space and without describing the distribution of clusters at fixed times, but rather exploiting the general statement on semigroups that
we have used here on a more complex state space (space of trees). In particular, the Lyapunov functions that we have exhibited here may still be relevant to deal with processes in more general state spaces.

\subsection{Finite-distance contact tracing}
In our model, the contact tracing is implemented for the whole cluster despite of its size, which is very rare in reality. It is natural to ask whether the properties observed in our model also hold for the finite-distance contact tracing. The answer is probably true, but the justification requires some new mathematical techniques.

\begin{appendix}
\section{Conditional expectation and variance}\label{varianceApp} 
We prove here Lemma~\ref{appen} for the conditional expectation and variance  estimates  of
$$
Z_{k,p,K}:= \bracket{\widetilde{X}_{(k+1)\Delta-}, \x{p}_{>K}} - \bracket{\widetilde{X}_{k\Delta}, \x{p}_{>K}},$$
used in \eqref{eq:bcIntegralBound}. Here $(\widetilde{X}_t)_{t \in [k\Delta, (k+1)\Delta)}$, containing only growth and no fragmentation nor isolation, is a coupling of $(X_t)_{t\in [k\Delta,(k+1)\Delta)}$. 

\begin{proof}[Proof of Lemma \ref{appen}] For the conditional expectation, we decompose it using the genealogy of clusters
\begin{equation}\label{eq:BdExptDecom}
\E[Z_{k,p, K} \vert \mcl F_{k\Delta}] 
= \sum_{u \in \mcl U_{k\Delta}}\E\left[\x{p}_{>K}(\widetilde{X}_{(k+1)\Delta}^u) - \x{p}_{>K}(\widetilde{X}_{k\Delta}^u) \,\, \Big\vert\,\, \mcl F_{k\Delta}\right].
\end{equation} 
We observe that for any $0<a\leq b$, 
\begin{align*}
\x{p}_{>K}(b) - \x{p}_{>K}(a) 
&= (b^p - a^p)\Ind{b > K} + a^p \Ind{a \leq K < b}\Ind{a \leq K} + a^p\Ind{a \leq K < b}\Ind{a > K}. 
\end{align*}
On the right hand side, the second term can be bounded by ${K^p\Ind{a \leq K < b}}$ and the third term is zero, so we have 
\begin{align}\label{eq:IndicatorTrick}
\x{p}_{>K}(b) - \x{p}_{>K}(a) \leq (b^p - a^p) + K^p\Ind{a \leq K < b}.
\end{align}
Since there is only growth in the process $\widetilde{X}_t$ on $[k\Delta,(k+1)\Delta)$, $\widetilde{X}_{(k+1)\Delta-}^u \geq   \widetilde{X}_{k\Delta}^u$ for any $u \in \mcl U_{k\Delta}$. So we can apply \eqref{eq:IndicatorTrick} with $a = \widetilde{X}_{k\Delta}^u$ and $b = \widetilde{X}_{(k+1)\Delta-}^u$ to obtain
that 
\begin{multline}\label{eq:BdExptTerms}
\E[Z_{k,p, K}\vert \mcl F_{k\Delta}] \leq \E\left[\bracket{\widetilde{X}_{(k+1)\Delta-}, \x{p}} - \bracket{\widetilde{X}_{k\Delta}, \x{p}}  \,\,\vert \,\, \mcl F_{k\Delta}\right] \\ + {K^p \sum_{u \in \mcl U_{k\Delta}}\E\Ll[ \Ind{ \widetilde{X}_{k\Delta}^u \leq K < \widetilde{X}_{(k+1)\Delta}^u } \vert \mcl F_{k\Delta}\Rr]}.
\end{multline}
For the first term on the right hand side in \eqref{eq:BdExptTerms}, we follow 
the proof of 
Lemma~\ref{lemma:duhamel}-(i) to get   
\begin{equation*}
\begin{split}
\E\Ll[\bracket{\widetilde{X}_{(k+1)\Delta-}, \x{p}} - \bracket{\widetilde{X}_{k\Delta}, \x{p}}  \,\,\vert \,\, \mcl F_{k\Delta}\Rr] &\leq \Ll(e^{2^{p-1} p\beta  \Delta} - 1\Rr)\bracket{\widetilde X_{k \Delta}, \x{p}}. 
\end{split}
\end{equation*}
For the second term in \eqref{eq:BdExptTerms},  we can control it by the total number of active clusters of size smaller than $K$ at $k\Delta$ that grow within $[k\Delta,(k+1)\Delta)$: 
\begin{align*}
& K^p \sum_{u \in \mcl U_{k\Delta}}\E\Ll[ \Ind{\widetilde{X}_{(k+1)\Delta}^u > K} - \Ind{{\widetilde X}_{k\Delta}^u \leq K} \,\, \Big\vert\,\, \mcl F_{k\Delta}\Rr] &\\
 &\leq K^p \sum_{u \in \mcl U_{k\Delta}}\E\Ll[\Ind{{\widetilde X}_{k\Delta}^u \leq K, \text{ the cluster labeled by }u \text{ grows in } \widetilde{X}_t  \text{ within } [k\Delta, (k+1)\Delta)} \,\, \Big\vert\,\, \mcl F_{k\Delta}\Rr]\\
 &\leq (1-e^{-\beta  \Delta K}) K^p \bracket{\widetilde X_{k\Delta},1}. 
\end{align*}
Plugging in the two inequalities to \eqref{eq:BdExptTerms} yields \eqref{eq:BdExptConditional}. \\

 For the conditional variance, we have 
\begin{align*}
\var\Ll[Z_{k,p,K} \,\, \vert  \,\, \mcl F_{k\Delta}\Rr] = \var\Ll[\sum_{u \in \mcl U_{k\Delta}} \Ll(\x{p}_{>K}(\widetilde{X}_{(k+1)\Delta-}^u) - \x{p}_{>K}(\widetilde{X}_{k\Delta}^u)\Rr) \,\, \vert  \,\, \mcl F_{k\Delta}\Rr]. 
\end{align*}
By branching property,  
\begin{align*}
\var\Ll[Z_{k,p,K} \,\, \vert  \,\, \mcl F_{k\Delta}\Rr] &= \sum_{u \in \mcl U_{k\Delta}} \var\Ll[\Ll(\x{p}_{>K}(\widetilde{X}_{(k+1)\Delta-}^u) - \x{p}_{>K}(\widetilde{X}_{k\Delta}^u)\Rr)\,\, \vert  \,\, \mcl F_{k\Delta}\Rr]\\
&\leq \sum_{u \in \mcl U_{k\Delta}} \E\Ll[\Ll(\x{p}_{>K}(\widetilde{X}_{(k+1)\Delta-}^u) - \x{p}_{>K}(\widetilde{X}_{k\Delta}^u)\Rr)^2 \,\, \vert  \,\, \mcl F_{k\Delta}\Rr]\\
&\leq \sum_{u \in \mcl U_{k\Delta}} \E\Ll[\x{p}^2_{>K}(\widetilde{X}_{(k+1)\Delta}^u-) - \x{p}^2_{>K}(\widetilde{X}_{k\Delta}^u) \,\, \vert  \,\, \mcl F_{k\Delta}\Rr]\\
&= \E\Ll[\bracket{\widetilde{X}_{(k+1)\Delta-}, \x{2p}_{>K}} - \bracket{\widetilde{X}_{k\Delta}, \x{2p}_{>K}}\,\, \vert  \,\, \mcl F_{k\Delta}\Rr].
\end{align*}
From the second line to the third line we use $(a-b)^2 \leq a^2 - b^2$ for any $a>b>0$. The rest is the same as in the computation of conditional expectation and we obtain \eqref{eq:BdVarConditional}.
\end{proof}


\section{Coupling for modified GFI processes}\label{sec:CouplingDetails}
We present here the details of the proof of Proposition~\ref{prop:betaMono}. We will prove the monotone property Lemma~\ref{lem:DomSto} in Section~\ref{sec:CouplingOne}, which is divided into the part for size in Section~\ref{sec:CouplingSize} and the part for waiting time in Section~\ref{sec:CouplingTime}. Then we apply this coupling along one cluster to the whole process in $[0,T]$ and prove Proposition~\ref{prop:betaMono} in Section~\ref{sec:CouplingProof}

Throughout this section, we fix $\theta, \gamma, T > 0$, $\beta' \geq \beta >0$ and two positive integers $n'_0 \geq n_0$. Our object is a probability space $\Ll(\Omega, \mcl G, \P^{\beta, \beta'}_{n_0, n'_0}\Rr)$ as a coupling such that two modified GFI processes live in it: one is of parameters $(\beta, \theta, \gamma)$ starting from an initial RRT of size $n_0$, and the other is of parameters $(\beta', \theta, \gamma)$ starting from an initial RRT of size $n'_0$.  We denote by $(\overline{X}_t)_{0 \leq t \leq T}$ the size process of the former, and $(\overline{X}'_t)_{0 \leq t \leq T}$ the size process of the latter. We will usually abuse these two notations to indicate the two modified GFI processes.

Recall the Ulam-Harris-Neveu notation defined in Section~\ref{UHN}. For any cluster $u \in \mathcal{U}$, its birth time is the moment when his parent makes the fragmentation, and its end of lifetime is the moment for its fragmentation or isolation event. During its lifetime, it has several growth events but the label is unchanged.

\subsection{Coupling for one cluster}\label{sec:CouplingOne}

\subsubsection{Coupling for the size distribution}\label{sec:CouplingSize}
As we said, the construction consists of the coupling of the event and the size distribution. We focus at first the later one, and study the passage of size domination between the generation. 

\begin{algorithm}[Coupling for the size distribution]\label{algo:size}
For the two modified processes $(\overline{X}'_t)_{0 \leq t \leq T}$ and  $(\overline{X}_t)_{0 \leq t \leq T}$ defined above and each cluster labeled $u \in \mathcal{U}$ appearing in $(\overline{X}_t)_{0 \leq t \leq T}$, we use two independent random variables 
\begin{align*}
    &\widehat{U}(u) \sim \operatorname{Uniform}[0,1], \qquad  \widehat{V}(u) \sim \operatorname{Bernoulli}\Ll(\frac{\gamma}{\gamma + \theta}\Rr),
    \end{align*}
to sample the fragmentation/isolation event when it comes to the end of lifetime. We denote by $N', N$ respectively for the size of $u$ in $(\overline{X}'_t)_{0 \leq t \leq T}$ and  $(\overline{X}_t)_{0 \leq t \leq T}$ at the end of its lifetime.
\begin{enumerate}
    \item If $\widehat{V}(u) = 1$, the event is fragmentation in both processes. We then define
        \begin{align}\label{eq:SecondChild}
            j := \Ll\lfloor \frac{N}{(N-1)\widehat{U}(u) + 1}\Rr\rfloor, \qquad j' := \Ll\lfloor \frac{N'}{(N'-1)\widehat{U}(u) + 1}\Rr\rfloor,
        \end{align}
        and we set $(N-j, j)$ (resp. $(N'-j',j')$) for the size of the first child cluster and the second child cluster in $(\overline{X}_t)_{0 \leq t \leq T}$ (resp. $(\overline{X}'_t)_{0 \leq t \leq T}$). 
    
    \item If $\widehat{V}(u) = 0$, the event is isolation in both processes
\end{enumerate}
\end{algorithm}

We justify that this coupling of size gives us the desired passage of the size domination between the generation.
\begin{proposition}\label{prop:fragDomination}
Algorithm~\ref{algo:size} is consistent with the fragmentation and isolation in the modified GFI process. Moreover, supposing $N' \geq N$, we also have the domination for the size of child clusters 
\begin{align}\label{eq:sizeDomination}
    j' \geq j, \qquad N'-j' \geq N - j.
\end{align}
\end{proposition}
\begin{proof}
As defined at the beginning of Section~\ref{subsubsec:mGFI}, the fragmentation and isolation are all attached to the edge in the modified GFI process. Then conditioned that an event is fragmentation or isolation, the probability to be fragmentation is $\frac{\gamma}{\gamma + \theta}$, which is independent of its size. Moreover, following the splitting property \ref{prop:Split}, the event that a cluster of size $N$ has a second child of size $j$ can be sampled by 
\begin{align*}
    \Ll(\frac{N}{N-1}\Rr) \frac{1}{j+1} < \widehat{U}(u) + \frac{1}{N-1} \leq \Ll(\frac{N}{N-1}\Rr) \frac{1}{j},
\end{align*}
which is equivalent to \eqref{eq:SecondChild}. This proves the consistence of the coupling of size. 

We then prove the size domination. Notice that the event type is same for the two processes, because it is determined by the common random variables $\widehat{U}(u), \widehat{V}(u)$. The domination thus \eqref{eq:sizeDomination} can be reduced to prove the following inequality: for every $x \in (0,1]$ and two positive integers $N < N'$, we have 
\begin{align}\label{eq:fragDomination}
    0 \leq \Ll\lfloor \frac{N'}{N' x + (1-x)}\Rr\rfloor - \Ll\lfloor \frac{N}{N x + (1-x)}\Rr\rfloor \leq N'-N.
\end{align}
For the first inequality, we notice that 
\begin{align*}
 \frac{N}{N x + (1-x)} = \frac{1}{x + (1-x)/N},   
\end{align*}
so it is increasing with respect to $N$. For the second inequality, we have
\begin{align*}
\Ll\lfloor \frac{N'}{N' x + (1-x)}\Rr\rfloor - \Ll\lfloor \frac{N}{N x + (1-x)}\Rr\rfloor &\leq    \frac{N'}{N' x + (1-x)} - \Ll(\frac{N}{N x + (1-x)} - 1\Rr)\\
&= \frac{(N'-N)(1-x)}{(1+(N'-1)x)(1+(N-1)x)} + 1\\
&< (N'-N) + 1.
\end{align*}
From the second line to the third line, because $x > 0$, we have $1-x < 1$, ${1+(N'-1)x > 1}$, ${1+(N-1)x > 1}$ and ``='' cannot be attained. Since the difference is an integer, this implies the second inequality in \eqref{eq:fragDomination} and we finish the proof.
\end{proof}

\subsubsection{Coupling for the waiting time}\label{sec:CouplingTime}

We now turn to the coupling of the waiting time, and aim to let larger cluster produce more growths within shorter lifetime. The main idea is a \textit{backward time-change sampling}: we sample at first a \textit{pre-waiting time} for the fragmentation/isolation event, which is generated by the current cluster size by supposing no growth events during its lifetime. Then we need to fill in the lifetime with the growth events, and make necessary time change with respect to the latest cluster size. After this time change, the pre-waiting time becomes the real waiting time of the  fragmentation/isolation event.  

Now, we give the detailed description of the backward time-change sampling in a modified GFI process of parameter $(\beta, \theta, \gamma)$.
\begin{algorithm}[Backward time-change sampling]\label{algo:Backward}
For $u \in \mathcal{U}$, we sample the events in its lifetime as following. Here all the random variables $W(u), \widetilde{W}(u), (Z_i(u))_{i \in \N_+}$ are independent. We denote by $n$ the cluster size of $u$ at birth. 
\begin{enumerate}[label=(\roman*)]
    \item If $n \geq 2$, we apply the following steps: 
    \begin{enumerate}
        \item[(i.1)] Set $\tau_0(u) = 0$. 
        
        \item[(i.2)] Sample $W(u) \sim \operatorname{Exp}((\gamma + \theta)(n-1))$ as the pre-waiting time from $\tau_0(u)$ for the fragmentation/isolation event for $u$. To unify the notation, we set $\kappa_0(u) := W(u)$ as the pre-waiting time from $\tau_0(u)$. 
        
        \item[(i.3)] For any $i \in \N$, suppose that we have already obtained $\tau_i(u)$ for the $i$-th growth event and $\kappa_i(u)$ as the pre-waiting time of the fragmentation/isolation from $\tau_i(u)$.  We sample the ${(i+1)}$-th growth event, whose waiting time from $\tau_{i}(u)$ is sampled from a random variable ${Z_{i+1}(u) \sim \operatorname{Exp}(\beta (n+i))}$. Then there are two cases.
        \begin{itemize}
            \item If $Z_{i+1}(u) < \kappa_i(u)$, then the $(i+1)$-th growth event happens before the fragmentation/isolation event. We set
            \begin{align}\label{eq:LifetimeRefresh1}
                \tau_{i+1}(u) := \tau_{i}(u) + Z_{i+1}(u),
            \end{align}
            as the moment for the $(i+1)$-th growth event. We also refresh the pre-waiting time for the fragmentation/isolation event from $\tau_{i+1}(u)$ by 
            \begin{align}\label{eq:LifetimeRefresh2}
                \kappa_{i+1}(u) := \Ll(\frac{n+i-1}{n+i}\Rr)\Ll(\kappa_i(u) - Z_{i+1}(u)\Rr).
            \end{align}
            
            \item If $Z_{i+1}(u) \geq \kappa_i(u)$, then there are only $i$ growth events before the fragmentation/isolation event. Therefore, we stop the recurrence of (i.3) and set 
            \begin{align}\label{eq:LifeLength}
                \tau_{\dagger}(u) := \tau_i(u) + \kappa_i(u), 
            \end{align}
            as the length of lifetime of $u$.
        \end{itemize}
\end{enumerate}
    \item If $n = 1$ at birth, we apply the following steps:
    \begin{enumerate}
        \item[(ii.1)] Sample $\widetilde{W}(u) \sim \operatorname{Exp}(\beta)$ for the waiting time of the first growth event of $u$. Set $\tau_0 := \widetilde{W}(u)$.
        \item[(ii.2)] Repeat step (i.2)-(i.3) for $u$ as a cluster of size $2$. 
    \end{enumerate}
\end{enumerate}
\end{algorithm}

This algorithm is also consistent with the modified GFI process thanks to the memory-less property of the exponential random variables. More precisely, we make use of Lemma~\ref{lem:Backward}, which ensures that the time change in the step (i.3) gives us an new independent exponential random variable after the growth. 

\begin{lemma}\label{lem:Backward}
For three independent random variables $Z_1, Z_2, Z_3$ such that ${Z_i \sim \operatorname{Exp}(\alpha_i)}$, with ${\alpha >0}$ and ${i = 1,2,3}$, we have 
\begin{align*}
    \Ll(Z_1, \frac{\alpha_2}{\alpha_3}(Z_2 - Z_1)\Rr)\Ind{Z_1 \leq Z_2} \eqdist (Z_1, Z_3)\Ind{Z_1 \leq Z_2}.
\end{align*}
\end{lemma}
\begin{proof}
It suffices to prove that, for any $f \in C_b(\R^2)$, we have 
\begin{align}\label{eq:ExpTest}
\E\Ll[f\Ll(Z_1, \frac{\alpha_2}{\alpha_3}(Z_2 - Z_1)\Rr)\Ind{Z_1 \leq Z_2} \Rr] = \E[f(Z_1, Z_3)\Ind{Z_1 \leq Z_2}].    
\end{align}
For the left-hand side, we have 
\begin{align*}
&\E\Ll[f\Ll(Z_1, \frac{\alpha_2}{\alpha_3}(Z_2 - Z_1)\Rr)\Ind{Z_1 \leq Z_2} \Rr]\\  & = \int_{0}^{\infty} \alpha_1 e^{-\alpha_1 z_1} \int_{z_1}^{\infty} \alpha_2 e^{-\alpha_2 z_2} f\Ll(z_1, \frac{\alpha_2}{\alpha_3}(z_2 - z_1)\Rr) \, \d z_2 \d z_1 \\
& = \int_{0}^{\infty} \alpha_1 e^{-(\alpha_1 + \alpha_2) z_1} \int_{z_1}^{\infty} \alpha_2 e^{-\alpha_2 (z_2-z_1)} f\Ll(z_1, \frac{\alpha_2}{\alpha_3}(z_2 - z_1)\Rr) \, \d z_2 \d z_1.
\end{align*}
By setting $z_3 = \frac{\alpha_2}{\alpha_3}(z_2 - z_1)$, we have 
\begin{multline}\label{eq:ExpTestLHS}
\E\Ll[f\Ll(Z_1, \frac{\alpha_2}{\alpha_3}(Z_2 - Z_1)\Rr)\Ind{Z_1 \leq Z_2} \Rr] \\  = \int_0^\infty  \alpha_1 e^{-(\alpha_1 + \alpha_2) z_1} \int_0^\infty \alpha_3 e^{-\alpha_3 z_3} f(z_1, z_3) \, \d z_3 \d z_1. 
\end{multline}
For the right-hand side of \eqref{eq:ExpTest}, we have 
\begin{equation}\label{eq:ExpTestRHS}
\begin{split}
& \E[f(Z_1, Z_3)\Ind{Z_1 \leq Z_2}] \\
& = \int_{0}^{\infty} \alpha_1 e^{-\alpha_1 z_1} \int_{z_1}^{\infty} \alpha_2 e^{-\alpha_2 z_2} \int_0^\infty \alpha_3 e^{-\alpha_3 z_3} f\Ll(z_1, z_3\Rr) \, \d z_3 \d z_2 \d z_1\\
& = \int_0^\infty  \alpha_1 e^{-(\alpha_1 + \alpha_2) z_1} \int_0^\infty \alpha_3 e^{-\alpha_3 z_3} f(z_1, z_3) \, \d z_3 \d z_1. 
\end{split}    
\end{equation}
We combine \eqref{eq:ExpTestLHS} and \eqref{eq:ExpTestRHS}, which proves Lemma~\ref{lem:Backward}.
\end{proof}

The advantage of Algorithm~\ref{algo:Backward} is that the waiting time of fragmentation/isolation event largely depends on one random variable $W(u)$. This helps us construct a coupling to ensure the domination of lifetime length.

\begin{algorithm}[Coupling of waiting time]\label{algo:Time}
For any $u \in \mathcal{U}$ appearing in $(\overline{X}_t)_{0 \leq t \leq T}$, let $n, n'$ be respectively its size in $(\overline{X}_t)_{0 \leq t \leq T}$ and $(\overline{X}'_t)_{0 \leq t \leq T}$ at birth. If $n \leq n'$, we have the following coupling for its evolution in two processes using Algorithm~\ref{algo:Backward}. Here for every quantity ($W(u), Z_i(u), \tau_i(u), \kappa_i(u) \cdots$) in Algorithm~\ref{algo:Backward}, we keep its original notation for that in $(\overline{X}_t)_{0 \leq t \leq T}$, and the notation with superscript ($W'(u), Z'_i(u), \tau'_i(u), \kappa'_i(u) \cdots$) for that in $(\overline{X}'_t)_{0 \leq t \leq T}$.
\begin{enumerate}[label=(\roman*)]
    \item Sample two sets of random variables $\Lambda(u), \widehat{\Lambda}(u)$ that 
    \begin{align*}
        \Lambda(u) := \{\chi(u), \eta(u), (\xi_i(u))_{i \in \N_+}\}, \qquad \widehat{\Lambda}(u) := \{\widehat{\chi}(u), \widehat{\eta}(u), (\widehat{\xi}_i(u))_{i \in \N_+}\}.
    \end{align*}
    All the random variables are independent and they satisfy
    \begin{align*}
    \chi(u), \eta(u), \xi_i(u), \widehat{\chi}(u), \widehat{\eta}(u), \widehat{\xi}_i(u) \sim \operatorname{Exp}(1), \quad \forall i \in \N_+.
    \end{align*}
	The random variables in $\widehat{\Lambda}(u)$ are used in both processes, while that in $\Lambda(u)$ are only used in $(\overline{X}_t)_{0 \leq t \leq T}$.  
    \item We use $\widehat{\eta}(u)$ to generate the common randomness concerning the waiting time of fragmentation/isolation 
    \begin{equation}\label{eq:Coupling1}
    \begin{split}
        W(u) := \frac{\widehat{\eta}(u)}{(\gamma + \theta) (n-1 + \Ind{n=1})}, \qquad W'(u) := \frac{\widehat{\eta}(u)}{(\gamma + \theta) (n'-1 + \Ind{n'=1})}.
    \end{split}
    \end{equation}
    Here the indicator function comes from the fact that there is a supplementary growth event before sampling the fragmentation/isolation event in (ii) of Algorithm~\ref{algo:Backward}; we will also see them in the following paragraphs. 
    \item For the coupling of the growth events, we distinguish two cases.
    \begin{enumerate}
        \item[(iii.1)] If $n'= n$, then we use the random variable $ \widehat{\Lambda}(u)$ to generate the common randomness of growths for the two processes. That is, we set 
        \begin{equation}\label{eq:Coupling2}
        \begin{split}
             &\widetilde{W}(u) = \frac{\widehat{\chi}(u)}{\beta}, \qquad Z_i(u) = \frac{\widehat{\xi}_i(u)}{\beta (n + i - 1 + \Ind{n=1})}, \quad \forall i \in \N_+, \\
            &\widetilde{W}'(u) = \frac{\widehat{\chi}(u)}{\beta'}, \qquad Z'_i(u) = \frac{\widehat{\xi}_i(u)}{\beta' (n + i - 1 + \Ind{n=1})},  \quad \forall i \in \N_+.
        \end{split}
        \end{equation}
        \item[(iii.2)] If $n' > n$, then our strategy is to let the growth in $(\overline{X}_t)_{0 \leq t \leq T}$ evolve independently until size $n'$, and then make coupling as (iii.1). That is, we let the first $(n'-n)$ growth events in $(\overline{X}_t)_{0 \leq t \leq T}$ be sampled independently from that of $(\overline{X}'_t)_{0 \leq t \leq T}$ using the random variables in $\Lambda(u)$ 
        \begin{align*}
            \widetilde{W}(u) = \frac{\chi(u)}{\beta}, \qquad Z_i(u) = \frac{\xi_i(u)}{\beta (n + i - 1 + \Ind{n=1})}, \quad \forall 1 \leq i \leq (n'-n-\Ind{n=1}). 
        \end{align*}
         If all the growth events above are before the fragmentation/isolation event in $(\overline{X}_t)_{0 \leq t \leq T}$, then the current cluster size of $u$ is also $n'$, and we let the rest of the the evolution of $u$ in $(\overline{X}_t)_{0 \leq t \leq T}$ be coupled with that in $(\overline{X}'_t)_{0 \leq t \leq T}$
        \begin{equation}\label{eq:Coupling3}
        \begin{split}
            Z'_i(u) = \frac{\widehat{\xi}_i(u)}{\beta' (n' + i - 1)},  \quad        Z_{i + n'-n-\Ind{n=1}}(u) = \frac{\widehat{\xi}_i(u)}{\beta (n' + i - 1)},  \quad \forall i \in \N_+.
        \end{split}
        \end{equation}
        Otherwise, the number of growth is less than $(n'-n)$ in $(\overline{X}_t)_{0 \leq t \leq T}$ and we only need the firth equation in \eqref{eq:Coupling3} for the growth events in $(\overline{X}'_t)_{0 \leq t \leq T}$. 
    \end{enumerate}
\end{enumerate}
\end{algorithm}

It is not hard to verify that Algorithm~\ref{algo:Time} gives a coupling of Algorithm~\ref{algo:Backward}. One only needs to check that all the random variables in Algorithm~\ref{algo:Time} have the proper parameters as in Algorithm~\ref{algo:Backward}. Meanwhile, some common random variables create correlations between the processes and this is the object of the coupling. As we said, thanks to the backward time-change sampling, we only need to sample one exponential random variable for the waiting time of fragmentation/isolation, and it is quite natural to use a common random variable $\widehat{\eta}(u)$ to create monotonicity. It is also the idea for the growth events, but the case $n < n'$ is a little more delicate.

\begin{proposition}\label{prop:CouplingBranch}
Let $u \in \mathcal{U}$ appearing in $(\overline{X}_t)_{0 \leq t \leq T}$, and suppose two positive integers $n' \geq n$ be respectively its size in $(\overline{X}'_t)_{0 \leq t \leq T}$ and $(\overline{X}_t)_{0 \leq t \leq T}$ at birth. Then in Algorithm~\ref{algo:Time}, the length of lifetime is shorter in $(\overline{X}'_t)_{0 \leq t \leq T}$ than that in $(\overline{X}_t)_{0 \leq t \leq T}$, and its size at the end of lifetime is also larger in the former process.  
\end{proposition}

Before proving Proposition~\ref{prop:CouplingBranch}, one intermediate step is to define a \textit{predicted length of lifetime} after $i$-th growth, as a generalization of \eqref{eq:LifeLength} in Algorithm~\ref{algo:Backward} 
\begin{align}\label{eq:LifeLengthI}
\tau_{\dagger,i}(u) := \tau_i(u) + \kappa_i(u).
\end{align}
The interpretation is that we have explored the information of events until the moment $\tau_i(u)$, and suppose no more growth event after it. Thus we make a sum of $\tau_i(u)$ and $\kappa_i(u)$. Heuristically, the more growth events we explore, the closer to $\tau_{\dagger}(u)$ this quantity is. We justify this observation.

\begin{lemma}\label{lem:LifeLengthI}
Let $n$ be the cluster size of $u$ at birth and $k$ be the number of growth events after $\tau_0(u)$ in Algorithm~\ref{algo:Backward}, then we have an expression that for every $0 \leq i \leq k$
\begin{align}\label{eq:LifeLengthIFormula}
\tau_{\dagger, i}(u) = \widetilde{W}(u)\Ind{n=1} + \frac{n-1+\Ind{n=1}}{n+i-1+\Ind{n=1}} W(u) + \sum_{j=1}^i \frac{i+1-j}{n+i-1+\Ind{n=1}} Z_j(u).
\end{align}
Moreover, the sequence $(\tau_{\dagger, i}(u))_{0 \leq i \leq k}$ is decreasing.
\end{lemma}
\begin{proof}
The indicator $\Ind{n=1}$ here is used to distinguish the special case $n=1$ in Algorithm~\ref{algo:Backward}. We prove \eqref{eq:LifeLengthIFormula} for the case $n \geq 2$ at first by recurrence. The case $i = 0$ is trivial by the expression that
\begin{align*}
\tau_{\dagger,0}(u) = \tau_0(u) + \kappa_0(u) = 0 + W(u) =  W(u).    
\end{align*}
Suppose that we have proved the case for $(i-1)$, then for the case $i$ we apply the formula \eqref{eq:LifetimeRefresh1} and \eqref{eq:LifetimeRefresh2} 
\begin{equation}\label{eq:LifeLengthIteration}
\begin{split}
\tau_{\dagger,i}(u) &= \tau_i(u) + \kappa_i(u)\\
&= \tau_{i-1}(u) + Z_i(u) + \Ll(\frac{n+i-2}{n+i-1}\Rr)\Ll(\kappa_{i-1}(u) - Z_{i}(u)\Rr) \\
&= \frac{1}{n+i-1}Z_i(u) + \tau_{i-1}(u) + \Ll(\frac{n+i-2}{n+i-1}\Rr)\kappa_{i-1}(u).
\end{split}
\end{equation}
We then apply the recurrence of \eqref{eq:LifeLengthIFormula} at $(i-1)$ and obtain 
\begin{align*}
\tau_{\dagger,i}(u) &= \frac{1}{n+i-1}(Z_i(u) + \tau_{i-1}(u)) + \Ll(\frac{n+i-2}{n+i-1}\Rr)(\kappa_{i-1}(u) + \tau_{i-1}(u))\\
&= \frac{1}{n+i-1} \sum_{j=1}^i Z_j(u) + \Ll(\frac{n+i-2}{n+i-1}\Rr)\Ll(\frac{n-1}{n+i-2} W(u) + \sum_{j=1}^{i-1} \frac{i-j}{n+i-2} Z_j(u)\Rr)\\
&= \frac{n-1}{n+i-1} W(u) + \sum_{j=1}^i \frac{i+1-j}{n+i-1} Z_j(u).
\end{align*}
This is the desired result of \eqref{eq:LifeLengthIFormula}. For the monotonicity, we apply the recurrence \eqref{eq:LifeLengthIteration}.
\begin{align*}
&\tau_{\dagger,i}(u) - \tau_{\dagger,i-1}(u) \\
&= \frac{1}{n+i-1}Z_i(u) + \tau_{i-1}(u) + \Ll(\frac{n+i-2}{n+i-1}\Rr)\kappa_{i-1}(u) - (\tau_{i-1}(u) + \kappa_{i-1}(u))\\
&= \frac{1}{n+i-1}(Z_i(u) - \kappa_{i-1}(u)).
\end{align*}
By (i.3) of Algorithm~\ref{algo:Backward}, the existence of $i$-th growth implies $Z_i(u) < \kappa_{i-1}(u)$. Therefore, the difference above is negative and we finish the proof.  

Then for the case $n=1$, it suffice to add the first growth event at $\tau_0(u)$, and treat the rest term as the evolution of a cluster of size $2$.
\end{proof}

\begin{proof}[Proof of Proposition~\ref{prop:CouplingBranch}]
We treat the two cases in the step (ii) of Algorithm~\ref{algo:Time}.
\begin{enumerate}
    \item For the case $n'= n$, the coupling in \eqref{eq:Coupling1} and \eqref{eq:Coupling2} gives us that 
    \begin{align}\label{eq:WZDomination}
    W'(u) = W(u), \qquad \widetilde{W}'(u) \leq \widetilde{W}(u), \qquad Z'_i(u) \leq Z(u), \forall i \in \N_+,
    \end{align}
    thanks to the fact $\beta' \geq \beta$ and they use the same random variables in $\widehat{\Lambda}(u)$. This implies in Algorithm~\ref{algo:Backward} that the waiting time for each growth event is always shorter in $(\overline{X}'_t)_{t \geq 0}$, while the pre-waiting time $\kappa_0(u), \kappa'_0(u)$ for the fragmentation/isolation event are same. By an induction in \eqref{eq:LifetimeRefresh2}, we also establish $\kappa'_{i}(u) \geq \kappa_{i}(u)$, which ensures that we can explore more growth events in $(\overline{X}'_t)_{t \geq 0}$. This finishes the statement about the size in Proposition~\ref{prop:CouplingBranch} for the case $n'=n$.
    
    Concerning the length of lifetime, we use the formula  \eqref{eq:LifeLengthIFormula} of the $i$-th predicted length of lifetime $\tau_{\dagger, i}(u)$ that 
    \begin{align*}
        \tau_{\dagger, i}(u) = \widetilde{W}(u)\Ind{n=1} + \frac{n-1+\Ind{n=1}}{n+i-1+\Ind{n=1}} W(u) + \sum_{j=1}^i \frac{i+1-j}{n+i-1+\Ind{n=1}} Z_j(u).
    \end{align*}
    Compare it with its version in $(\overline{X}'_t)_{t \geq 0}$, we have 
    \begin{multline*}
        \widetilde{W}(u)\Ind{n=1} + \frac{n-1+\Ind{n=1}}{n+i-1+\Ind{n=1}} W(u) + \sum_{j=1}^i \frac{i+1-j}{n+i-1+\Ind{n=1}} Z_j(u) \\ 
        \geq \widetilde{W}'(u)\Ind{n=1} + \frac{n-1+\Ind{n=1}}{n+i-1+\Ind{n=1}} W'(u) + \sum_{j=1}^i \frac{i+1-j}{n+i-1+\Ind{n=1}} Z'_j(u),
    \end{multline*}
    because of \eqref{eq:WZDomination}. This implies that $\tau_{\dagger, i}(u) \geq \tau'_{\dagger, i}(u)$. As we know, there exists a $k \in \N_+$ such that $\tau_{\dagger}(u) = \tau_{\dagger, k}(u)$, so we have 
    \begin{align*}
        \tau_{\dagger}(u) = \tau_{\dagger, k}(u) \geq \tau'_{\dagger, k}(u) \geq \tau'_{\dagger}(u).
    \end{align*} 
    Here the last inequality comes from the decreasing property in Lemma~\ref{lem:LifeLengthI} and the fact there are more growth events in $(\overline{X}'_t)_{t \geq 0}$.  This finishes the statement about the length of lifetime in Proposition~\ref{prop:CouplingBranch} for the case $n'=n$.
    
    \item For the case $n' > n$, there are also two cases.
    \begin{enumerate}
        \item There are less than $(n'-n)$ growth events for $u$ in $(\overline{X}_t)_{t \geq 0}$. For this case, the statement about the size is obvious as the cluster size is less than $n'$ in $(\overline{X}_t)_{t \geq 0}$ of $u$ at the end of lifetime. We consider the length of lifetime in $(\overline{X}_t)_{t \geq 0}$ that
        \begin{align*}
            \tau_{\dagger,i}(u) &= \widetilde{W}(u)\Ind{n=1} + \frac{n-1+\Ind{n=1}}{n+i-1+\Ind{n=1}} W(u) + \sum_{j=1}^i \frac{i+1-j}{n+i-1+\Ind{n=1}} Z_j(u)\\
            &\geq \frac{n-1+\Ind{n=1}}{n+i-1+\Ind{n=1}} W(u)\\
            &=\frac{\widehat{\eta}(u)}{(\gamma + \theta) (n+i-1+\Ind{n=1})}.
        \end{align*}
        In the last line, we make use of \eqref{eq:Coupling1}. Notice that $n+i-1+\Ind{n=1}$ is the cluster size of $u$ after several growth events in its lifetime, and it should be less than $n'$ by assumption. Thus, we obtain 
        \begin{align*}
            \tau_{\dagger,i}(u) \geq \frac{\widehat{\eta}(u)}{(\gamma + \theta)n'} = W'(u) = \tau'_{\dagger,0}(u) \geq  \tau'_{\dagger}(u),
        \end{align*}
        where the last inequality comes from the decreasing property in Lemma~\ref{lem:LifeLengthI}. We choose $i$ as the index of the last growth event of $u$ in $(\overline{X}_t)_{t \geq 0}$, and this concludes the statement about the length of lifetime. 
        
        \item There are at least $(n'-n)$ growth events for $u$ in $(\overline{X}_t)_{t \geq 0}$. From (ii).b of Algorithm~\ref{algo:Time}, after $(n'-n)$ growth events, $u$ has the same cluster size in two processes and the growth events are totally coupled. Equation \eqref{eq:Coupling3} gives us 
        \begin{align}\label{eq:ZZDomination}
              Z'_i(u) \leq Z_{i + n'-n-\Ind{n=1}}(u),  \qquad \forall i \in \N_+,
        \end{align}
        thus the domination of waiting time for the growth events is also established. For the pre-waiting time of the fragmentation/isolation event, \eqref{eq:LifetimeRefresh2} gives us 
        \begin{align*}
            \kappa_{i+1}(u) = \Ll(\frac{n+i-1+\Ind{n=1}}{n+i+\Ind{n=1}}\Rr)\Ll(\kappa_i(u) - Z_{i+1}(u)\Rr) \leq \Ll(\frac{n+i-1+\Ind{n=1}}{n+i+\Ind{n=1}}\Rr) \kappa_i(u).
        \end{align*}
        So a telescope formula entails
        \begin{equation}\label{eq:DominationWT}
        \begin{split}
            \kappa_{n'-n-\Ind{n=1}}(u) &\leq \frac{n-1+\Ind{n=1}}{n + (n'-n-\Ind{n=1}) - 1 +\Ind{n=1}}\kappa_{0}(u)\\
            &= \frac{n-1+\Ind{n=1}}{n' - 1} \frac{\widehat{\eta}(u)}{(\gamma + \theta) (n-1 + \Ind{n=1})}\\
            &= \kappa'_0(u).
        \end{split}
        \end{equation}
        Here we use the the fact that $\widetilde{W}'(u), \widetilde{W}(u)$ are sampled by the same random variable $\widehat{\eta}(u)$ in \eqref{eq:Coupling1}. Therefore, \eqref{eq:DominationWT} also gives us a domination of the pre-waiting time that ${\kappa_{n'-n-\Ind{n=1}}(u) \leq \kappa'_0(u)}$. Combining this with \eqref{eq:ZZDomination}, we can make use of a similar argument like the case $n'=n$ for the part of $u$ in $(\overline{X}_t)_{t \geq 0}$ after $(n'-n)$ growth events, together with $u$ in $(\overline{X}'_t)_{t \geq 0}$ from its birth. This concludes our proof.
    \end{enumerate}
\end{enumerate}
\end{proof}

\begin{proof}[Proof of Lemma~\ref{lem:DomSto}] A combination the coupling in Algorithm~\ref{algo:size} and Algorithm~\ref{algo:Time} gives us desired result; see Proposition~\ref{prop:fragDomination} and Proposition~\ref{prop:CouplingBranch} for proof.
\end{proof}

\subsection{Complete proof of Proposition~\ref{prop:betaMono}}\label{sec:CouplingProof}

We can now apply the coupling in previous sections to the whole process $(\overline{X}_t)_{0 \leq t \leq T}$ and $(\overline{X}'_t)_{0 \leq t \leq T}$, which gives us desired property in the previous discussion in Section~ \ref{subsec:Regularity2}.

\begin{algorithm}[Coupling for modified GFI processes]\label{algo:Coupling}
We define a set of random variables 
\begin{align}\label{eq:Gamma}
\Gamma(u) := \{\widehat{U}(u), \widehat{V}(u)\} \cup \widehat{\Lambda}(u) \cup \Lambda(u),
\end{align}
where $\widehat{U}(u), \widehat{V}(u)$ are defined in Algorithm~\ref{algo:size} and $\widehat{\Lambda}(u), \Lambda(u)$ are defined in Algorithm~\ref{algo:Time}. We construct our probability space $\Ll(\Omega, \mcl G, \P^{\beta, \beta'}_{n_0, n'_0}\Rr)$ and processes $(\overline{X}_t, \overline{X}'_t)_{0 \leq t \leq T}$ in it.
\begin{enumerate}[label = (\roman*)]
    \item Let $\Omega := (\Gamma(u))_{u \in \mathcal{U}}$ and define $\mcl G:= \sigma\Ll(\cup_{u \in \mcl U}\Gamma(u)\Rr)$, then sample i.i.d. random variable sets $(\Gamma(u))_{u \in \mcl U}$ which gives the probability $\P^{\beta, \beta'}_{n_0, n'_0}$. 
    \item Starting from the cluster of label $\emptyset$, follow the breadth-first search order, i.e. by the order $\emptyset, 1, 2, 11, 12, 21,22, 111 \cdots$ (see \cite[Chapter 22.2]{algorithm} for reference) and apply Algorithm~\ref{algo:Time} and Algorithm~\ref{algo:size} with respect to its birth time to each branch.
\end{enumerate}
\end{algorithm}

Recall the Ulam-Harris-Neveu notation defined in Section~\ref{UHN}, where we use $\mcl U_t$ for the labels of the alive clusters at moment $t$, $\mcl U_t(u)$ for the one issued from $u$ and $\overline{X}^u_t$ for the cluster size of label $u$. We keep the original version for the process $(\overline{X}_t)_{0 \leq t \leq T}$ and $\mcl U_t, \mcl U'_t(u), \overline{X}'^u_t$ for its version in $(\overline{X}'_t)_{0 \leq t \leq T}$. The following proposition captures the intuition that ``$(\overline{X}'_t)_{0 \leq t \leq T}$ evolves quicker than $(\overline{X}_t)_{0 \leq t \leq T}$, and for every cluster  the processes are no longer coupled from a certain random time''.

\begin{proposition}\label{prop:MonoCoupling}
The probability space $\Ll(\Omega, \mcl G, \P^{\beta, \beta'}_{n_0, n'_0}\Rr)$ defined in Algorithm~\ref{algo:Coupling} satisfies the following properties.
\begin{itemize}
\item $(\overline{X}_t)_{0 \leq t \leq T}$ and $(\overline{X}'_t)_{0 \leq t \leq T}$ live in it as a modified GFI process with respect to their natural filtration.
\item Under the filtration $(\mcl G_t)_{0 \leq t \leq T}$ defined as  
\begin{align}\label{eq:Gfiltration}
\mcl G_t := \sigma \Ll((\overline{X}'^u_s)_{0 \leq s \leq t}, (\overline{X}^u_s)_{0 \leq s \leq T}, u \in \mcl U\Rr),
\end{align}
for every $u \in \mcl U_T$, there exists a stopping time $\tau_T(u) \in [0,T]$ with respect to $(\mcl G_t)_{0 \leq t \leq T}$ such that $\overline{X}'^u_{\tau_T(u)} > 0$ and $\Ll(\sum_{v \in \mcl U'_t(u)}\delta_{\overline{X}'^v_t}\Rr)_{\tau_T(u) \leq t \leq T}$ is Markov process with respect to the filtration $(\mcl G_t)_{0 \leq t \leq T}$.
\end{itemize}
\end{proposition}

\begin{proof}[Proof of Proposition~\ref{prop:MonoCoupling}] 
The first statement is the result of the recurrence of Algorithm~\ref{algo:size} and \ref{algo:Time}. In this recurrence, Proposition~\ref{prop:CouplingBranch} ensures the monotonicity of the length of lifetime and cluster size when fragmentation, while Proposition~\ref{prop:fragDomination} pass this monotonicity to the child clusters. Therefore, $(\overline{X}'_t)_{0 \leq t \leq T}$ evolves quicker than $(\overline{X}_t)_{0 \leq t \leq T}$. 

Concerning the second statement, when a cluster of label $u$ stops its evolution at $T$ in $(\overline{X}_t)_{0 \leq t \leq T}$, its trajectory translates to the information about $\Gamma(v)_{v \in \mcl U}$, which then translates to its coupled one in $(\overline{X}'_t)_{0 \leq t \leq T}$: it is still alive at a moment $\tau_{T}(u)$. This dependence implies that $\tau_{T}(u)$ is a stopping time with respect to $(\mcl G_t)_{0 \leq t \leq T}$. Naturally, for the rest of process $(\overline{X}'^u)_{0 \leq t \leq T}$ after $\tau_T(u)$, because $(\overline{X}_t)_{0 \leq t \leq T}$ does not provide more information and then we obtain the Markov property. Thus we finish the proof of the second statement.
\end{proof}

Proposition~\ref{prop:MonoCoupling} gives a rigorous proof of Proposition~\ref{prop:betaMono}.
\begin{proof}[Proof of Proposition~\ref{prop:betaMono}]
We use $\E^{\beta}_{\delta_{n_0}}, \E^{\beta'}_{\delta_{n'_0}}$ respectively for the expectation on the probability space of $(\overline{X}_t)_{t \geq 0}$ and $(\overline{X}'_t)_{t \geq 0}$. Our main object is to compare $\E^{\beta}_{\delta_{n_0}}\Ll[\bracket{\overline{X}_T, 1}\Rr]$ and $\E^{\beta'}_{\delta_{n'_0}}\Ll[\bracket{\overline{X}'_T, 1}\Rr]$, which will illustrate Malthusian behavior for large $T$.

We use the common probability space $\P^{\beta, \beta'}_{n_0, n'_0}$ constructed in Algorithm~\ref{algo:Coupling}, and we write  $\E^{\beta, \beta'}_{n_0, n'_0}$ for the associated expectation. By Proposition~\ref{prop:MonoCoupling}, under $\P^{\beta, \beta'}_{n_0, n'_0}$, every cluster in $\mathcal{U}_T$ has its coupled one in $(\overline{X}'_t)_{0 \leq t \leq T}$, which is still alive at the moment $\tau_T(u)$ and produces the clusters in $\mathcal{U}'_T$. Then we have
\begin{align*}
\E^{\beta'}_{\delta_{n'_0}}\Ll[\bracket{\overline{X}'_T , 1}\Rr] = \E^{\beta, \beta'}_{n_0, n'_0}\Ll[\sum_{v \in \mathcal{U}'_T} 1\Rr] = \E^{\beta, \beta'}_{n_0, n'_0}\Ll[\sum_{u \in \mathcal{U}_T } \sum_{v \in \mathcal{U}'_T(u)} 1  \Rr]. 
\end{align*}
Since the modified GFI process is c\`adl\`ag, we apply the strong Markov property that 
\begin{align*}
\E^{\beta'}_{\delta_{n'_0}}\Ll[\bracket{\overline{X}'_T , 1}\Rr] &= \E^{\beta, \beta'}_{n_0, n'_0}\Ll[\sum_{u \in \mathcal{U}_T }  \E^{\beta, \beta'}_{n_0, n'_0}\Ll[\sum_{v \in \mathcal{U}'_T(u)} 1  \Bigg\vert \mathcal{G}_{\tau_T(u)} \Rr]\Rr]\\
&= \E^{\beta, \beta'}_{n_0, n'_0}\Ll[\sum_{u \in \mathcal{U}_T } \E^{\beta'}_{\delta_{\overline{X}'^u_{\tau_T(u)}}}\Ll[\sum_{v \in \mathcal{U}'_T(u)} 1  \Bigg\vert \mathcal{G}_{\tau_T(u)} \Rr]\Rr],
\end{align*}
and the quantity ${\E^{\beta'}_{\delta_{\overline{X}'^u_{\tau_T(u)}}}\Ll[\sum_{v \in \mathcal{U}'_T(u)} 1  \Bigg\vert \mathcal{G}_{\tau_T(u)} \Rr]}$ can be seen as the mean of offspring for a modified GFI process issued from a RRT of size $\overline{X}'^u_{\tau_T(u)}$.
 
Lemma~\ref{lem:modiGFIphase} gives us some indication about the phases of the modified GFI process, but here we need a more precise result. In fact, the mapping $t \mapsto \E^{\beta'}_{\delta_{n'_0}}[\langle\overline{X}'_t, 1 \rangle]$ is increasing (resp. decreasing) when $\gamma > \theta$ (resp. $\gamma < \theta$). The idea is to use the first moment generator \eqref{eq:Generator2} and a constant test function $\mathbf{1}$ that 
\begin{align*}
    \overline{\L} \mathbf{1} (n) = (\gamma - \theta)(n-1).
\end{align*}
Therefore, the sign depends on $(\gamma - \theta)$. For the case $\gamma > \theta$, we obtain 
\begin{align*}
\forall u \in \mathcal{U}_T, \qquad  \E^{\beta'}_{\delta_{\overline{X}'^u_{\tau_T(u)}}}\Ll[\sum_{v \in \mathcal{U}'_T(u)} 1  \Bigg\vert \mathcal{G}_{\tau_T(u)} \Rr] \geq 1,
\end{align*}
and from this we deduce 
\begin{align}\label{eq:Mono1}
\E^{\beta'}_{\delta_{n'_0}}\Ll[\bracket{\overline{X}'_T , 1}\Rr] \geq \E^{\beta, \beta'}_{n_0, n'_0}\Ll[\sum_{u \in \mathcal{U}_T } 1\Rr] = \E^{\beta}_{\delta_{n_0}}\Ll[\bracket{\overline{X}_T , 1}\Rr].
\end{align}
Because of the Malthusian behavior \eqref{eq:Spectral}, for a large $T$ we have
\begin{align*}
\E^{\beta'}_{\delta_{n'_0}}\Ll[\bracket{\overline{X}'_T , 1}\Rr] \sim h'(n'_0)e^{\overline{\lambda}'T},  \qquad  \E^{\beta}_{\delta_{n_0}}\Ll[\bracket{\overline{X}_T , 1}\Rr] \sim h(n_0)e^{\overline{\lambda}T},
\end{align*}
where $h',h$ are the associated eigenvectors. The monotonicity \eqref{eq:Mono1} thus implies for the case $\gamma > \theta$: 
\begin{itemize}
    \item when $\beta' > \beta$, we have $\overline{\lambda}' \geq \overline{\lambda}$ ;
    \item when $\beta' = \beta, n_0' > n_0$, we have $h(n'_0) \geq h(n_0)$.
\end{itemize}
The case $\gamma < \theta$ is similar. Finally, recall ${\lambda(\beta, \theta, \gamma) = \overline{\lambda}(\beta, \theta, \gamma) + \theta}$ as discussed in Section~\ref{subsubsec:mGFI}, we thus obtain the desired result. 
\end{proof}

\end{appendix}

\begin{acks}[Acknowledgments]
This work was partially funded by the Chair ``Mod\'elisation Math\'ematique et Biodiversit\'e" of VEOLIA-Ecole Polytechnique-MNHN-F.X and ANR ABIM 16-CE40-0001 and ANR NOLO 20-CE40-0015. L.Y. acknowledges the support of the National Natural Science Foundation of China (Youth Programme, Grant: 11801458).
\end{acks}

\bibliographystyle{imsart-number} 
\bibliography{RRTRef}       

\end{document}